\documentclass[11pt]{article}
\usepackage{amsmath,amssymb,amsthm,amscd}

\newtheorem{theorem}{Theorem}[section]
\newtheorem{lemma}[theorem]{Lemma}
\newtheorem{proposition}[theorem]{Proposition}
\newtheorem{corollary}[theorem]{Corollary}
\newtheorem{example}[theorem]{Example}

\theoremstyle{plain}

\theoremstyle{definition}
\newtheorem{definition}[theorem]{Definition}
\newtheorem{remark}[theorem]{Remark}

\numberwithin{equation}{section}

\newcommand{\Ext}{\operatorname{Ext}}

\newcommand{\id}{\operatorname{id}}
\newcommand{\Ker}{\operatorname{Ker}}

\newcommand{\pr}{\operatorname{pr}}

\newcommand{\N}{\mathbb{N}}
\newcommand{\T}{\mathbb{T}}
\newcommand{\Z}{\mathbb{Z}}
\newcommand{\Zp}{{\mathbb{Z}}_+}

\def\M{\mathcal{M}}
\def\A{\mathcal{A}}
\def\F{\mathcal{F}}
\def\P{\mathcal{P}}
\def\Q{\mathcal{Q}}
\def\X{\mathcal{X}}
\def\Y{\mathcal{Y}}

\def\R{\mathcal{R}}

\def\Ext{{{\operatorname{Ext}}}}

\def\det{{{\operatorname{det}}}}

\def\ev{{{\operatorname{ev}}}}

\def\LGBS{({\frak L}^-, {\frak L}^+)}
\def\LLGBS{({\frak L}^-_\Lambda, {\frak L}^+_\Lambda)}

\title{$C^*$-algebras associated with two-sided subshifts
\\
}
\author{Kengo Matsumoto \\
Department of Mathematics \\
Joetsu University of Education \\
Joetsu, 943-8512, Japan
}

\begin{document}
\maketitle

\date{}

\def\det{{{\operatorname{det}}}}

%\maketitle
\begin{abstract}   
This paper is a continuation of the paper entitled 
``Subshifts, $\lambda$-graph bisystems and $C^*$-algebras'',
arXiv:1904.06464.
A $\lambda$-graph bisystem consists of a pair of two labeled Bratteli diagrams
satisfying certain compatibility condition on their edge labeling.
For any two-sided subshift $\Lambda$,
there exists a $\lambda$-graph bisystem satisfying a special property called FPCC.
We will construct an AF-algebra ${\mathcal{F}}_{\frak L}$ 
with shift automorphism
$\rho_{\frak L}$
from a $\lambda$-graph bisystem 
$({\frak L}^-,{\frak L}^+)$, 
and  define a $C^*$-algebra 
${\mathcal R}_{\frak L}$
by the crossed product
${\mathcal{F}}_{\frak L}\rtimes_{\rho_{\frak L}}\Z$.
It is a two-sided subshift analogue of asymptotic Ruelle algebras constructed from Smale spaces. 
If $\lambda$-graph bisystems come from  two-sided subshifts,
these $C^*$-algebras are proved to be invariant under topological conjugacy of the underlying subshifts.
 We will present a simplicity condition of 
the $C^*$-algebra ${\mathcal R}_{\frak L}$
and the K-theory formulas of the $C^*$-algebras
${\mathcal{F}}_{\frak L}$ and  ${\mathcal R}_{\frak L}$.
The K-group for the AF-algebra
${\mathcal{F}}_{\frak L}$ is regarded as a two-sided extension of the dimension group
of subshifts.
\end{abstract}



\bigskip

\medskip

Contents:

\begin{enumerate}
\renewcommand{\theenumi}{\arabic{enumi}}
\renewcommand{\labelenumi}{\textup{\theenumi}}
\item Introduction
%\item Subshifts, $\lambda$-graph systems and its $C^*$-algebras
%\item $\lambda$-graph bisystems and symbolic matrix bisystems
\item Subshifts and $\lambda$-graph bisystems
%\item Strong shift equivalence and shift equivalence
%\item Nonnegative matrix bisystems
%\item Dimension groups, 
\item Two-sided AF algebras 
\item $C^*$-algebras associated with subshifts
\item Invariance under topological conjugacy
\item Dimension groups and K-theory formulas
\item Topological Markov shifts
\item Even shift
%\item A duality: $\lambda$-graph systems as $\lambda$-graph bisystems
\end{enumerate}

\newpage

%%%%%%%%%%%%%%%%%%%%%%%%%%%%%%%%%%%%%%%%%%%%%%%%%%%%   

%%%%%%%%%%%%%%%%%%%%%%%%%%%%%%%%%%%%%%%
\section{Introduction}
%%%%%%%%%%%%%%%%%%%%%%%%%%%%%%%%

This paper is a continuation of the preprint  \cite{MaPre2019} entitled 
``Subshifts, $\lambda$-graph bisystems and $C^*$-algebras'',
arXiv:1904.06464.

Let $\Sigma$ be a finite set called an alphabet.
Let $\sigma:\Sigma^\Z \longrightarrow \Sigma^\Z$ be a homeomorphism of the shift
$\sigma((x_n)_{n\in \N}) = (x_{n+1})_{n\in \N}$, where  
$\Sigma^\Z$  is endowed with its product topology.
A closed $\sigma$-invariant subset $\Lambda$ of $\Sigma^\Z$ is called a subshift.
It is a topological dynamical system
$(\Lambda,\sigma)$ on the compact Hausdorff space $\Lambda$. 
In \cite{MaDocMath1999}, the author introduced a notion of 
$\lambda$-graph system as a generalization of finite labeled directed graphs.
A $\lambda$-graph system over a finite alphabet $\Sigma$ 
is a labeled Bratteli diagram $(V, E, \lambda, \iota)$ with a labeling map 
$\lambda:E \longrightarrow \Sigma$ of edges and an additional structure
$\iota: V\longrightarrow V$.  
The map $\iota : V = \cup_{l=0}^\infty V_l\longrightarrow \cup_{l=1}^\infty V_l$ 
consists of  a sequence  $\iota: V_{l+1}\longrightarrow V_l$ of sujections between vertices.
In \cite{MaDocMath1999},
it was proved that   
any $\lambda$-graph system presents a subshift, 
and conversely any subshift may be presented by a $\lambda$-graph system
called the canonical $\lambda$-graph system. 
It was proved that two subshifts are topologically conjugate 
if and only if their canonical $\lambda$-graph systems are 
(properly) strong shift equivalent.
  In \cite{MaDocMath2002}, a certain $C^*$-algebra associated with  a $\lambda$-graph system was introduced  as a generalization of Cuntz--Krieger algebras.
Since algebraic structure of the Cuntz--Krieger algebras     
reflect one-sided shift structure of the underlying topological Markov shifts,
the $C^*$-algebras associated with $\lambda$-graph systems  
reflect one-sided shift structure of the underlying subshifts.
 
There is a construction of $C^*$-algebras 
that reflect two-sided shift structure of topological Markov shifts.
It has been presented by D. Ruelle \cite{Ruelle2}, and I. Putnam \cite{Putnam1}.
They constructed 
several classes of $C^*$-algebras from Smale spaces.
Two-sided topological Markov shifts are typical examples of Smale spaces,
and the Smale space $C^*$-algebras for  topological Markov shifts 
are asymptotic Ruelle algebras $\R_A^a$. 
Let $(\Lambda_A, \sigma_A)$ be a topological Markov shift defined by a nonnegative matrix $A$.
The $C^*$-algebra  $\R_A^a$ is realized, roughly speaking, as the crossed product $C^*$-algebra 
$\F_A\rtimes_{\sigma_A}\Z$ of the two-sided AF-algebra defined by the matrix $A$ 
by the shift automorphism $\sigma_A$ (\cite{Holton}, \cite{Putnam1}, \cite{PutSp}, etc. ).   
It is well-known that a subshift can not be a Smale space unless it is a topological Markov shift.
Hence we may not apply the construction of $C^*$-algebras for general subshifts.

In this paper, we will construct a two-sided AF-algebra having shift automorphism 
from a general subshift,
so that we may construct a subshift version of asymptotic Ruelle algebras.
To construct AF-algebras from general subshifts,
we will use a notion of $\lambda$-graph bisystem introduced in a recent paper
\cite{MaPre2019}.
A $\lambda$-graph bisystem $\LGBS$ consists of a pair of two labeled Bratteli diagrams
${\frak L}^-,{\frak L}^+$
satisfying certain compatibility condition on their edge labeling.
Let $\LGBS$ be a $\lambda$-graph bisystem,
that consists of two labeled Bratteli diagrams
${\frak L}^- =(V, E^-, \lambda^-)$ and ${\frak L}^+ =(V, E^+, \lambda^+).$
They have the common vertex set $V = \cup_{l=0}^\infty V_l$.
The former one ${\frak L}^-$ has 
 upward directed edges $E^- = \cup_{l=0}^\infty E^-_{l,l+1}$,
and the latter one ${\frak L}^+$ has downward directed edges 
$E^+ = \cup_{l=0}^\infty E^+_{l,l+1}$.
The labeling maps of edges 
$\lambda^-: E^-\longrightarrow \Sigma^-$ 
and
$\lambda^+: E^+\longrightarrow \Sigma^+$ 
satisfy certain compatibility condition, 
called local property of $\lambda$-graph bisystem.
For any two-sided subshift $\Lambda$,
there exists a $\lambda$-graph bisystem satisfying a special property called 
FPCC (Follower-Predecessor Compatibility Condition).
We will construct an AF-algebra $\F_{\frak L}$ from a $\lambda$-graph bisystem 
$({\frak L}^-,{\frak L}^+)$ satisfying FPCC.
The AF-algebra $\F_{\frak L}$ has a shift automorphism $\rho_{\frak L}$ 
arising from a shift homeomorphism $\sigma$ on $\Lambda$. 
Then we may construct the crossed product $C^*$-algebra $\F_{\frak L}\rtimes_{\rho_{\frak L}}\Z$ 
denoted by $\R_{\frak L}$.
It is a two-sided subshift analogue of asymptotic Ruelle algebras constructed from Smale spaces. 
Under a certain irreducibility condition on $\LGBS$, we have 
\begin{theorem}[{Proposition \ref{prop:simplicity}, Theorem \ref{thm:simplicity}}]
Let $\LGBS$ be a $\lambda$-graph bisystem satisfying FPCC.
%\hspace{6cm}
\begin{enumerate}
\renewcommand{\theenumi}{\roman{enumi}}
\renewcommand{\labelenumi}{\textup{(\theenumi)}}
\item
If  $\LGBS$ is irreducible,
then
the AF-algebra $\F_{\frak L}$ is simple.
\item
If  $\LGBS$ satisfies condition (I) and is irreducible,
then
the $C^*$-algebra $\R_{\frak L}$ is simple.
\end{enumerate}
\end{theorem}
For a two-sided subshift $\Lambda$,
there exists a canonical method to construct a $\lambda$-graph bisystem
satisfying FPCC.
The $\lambda$-graph bisystem is called   
the canonical $\lambda$-graph bisystem for a subshift $\Lambda$
and written $({\frak L}_\Lambda^-,{\frak L}_\Lambda^+)$.
The $C^*$-algebras $\F_{{\frak L}_\Lambda},\R_{{\frak L}_\Lambda}$
and the automorphism $\rho_{{\frak L}_\Lambda}$
for the canonical $\lambda$-graph bisystem
$({\frak L}_\Lambda^-,{\frak L}_\Lambda^+)$
for a subshift $\Lambda$
are written $\F_\Lambda, \R_\Lambda$ and $\rho_\Lambda$, respectively.
We may prove the following theorem.
\begin{theorem}[{Corollary \ref{cor:conjugacy}}] 
Suppose that two subshifts 
$\Lambda_1, \Lambda_2$ are topologically conjugate.
Then there exists an isomorphism 
$\Phi: \F_{\Lambda_1}\longrightarrow \F_{\Lambda_2}$
of $C^*$-algebras such that 
$\Phi\circ\rho_{\Lambda_1} =\rho_{\Lambda_2}\circ\Phi$.
Hence it induces an isomorphism 
$\widehat{\Phi}: \R_{\Lambda_1}\longrightarrow \R_{\Lambda_2}$
between their  crossed products
$\F_{\Lambda_1}\rtimes_{\rho_{\Lambda_1}}\Z$
and
$\F_{\Lambda_2}\rtimes_{\rho_{\Lambda_2}}\Z$.
\end{theorem}
The above theorem shows that  the triple of its K-theory group  
\begin{equation*}
(K_0(\F_{\Lambda}),K_0(\F_{\Lambda})_+, \rho_{{\Lambda}*})  
 \end{equation*}
is invariant under topological conjugacy of subshift
$\Lambda$.
If $\Lambda$ is a topological Markov shift
$\Lambda_A$ defined by a nonnegative matrix $A$, 
then the $C^*$-algebra $\F_{\Lambda_A}$ is isomorphic to the AF-algebra defined by two-sided asymptotic equivalence relation on $\Lambda_A$
seen in \cite{Ruelle2} and \cite{Putnam1} (cf.  \cite{Holton}, \cite{PutSp}, etc.). 
The algebra and its K-group are deeply studied by
Killough--Putnam in \cite{KilPut} (cf. \cite{Holton}).

Two subshifts $\Lambda_1, \Lambda_2$
are said to be flip conjugate if 
$\Lambda_1$ is topologically conjugate to $\Lambda_2$ or its transpose
${}^t\!\Lambda_2$.
As a corollary, we have  
\begin{corollary}[{Corollary \ref{cor:flip}}]
Suppose that two subshifts 
$\Lambda_1, \Lambda_2$ are flip conjugate.
Then there exists an isomorphism 
$\Psi:  \R_{\Lambda_1}\longrightarrow \R_{\Lambda_2}$
of $C^*$-algebras.
\end{corollary}
Hence the $C^*$-algebra 
$\R_\Lambda$ and its K-theory groups $K_*(\R_\Lambda)$
are invariant under flip conjugacy of subshifts.

%%%%%%%%%%%%%%%%%%%%%%%%%%%%%

We will compute the K-theoretic triple 
$(K_0(\F_{\frak L}),K_0(\F_{\frak L})_+, \sigma_{{\frak L}*})  
$
and the K-groups 
$K_i(\R_{\frak L}), i=0,1$
in terms of the $\lambda$-graph bisystem $\LGBS$ satisfying FPCC.
Let us denote by $\{v_1^l,\dots, v_{m(l)}^l \}$
the vertex set $V_l$ for $l \in \Zp =\{0,1,2,\dots \}$.
For two vertices $v_i^l\in V_l$ and $v_j^{l+1} \in V_{l+1}$, 
let $M^-_{l,l+1}(i,j)$ (resp. $M^+_{l,l+1}(i,j)$)
be the nonnegative integer of the cardinal number of edges in 
${\frak L}^-$ (resp. ${\frak L}^+$)
starting at $v_j^{l+1}$ (resp. $v_i^l$)
and ending with $v_i^l$ (resp. $v_j^{l+1}$).
We then have   
$m(l) \times m(l+1)$ nonnegative matrices 
$M^-_{l,l+1}$ and $ M^+_{l,l+1}$.
By the local property of $\lambda$-graph bisystem, the commutation relations
\begin{equation}
M^-_{l,l+1}M^+_{l+1,l+2} =M^+_{l,l+1}M^-_{l+1,l+2}, \qquad l \in \Zp \label{eq:MMnnmbs}
\end{equation}
hold.
The sequence of pairs $(M_{l,l+1}^-,M_{l,l+1}^+)_{l\in \Zp}$
of nonnegative matrices satisfying \eqref{eq:MMnnmbs}
is called a nonnegative matrix bisystem.
The transpose ${}^t\!M^{-}_{l,l+1}$ of the matrix $M_{l,l+1}^-$
naturally induces an order preserving homomorphism from
${\Z}^{m(l)}$ to
${\Z}^{m(l+1)}$,
where 
${\Z}^{m(l)}$ and
${\Z}^{m(l+1)}$ have their natural positive cones  
${\Z}^{m(l)}_+$ and
${\Z}^{m(l+1)}_+$, respectively.
%%%%%%%%%%%%%%%%%%%%%%%%%%%%%%%%%%%%%%%%%%%%%%%%%%%%%
%${\Z}^{m(l)}_{+}$ of the group ${\Z}^{m(l)}$ is defined by
%$$ {\Z}^{m(l)}_+ = \{  (n_1,n_2,\dots,n_{m(l)}) \in {\Z}^{m(l)} |
%n_i \in \Zp, i=1,2,\dots,m(l) \}. $$
%We put the inductive limits:@@@
%%%%%%%%%%%%%%%%%%%%%%%%%%%%%%%%%%%%%%
Let us denote by
${\Z}_{M^-}$ the inductive limit 
$\varinjlim \{ {}^t\!M^{-}_{l,l+1}: {\Z}^{m(l)} \longrightarrow {\Z}^{m(l+1)} \}
$
of the abelian group.
Its positive cone
$\varinjlim \{ {}^t\!M^{-}_{l,l+1}: {\Z}^{m(l)}_+ \longrightarrow {\Z}^{m(l+1)}_+ \}
$ is denoted by
${\Z}_{M^{-}}^{+}$.
The sequence of the transposed matrices  ${}^t\!M^{+}_{l,l+1}$
of $M^+_{l,l+1}$ 
 naturally induces an order preserving endomorphism denoted by
$\lambda_{M^+}$ on the ordered group 
${\Z}_{M^-}$. 
We set 
$
{\Z}_{M^-}(k) = {\Z}_{M^-}
$
and
$
{\Z}^+_{M^-}(k) = {\Z}^+_{M^-}
$
for
$ k \in \N$, 
and define an abelian group and its positive cone by the inductive 
limits:
\begin{align}
\Delta_{(M^-,M^+)} & = 
\underset{k}{\varinjlim} \{ \lambda_{M^+}:{\Z}_{M^-}(k) 
                    \longrightarrow {\Z}_{M^-}(k+1) \}, \label{eq:dim}\\
\Delta^+_{(M^-,M^+)} & = 
\underset{k}{\varinjlim}\{ \lambda_{M^+}:{\Z}^+_{M^-}(k) 
                    \longrightarrow {\Z}^+_{M^-}(k+1) \}.\label{eq:dimpo}
\end{align} 
We call the ordered group
$(\Delta_{(M^-,M^+)},\Delta^+_{(M^-,M^+)})$
the {\it dimension group  for}\/ $(M^-,M^+)$.
%Let us denote by $[X,k]$ for $X \in {\Z}_{M^-}$ and
%$k \in \N$ the element $[X,k] $ in $\Delta_{(M^-,M^+)}.$Since 
The map
$
\delta_{(M^-, M^+)} :{\Z}_{M^-}(k) 
                    \longrightarrow {\Z}_{M^-}(k+1) 
$
defined by
$\delta_{(M^-, M^+)}([X,k]) = ([X,k+1])$
for
$X \in {\Z}_{M^-}$ 
yields an automorphism on 
$(\Delta_{(M^-,M^+)},\Delta_{(M^-,M^+)}^+)$
called the {\it dimension automorphism}.
We call the triple
$(\Delta_{(M^-,M^+)},\Delta^+_{(M^-,M^+)},\delta_{(M^-,M^+)})$
the {\it dimension triple for}\/ $(M^-,M^+)$.
Let
$\Lambda$ be a subshift and
$(M^-_\Lambda,M^+_\Lambda)$ its associated nonnegative matrix bisystem 
for the canonical $\lambda$-graph bisystem for the subshift $\Lambda$.
Then the {\it bidimension triple}\/
$(\Delta_{\Lambda},\Delta^+_{\Lambda},\delta_{\Lambda})$
 for $\Lambda$
is defined to be the dimension triple
$(\Delta_{(M^-_\Lambda,M^+_\Lambda)},
\Delta^+_{(M^-_\Lambda, M^+_\Lambda)},
\delta_{(M^-_\Lambda, M^+_\Lambda)})$.

\begin{theorem}[{Theorem \ref{thm:DDelta}}]
Let ${\frak L}$ be a $\lambda$-graph bisystem satisfying FPCC and $(M^-, M^+)$ 
its nonnegative matrix bisystem.
We then have 
\begin{equation*}
(K_0(\F_{\frak L}),K_0(\F_{\frak L})_+, \rho_{{\frak L}*})  
 \cong
(\Delta_{(M^-, M^+)}, \Delta_{(M^-, M^+)}^+, \delta_{(M^-, M^+)}).
\end{equation*}
\end{theorem}
For a two-sided subshift $\Lambda$, we in particular  have 
the following formula
\begin{equation*}
(K_0(\F_{\Lambda}),K_0(\F_{\Lambda})_+, \rho_{{\Lambda}*})  
\cong
(\Delta_{\Lambda},\Delta^+_{\Lambda},\delta_{\Lambda}).
 \end{equation*}
We will compute the K-groups $K_i(\R_{\frak L}), i=0,1$
of the $C^*$-algebra $\R_{\frak L}$ in the following way.
Let
$(M^-,M^+)$ be a nonnegative matrix bisystem.
For $l \in \Zp$, 
we set the abelian groups
\begin{align*}
K_0^l(M^-,M^+) & = 
{\Z}^{m(l+1)} / ({}^t\!M^{-}_{l,l+1} - {}^t\!M^{+}_{l,l+1}){\Z}^{m(l)},\\
K_1^l(M^-,M^+) & =
 \text{Ker}({}^t\!M^{-}_{l,l+1} - {}^t\!M^{+}_{l,l+1}) 
 \text{ in } {\Z}^{m(l)}. 
\end{align*}
By the commutation relation \eqref{eq:MMnnmbs},
the matrix ${}^t\!M^-_{l,l+1}$ induces homomorphisms
\begin{align*}
 {}^t\!M_0^{-l}:&
 K_0^l(M^-,M^+) \longrightarrow K_0^{l+1}(M^-,M^+),\\
 {}^t\!M_1^{-l}:&
 K_1^l(M^-,M^+) \longrightarrow K_1^{l+1}(M^-,M^+).
\end{align*}
We then have the following K-theory formulas for the $C^*$-algebra 
$\R_{\frak L}$.
\begin{theorem}[{Theorem \ref{thm:RLKgroup}}]
\begin{align*}
K_0(\R_{\frak L}) 
 & = \underset{l}{\varinjlim} \{ {}^t\!M_0^{-l}:
 K_0^l(M^-,M^+) \longrightarrow K_0^{l+1}(M^-,M^+) \},\\
K_1(\R_{\frak L}) & = \underset{l}{\varinjlim} \{ {}^t\!M_1^{-l}:
 K_1^l(M^-,M^+) \longrightarrow K_1^{l+1}(M^-,M^+) \}.
\end{align*}
\end{theorem}
Some examples of the above dimension group (Proposition \ref{prop:7.4})
and K-groups (Proposition \ref{prop:7.5} and Proposition \ref{prop:7.6})
for topological Markov shifts are presented in Section 7.
The K-groups for the even shift that is not any topological Markov shifts are computed in Section 8
(Proposition \ref{prop:evenmainc} and Proposition \ref{prop:evenmainlambda}). 

\bigskip
Throughout the paper,
the notation $\N, \Zp$ will denote the set of positive integers, the set of nonnegative integers,
respectively.
By a nonnegative matrix we mean a finite rectangular matrix with entries in nonnegative integers.
 For a finite set $\Sigma$, the notation $|\Sigma|$ denotes its cardinality. 

%\newpage

%%%%%%%%%%%%%%%%%%%%%%%%%%%%%%%%%%%%%%%%%%
%%%%%%%%%%%%%%%%%%%%%%%%%%%%%%%%%%%%%%%
\section{Subshifts and  $\lambda$-graph bisystems}
%%%%%%%%%%%%%%%%%%%%%%%%%%%%%%%%%%%%
%%%%%%%%%%%%%%%%%%%%%%%%%%%%%%%
Let $\Sigma$ be a finite set, which we call an alphabet. 
We call each element of $\Sigma$ a symbol or a label.
Let us denote by $\Sigma^{\Z}$ the set of bi-infinite sequences
$(x_n)_{n \in \Z}$ of  $\Sigma.$
The set
$\Sigma^{\Z}$ 
is a compact Hausdorff space
by the infinite  product topology.
The homeomorphism of the shift
$\sigma: \Sigma^{\Z} \longrightarrow \Sigma^{\Z}$
is defined by
$\sigma((x_n)_{n\in \Z}) =(x_{n+1})_{n\in \Z}.$
A $\sigma$-invariant closed subset
$\Lambda \subset \Sigma^\Z$, that is $\sigma(\Lambda) =\Lambda,$
is called a subshift, that is a topological dynamical system
$(\Lambda, \sigma)$.
The space $\Lambda$ is called the shift space for $(\Lambda, \sigma).$ 
We often write the subshift $(\Lambda, \sigma)$ as $\Lambda$ 
for brevity. 
Let us denote by $B_n(\Lambda)$
the set of admissible words in $\Lambda$ with length $n$, 
that is 
$B_n(\Lambda) = \{(x_1,\dots,x_n) \in \Sigma^n \mid (x_i)_{i\in \Z}\in \Lambda \}.$
The subshift
$(\Sigma^\Z, \sigma)$  is called the full $|\Sigma|$-shift.
For an $N\times N$ matrix $A = [A(i,j)]_{i,j=1}^N$
with its entries $A(i,j)$ in $\{0,1\},$
the topological Markov shift $\Lambda_A$ 
is defined by
\begin{equation}
\Lambda_A = \{ (x_n)_{n \in \Z} \in \{1,2,\dots, N\}^\Z \mid
A(x_n, x_{n+1}) =1 \text{ for all } n \in \Z \}. \label{eq:LambdaA}
\end{equation}    
It is often called  a shift of finite type or simply SFT.
For a finite labeled directed graph
$\mathcal{G} = (\mathcal{V}, \mathcal{E},\lambda)$
with
vertex set $\mathcal{V},$
edge set $\mathcal{E}$
and labeling map 
$\lambda: \mathcal{E}\longrightarrow \Sigma$,
one may define a subshift $\Lambda_{\mathcal{G}}$
 consisting of bi-infinite label sequences of concatenating paths  
in the labeled graph ${\mathcal{G}}.$
It is called a sofic shift (\cite{Fis}, \cite{Kr84}, \cite{Kr87}, \cite{We}).
There are lots of subshifts that are not sofic shifts (see \cite{LM}, etc.).

A $\lambda$-graph system
${\frak L} =(V,E,\lambda,\iota)$ over $\Sigma$ 
 is a graphical object to present a general subshift
(\cite{MaDocMath1999}).
It consists of 
a vertex set
$
V = \cup_{l \in \Zp} V_{l}
$
and  edge set
$
E = \cup_{l \in \Zp} E_{l,l+1}
$
that is labeled with symbols in $\Sigma$ by $\lambda: E \rightarrow \Sigma$, 
and that is supplied with a surjective map
$
\iota( = \iota_{l,l+1}):V_{l+1} \rightarrow V_l
$
for each
$
l \in  \Zp.
$
The surjective map
$\iota: V\longrightarrow V$ satisfies a certain compatibility condition with labeling 
on edges, called local property of $\lambda$-graph system.
See \cite{MaDocMath1999} for a general theory of $\lambda$-graph systems.  
We emphasize that the $\lambda$-graph systems reflect right-one-sided structure of subshifts.

In \cite{MaPre2019}, the author generalized the notion of $\lambda$-graph system 
and introduce a notion of $\lambda$-graph bisystem, that is a two-sided extension of
$\lambda$-graph system.
It is defined in the following way.
For a directed edge $e$, denote by $s(e)$ and $t(e)$ its 
source vertex and terminal vertex, respectively.
Let $\Sigma^-$ and $\Sigma^+$ be two finite alphabets.

\begin{definition}[{\cite[Definition 3.1]{MaPre2019}}] \label{def:lambdabisystem}
A $\lambda$-graph bisystem $({\frak L}^-, {\frak L}^+)$ is a pair of labeled Bratteli diagrams 
${\frak L}^- =(V, E^-, \lambda^-)$ over $\Sigma^-$ and 
$ {\frak L}^+=(V, E^+, \lambda^+)$ over  $\Sigma^+$
satisfying the following five conditions:
\begin{enumerate}
\renewcommand{\theenumi}{\roman{enumi}}
\renewcommand{\labelenumi}{\textup{(\theenumi)}}
\item
${\frak L}^-$
and
${\frak L}^+$
have its common vertex set $V = \cup_{l\in \Zp} V_l$,
that is a disjoint union of finite sets $V_l, l\in \Zp$
with $m(l) \le m(l+1)$ for $l\in \Zp,$
where
$m(l) := | V_l | < \infty. 
$
\item
The edge sets $E^-$ and $E^+$ are 
disjoint unions of finite sets 
$
E^- = \cup_{l\in \Zp} E_{l, l+1}^-
$
and
$
E^+ = \cup_{l\in \Zp} E_{l, l+1}^+,
$
respectively. 
\item
(1) Every edge $e^- \in E_{l,l+1}^-$ satisfies $s(e^-) \in V_{l+1}, \,\, t(e^-) \in V_l,$ 
and for every vertex $v \in V_l$ with $l\ne 0,$ 
there exist $e^- \in E_{l,l+1}^-, f^- \in E_{l-1, l}^-$ such that 
$v = s(f^-) = t(e^-)$,
 and for every vertex $v \in V_0,$ 
there exists $e^- \in E_{0,1}^-$ such that 
$v = t(e^-)$,
 
(2) 
Every edge $e^+ \in E_{l,l+1}^+$ satisfies $s(e^+) \in V_{l}, \,\, t(e^+) \in V_{l+1},$ 
and for every vertex $v \in V_l$ with $l\ne 0,$ 
there exist $e^+ \in E_{l,l+1}^+, f^+ \in E_{l-1, l}^+$ such that 
$v = t(f^+) = s(e^+)$,
 and for every vertex $v \in V_0,$ 
there exists $e^+ \in E_{0,1}^+$ such that 
$v = s(e^+).$

\item %(resolving property)
(1) The condition $s(e^-) = s(f^-), \, \lambda^-(e^-) = \lambda^-(f^-)$
for $e^-, f^- \in E^-$
implies $e^- = f^-.$
This condition is said to be right-resolving 
for the labeling map $\lambda^-: E^-\longrightarrow \Sigma^-$.

(2) The condition
$t(e^+) = t(f^+), \, \lambda^+(e^+) = \lambda^+(f^+)$
for $e^+, f^+\in E^+$ implies $e^+ = f^+.$
This condition is said to be left-resolving
for the labeling map $\lambda^+: E^+\longrightarrow \Sigma^+$.

\item %(local property)
For every pair 
$u \in V_{l}, \, v \in V_{l+2}$ with $l \in \Zp$,
 we put
\begin{align*}
E_+^-(u,v)
: =  & \{ (e^-, e^+) \in E_{l, l+1}^-\times E_{l+1,l+2}^+ \mid 
t(e^-) = u, \, s(e^-) = s(e^+), t(e^+) = v \},\\
E_-^+(u,v)
:=  & \{ (f^+, f^-) \in E_{l, l+1}^+\times E_{l+1,l+2}^- \mid 
s(f^+) = u, \, t(f^+) = t(f^-), s(f^-) = v \}. 
\end{align*}
Then there exists a bijective correspondence 
$
\varphi:E_+^-(u, v) \longrightarrow E_-^+(u,v)
$
satisfying 
$\lambda^-(e^-) = \lambda^-(f^-), \, 
\lambda^+(e^+) = \lambda^+(f^+)$
whenever $\varphi(e^-, e^+) = (f^+, f^-).$ 
\end{enumerate}
The property (v) is called the local property of $\lambda$-graph bisystem.
The pair $\LGBS$ is called a $\lambda$-{\it graph bisystem over}\/ $\Sigma^\pm.$  
\end{definition}
We write 
$\{v_1^l, \dots, v_{m(l)}^l\}$ for the vertex set $V_l.$ 
The transition matrices $A^-_{l,l+1}, A^+_{l,l+1}$
for  ${\frak L}^-, {\frak L}^+$ respectively
are defined by setting
\begin{align}
A_{l,l+1}^-(i,\beta,j)
 & =
{\begin{cases}
1 &  
    \text{if} \ t(e^-) = v_i^{l}, \lambda^-(e^-) = \beta,
                       s(e^-) = v_j^{l+1} 
    \text{ for some }    e^- \in E_{l,l+1}^-, \\
0           & \text{ otherwise,}
\end{cases}} \label{eq:AM} \\
A_{l,l+1}^+(i,\alpha,j)
 & =
{\begin{cases}
1 &  
    \text{if} \ s(e^+) = v_i^{l}, \lambda^+(e^+) = \alpha,
                       t(e^+) = v_j^{l+1} 
    \text{ for some }    e^+ \in E_{l,l+1}^+, \\
0           & \text{ otherwise}
\end{cases}} \label{eq:AP} 
\end{align}
for
$
i=1,2,\dots,m(l),\ j=1,2,\dots,m(l+1),  \, \beta\in \Sigma^-, \, \alpha \in \Sigma^+.$ 
The local property of $\lambda$-graph bisystem $\LGBS$
ensures us the following identity
\begin{equation}
\sum_{j=1}^{m(l+1)}A_{l,l+1}^-(i,\beta,j)A_{l+1,l+2}^+(j,\alpha,k)
=\sum_{j=1}^{m(l+1)} A_{l,l+1}^+(i,\alpha,j)A_{l+1,l+2}^-(j,\beta,k) \label{eq:local}
\end{equation}
for $i=1,2,\dots,m(l),\ k=1,2,\dots,m(l+2),  \, \beta\in \Sigma^-, \, \alpha \in \Sigma^+.$ 
The pair 
$(A^-, A^+)=(A^-_{l,l+1}, A^+_{l,l+1})_{l\in\Zp}$ 
of the sequences of the matrices $A^-_{l,l+1}, A^+_{l,l+1}, l\in\Zp$ 
is called the transition matrix bisystem for $\LGBS$.
If its top vertex set $V_0$ is a singleton,
 $\LGBS$ is said to be {\it standard}.\/
If  $\Sigma^- = \Sigma^+,$
 $\LGBS$ is said to {\it have a common alphabet}.\/
In this case, we write the alphabet $\Sigma^- = \Sigma^+$
as $\Sigma$.
We write an edge 
$e^- \in E^-$ (resp. $e^+ \in E^+$) as $e$ without 
$-$ sign (resp. $+$ sign) unless we specify.

Let $({\frak L}^-, {\frak L}^+)$ be a $\lambda$-graph bisystem over $\Sigma^\pm.$
For a vertex $u \in V_l$, 
its follower set $F(u)$ in ${\frak L}^-$ and
its predecessor set $P(u)$ in ${\frak L}^+$
are defined in the following way:
\begin{align*}
F(u) := & \{ (\beta_1, \beta_2,\dots,\beta_l) \in {(\Sigma^{-})}^l \mid
\beta_1 =\lambda^-(f_l), \beta_2= \lambda^-(f_{l-1}),\dots, \beta_l=\lambda^-(f_1); \\ 
f_l \in E_{l-1,l}^-,& f_{l-1} \in E_{l-2, l-1}^-, \dots, f_1 \in E_{0,1}^-,\,\,  
 s(f_l) = u, t(f_l) = s(f_{l-1}), \dots, t(f_2) = s(f_1) \}.
\end{align*}
The set  $F(u)$ is figured such as  
\begin{equation*}
\qquad u \overset{\beta_1}{\longrightarrow} 
\bigcirc\overset{\beta_2}{\longrightarrow}
 \cdots \overset{\beta_{l-1}}{\longrightarrow}
\bigcirc\overset{\beta_l}{\longrightarrow}\bigcirc
\qquad \text{ in } \quad {\frak L}^-.
\end{equation*}
Similarly,
\begin{align*}
P(u) := & \{(\alpha_1, \alpha_2, \dots,\alpha_l)  \in {(\Sigma^{+})}^l \mid
\alpha_1=\lambda^+(e_1), \alpha_2=\lambda^+(e_2),\dots, \alpha_l=\lambda^+(e_l);\\ 
 e_1 \in E_{0,1}^+,& e_2 \in E_{1, 2}^+, \dots, e_l \in E_{l-1,l}^+,\,\,  
t(e_1) =s(e_2), t(e_2) = s(e_3), \dots, t(e_{l-1}) =s(e_l), t(e_l) = u \}.
\end{align*}
The set $P(u)$ is figured such as  
\begin{equation*}
 \qquad 
\bigcirc\overset{\alpha_1}{\longrightarrow} 
\bigcirc\overset{\alpha_2}{\longrightarrow}
 \cdots \overset{\alpha_{l-1}}{\longrightarrow}
\bigcirc\overset{\alpha_{l}}{\longrightarrow}u
\qquad \text{ in } \quad {\frak L}^+.
\end{equation*}
If a standard $\lambda$-graph bisystem $\LGBS$ 
having a common alphabet 
satisfies the condition
$F(u) = P(u)$  
for every vertex $u \in V_l, l\in \N,$
it is said to satisfy
 {\it{Follower-Predecessor Compatibility Condition},}
FPCC for brevity.

In \cite[Example 3.2]{MaPre2019}, several examples including $\lambda$-graph systems 
\cite{MaDocMath1999} are presented.
We will present a couple of examples of $\lambda$-graph bisystems
satisfying FPCC.

\begin{example} \label{ex:3.2}
\end{example}

{\bf (i)  A $\lambda$-graph bisystem for full $N$-shift.}

Let $N$ be a positive integer with $N >1.$
Take a finite alphabet
$\Sigma= \{ \alpha_1,\dots,\alpha_N\}$.
We will construct a $\lambda$-graph bisystem
$(\frak{L}_N^-,\frak{L}_N^+)$ satisfying FPCC
in the following way.
Let
$V_l=\{v_l\}$ one point set for each $l \in \Zp,$
and
$
 E_{l,l+1}^- =\{ e_1^-, \dots, e_N^-\}$,
$ E_{l,l+1}^+ =\{ e_1^+, \dots, e_N^+\}
$
such that 
\begin{gather*}
s(e_i^-) = v_{l+1}, \qquad t(e_i^-) = v_l, \qquad
\lambda^-(e_i^-) = \alpha_i \quad \text{ for } i=1,\dots, N, \, l \in \Zp,\\
s(e_i^+) = v_l, \qquad t(e_i^+) = v_{l+1}, \qquad
\lambda^+(e_i^+) = \alpha_i \quad \text{ for } i=1,\dots, N, \, l \in \Zp.   
\end{gather*}
We set
$
\frak{L}_N^- = (V, E^-, \lambda^-)$
and
$\frak{L}_N^+ = (V, E^+, \lambda^+)$.
Then   
$(\frak{L}_N^-,\frak{L}_N^+)$ is a $\lambda$-graph bisystem
satisfying FPCC.

\medskip

{\bf (ii) A standard $\lambda$-graph bisystem for golden mean shift.}

The topological Markov shift defined by the matrix  
$F = 
\begin{bmatrix}
1  & 1 \\
1 & 0
\end{bmatrix}
$
is called the golden mean shift (cf. \cite{LM}).
Let 
$\Sigma = \{\alpha, \beta\}$.
We set
$V_0= \{v_1^0\}, \,V_1 =\{ v_1^1, v_2^1\},\,
V_l =\{ v_1^l, v_2^l, v_3^l, v_4^l\}$ for $l \ge 2.$
The labeled Bratteli diagram ${\frak L}_F^-$
is defined as follows.
Define upward directed edges labeled symbols in $\Sigma$ 
such as
\begin{gather*}
v_1^0\overset{\alpha}{\longleftarrow}v_1^1, \quad
v_1^0\overset{\alpha}{\longleftarrow}v_2^1, \quad
v_1^0\overset{\beta}{\longleftarrow}v_1^1, \\
v_1^1\overset{\alpha}{\longleftarrow}v_1^2, \quad
v_1^1\overset{\alpha}{\longleftarrow}v_3^2, \quad
v_2^1\overset{\alpha}{\longleftarrow}v_2^2, \quad
v_2^1\overset{\alpha}{\longleftarrow}v_4^2, \quad
v_2^1\overset{\beta}{\longleftarrow}v_1^2, \quad
v_2^1\overset{\beta}{\longleftarrow}v_2^2, \\
v_1^l\overset{\alpha}{\longleftarrow}v_1^{l+1}, \quad
v_1^l\overset{\alpha}{\longleftarrow}v_3^{l+1}, \quad
v_2^l\overset{\alpha}{\longleftarrow}v_2^{l+1}, \quad
v_2^l\overset{\alpha}{\longleftarrow}v_4^{l+1}, \quad
v_3^l\overset{\beta}{\longleftarrow}v_1^{l+1}, \quad
v_4^l\overset{\beta}{\longleftarrow}v_2^{l+1} \\
\end{gather*}
for $l \ge2.$
The other labeled Bratteli diagram ${\frak L}_F^+$
is defined as follows.
Define downward directed edges labeled symbols in $\Sigma$ such as
\begin{gather*}
v_1^0\overset{\alpha}{\longrightarrow}v_1^1, \quad
v_1^0\overset{\alpha}{\longrightarrow}v_2^1, \quad
v_1^0\overset{\beta}{\longrightarrow}v_1^1, \\
v_1^1\overset{\alpha}{\longrightarrow}v_1^2, \quad
v_1^1\overset{\alpha}{\longrightarrow}v_2^2, \quad
v_2^1\overset{\alpha}{\longrightarrow}v_3^2, \quad
v_2^1\overset{\alpha}{\longrightarrow}v_4^2, \quad
v_2^1\overset{\beta}{\longrightarrow}v_1^2, \quad
v_2^1\overset{\beta}{\longrightarrow}v_3^2, \\
v_1^l\overset{\alpha}{\longrightarrow}v_1^{l+1}, \quad
v_1^l\overset{\alpha}{\longrightarrow}v_2^{l+1}, \quad
v_3^l\overset{\alpha}{\longrightarrow}v_3^{l+1}, \quad
v_3^l\overset{\alpha}{\longrightarrow}v_4^{l+1}, \quad
v_2^l\overset{\beta}{\longrightarrow}v_1^{l+1}, \quad
v_4^l\overset{\beta}{\longrightarrow}v_3^{l+1} \\
\end{gather*}
for $l \ge2.$
The pair $({\frak L}_F^-,{\frak L}_F^+)$
becomes a standard $\lambda$-graph bisystem satisfying FPCC.
It is figured in Figure \ref{fig:fiboii} in the end of this section.

\medskip

{\bf (iii) A (non-standard) $\lambda$-graph bisystem for golden mean shift.}

In the above example, we consider the 
labeled graphs for $l\ge2$.
Namely, 
$
V_l =\{ v_1^l, v_2^l, v_3^l, v_4^l\}$ for $l \ge 0.$
The labeled Bratteli diagram ${\frak L}_F^-$
is defined as follows.
Define upward directed edges labeled symbols in $\Sigma =\{\alpha,\beta\}$ 
such as
\begin{equation*}
v_1^l\overset{\alpha}{\longleftarrow}v_1^{l+1}, \quad
v_1^l\overset{\alpha}{\longleftarrow}v_3^{l+1}, \quad
v_2^l\overset{\alpha}{\longleftarrow}v_2^{l+1}, \quad
v_2^l\overset{\alpha}{\longleftarrow}v_4^{l+1}, \quad
v_3^l\overset{\beta}{\longleftarrow}v_1^{l+1}, \quad
v_4^l\overset{\beta}{\longleftarrow}v_2^{l+1} 
\end{equation*}
for $l \ge 0.$
Similarly,
the other labeled Bratteli diagram ${\frak L}_F^+$
is defined as follows.
Define downward directed edges labeled symbols in $\Sigma$ such as
\begin{equation*}
v_1^l\overset{\alpha}{\longrightarrow}v_1^{l+1}, \quad
v_1^l\overset{\alpha}{\longrightarrow}v_2^{l+1}, \quad
v_3^l\overset{\alpha}{\longrightarrow}v_3^{l+1}, \quad
v_3^l\overset{\alpha}{\longrightarrow}v_4^{l+1}, \quad
v_2^l\overset{\beta}{\longrightarrow}v_1^{l+1}, \quad
v_4^l\overset{\beta}{\longrightarrow}v_3^{l+1}
\end{equation*}
for $l \ge 0.$
The pair $({\frak L}_F^-,{\frak L}_F^+)$
becomes a  $\lambda$-graph bisystem satisfying FPCC.
It is figured in Figure \ref{fig:fiboiii} in the end of this section.

%\medskip
%The following is a generalization of (iii).

\medskip

{\bf (iv) A standard $\lambda$-graph bisystem for square matrices with entries in $\{0,1\}$.}

Let $M =[m_{i,j}]_{i,j=1}^N$ be an $N\times N$ square matrix with entries in $\{0,1\}$.  
We set 
$N^2 \times N^2$ matrices $M^-, M^+$
over $\Sigma$ by setting
\begin{equation*}
M^- := 
\begin{bmatrix}
m_{11} I_N & m_{21} I_N  & \cdots & m_{N1} I_N  \\
m_{12} I_N  & m_{22} I_N  & \cdots  & m_{N2} I_N  \\
\vdots              & \vdots              & \ddots & \vdots    \\
m_{1N} I_N  & m_{2N} I_N &  \cdots& m_{NN} I_N 
\end{bmatrix},
\qquad
M^+ := 
\begin{bmatrix}
M  & 0    & \dots & 0 \\
0     & M & \ddots & \vdots \\
\vdots     & \ddots  & \ddots & 0 \\
0     & \dots     &  0 & M
\end{bmatrix}
\end{equation*}
where $I_N$ is the $N\times N$ identity matrix. 
Let $V_0 = \{ v^0_1\}$ and $V_l =\{ v_{ij}\}_{i,j=1}^{N}$
be the $N^2$ vertex set for $l=1,2, \dots$,
and
$\Sigma =\{\alpha_{ij}\}_{i,j=1}^{N}$ 
an $N^2$ alphabet set.
We consider two Bratteli diagrams 
${\frak L}_{M}^-,{\frak L}_{M}^+$ over $\Sigma$ defined by the matrices
$M^-, M^+$, respectively in the following way.
To define upward Bratteli diagram 
${\frak L}_{M}^-$,
define upward directed edges labeled symbols in $\Sigma$ 
such as for $i,j,k=1,2,\dots,N$,
\begin{align*}
v_1^0 & \overset{\alpha_{ij}}{\longleftarrow}v_{jk}^{1}  \quad
\text{ in } {\frak L}_M^- \quad
\text{ if } m_{ij} \ne 0, \\
v_{ik}^l & \overset{\alpha_{ij}}{\longleftarrow}v_{jk}^{l+1} \quad
\text{ in } {\frak L}_M^- \quad
\text{ if } m_{ij} \ne 0 \quad \text{ for } l =1,2,\dots.
\end{align*}
Similarly,
to define downward Bratteli diagram 
${\frak L}_{M}^+$,
define downward directed edges labeled symbols in $\Sigma$ 
such as for $i,j,k=1,2,\dots,N$,
\begin{align*}
v_1^0 & \overset{\alpha_{ij}}{\longrightarrow}v_{kj}^{1}  \quad
\text{ in } {\frak L}_M^+ \quad
\text{ if } m_{ij} \ne 0, \\
v_{ki}^l& \overset{\alpha_{ij}}{\longrightarrow}v_{kj}^{l+1} \quad
\text{ in } {\frak L}_M^+ \quad
\text{ if } m_{ij} \ne 0 \quad \text{ for } l =1,2,\dots.
\end{align*}
The pair $({\frak L}_M^-,{\frak L}_M^+)$
becomes a  $\lambda$-graph bisystem satisfying FPCC
(see \cite[Lemma 4.2]{MaPre2019} for detail).

\bigskip

%%%%%%%%%%%%%%%%%%%%%%%%%%%%%%%%%
Let $\LGBS$ be a $\lambda$-graph bisystem satisfying FPCC.
We say that $\LGBS$ presents a subshift $\Lambda$ 
if the set of concatenated finite labeled paths in ${\frak L}^-$
coincides with the set $B_*(\Lambda) = \cup_{n=0}^\infty B_n(\Lambda)$ of admissible words of $\Lambda$,
and similarly the set in ${\frak L}^+$ also coincides with  $B_*(\Lambda)$. 
In \cite{MaPre2019}, 
it was proved that any $\lambda$-graph bisystem $\LGBS$ satisfying FPCC
presents a two-sided subshift, and conversely, 
 any two-sided subshift $\Lambda$ over $\Sigma$
is presented by a $\lambda$-graph bisystem $\LGBS$ satisfying FPCC.
It is written $({\frak L}_\Lambda^-, {\frak L}_\Lambda^+)$
and called the canonical $\lambda$-graph bisystem for $\Lambda$.
We will briefly review the construction  
of the canonical $\lambda$-graph bisystem 
$({\frak L}_\Lambda^-, {\frak L}_\Lambda^+)$ for $\Lambda$
following \cite{MaPre2019}.

We fix a subshift $\Lambda$ over $\Sigma.$
For $k,l \in \Z$ with $k < l$, 
we put $n(k,l) = l-k-1 \in \Zp$.
For  $x=(x_n)_{n\in \Z} \in \Lambda,$
we put a set of words
\begin{align*}
W_{k,l}(x)  
:= 
& \{ (\mu_{k+1},\mu_{k+2},\dots,\mu_{l-1}) \in B_{n(k,l)}(\Lambda) \mid \\
& (\dots, x_{k-1}, x_k, \mu_{k+1},\mu_{k+2},\dots,\mu_{l-1}, 
x_l, x_{l+1},\dots ) \in \Lambda \},
\end{align*}
%\begin{equation*}
%\cdots \underset{x_{k-1}}{\longrightarrow} \underset{x_k}{\longrightarrow}
%\square\square \cdots \square 
%\underset{x_l}{\longrightarrow} \underset{x_{l+1}}{\longrightarrow}
%\cdots
%\end{equation*}
and
$
W_{k,k+1}(x) := \emptyset. 
$
For $x, y \in \Lambda$,
if $W_{k,l}(x) = W_{k,l}(y),$
we write
$x\overset{c}{\underset{(k,l)}{\sim}}y$
and call it
$(k,l)$-{\it{centrally equivalent}}.
Define the set of equivalence classes
$$\Omega_{k,l}^c = \Lambda/\overset{c}{\underset{(k,l)}{\sim}},$$
that is a finite set because the set of words length less than or equal to 
$n(k,l)$ is finite.
Let $m(k,l) =|\Omega_{k,l}^c|$ the cardinal number of 
the finite set $\Omega_{k,l}^c$.
We denote by 
$\{ C_1^{k,l},C_2^{k,l}, \dots, C_{m(k,l)}^{k,l} \}
$ 
the set 
$\Omega_{k,l}^c$
of $\overset{c}{\underset{(k,l)}{\sim}}$ equivalence classes. 
Since
for $x, y \in \Lambda$, we have
$x\overset{c}{\underset{(k,l)}{\sim}}z$
if and only if
$\sigma(x)\overset{c}{\underset{(k-1,l-1)}{\sim}}\sigma(z),$
 we may identify 
$\Omega^c_{k,l}$ with $\Omega^c_{k-1,l-1}$
and 
$C_i^{k,l}$ with $C_i^{k-1,l-1}$ for $i=1,2,\dots,m(k,l)$
through the shift
$\sigma: \Lambda\longrightarrow \Lambda$
so that we identify
$\Omega^c_{k,l}$ with $\Omega^c_{k+n,l+n}$
and 
$C_i^{k,l}$ with $C_i^{k+n,l+n}$ for all 
$n \in \Z, \, i=1,2,\dots,m(k,l).$

%%%%%%%%%%%%%%%%%%%%%%%%%%%%%%%%%%%%
For 
$x=(x_n)_{n\in \Z}  \in C_i^{k,l}$ and $\beta \in \Sigma$
such that  
$ (\beta, \nu_{k+2},\nu_{k+3},\dots,\nu_{l-1}) \in W_{k,l}(x)$ 
for some $\nu =( \nu_{k+2},\nu_{k+3},\dots,\nu_{l-1}) \in B_{l-k-2}(\Lambda)$,
the bi-infinite sequence
\begin{equation*}
 x(\beta,\nu) : = 
(\dots, x_{k-1}, x_k, \beta,\nu_{k+2},\nu_{k+3},\dots,\nu_{l-1}, x_l, x_{l+1},\dots ) \in \Lambda
\end{equation*}
belongs to $\Lambda.$
If $x(\beta, \nu)$ belongs to $C_h^{k+1,l},$
we write
$C_i^{k,l}\beta \subset C_h^{k+1,l}.$
As in \cite[Lemma 4..2]{MaPre2019},
the notation $C_i^{k,l}\beta \subset C_h^{k+1,l}$  
is  well-defined, that is, it does not depend on the choice of 
$x=(x_n)_{n\in \Z}  \in C_i^{k,l}$ and
$\nu =( \nu_{k+2},\nu_{k+3},\dots,\nu_{l-1}) \in B_{l-k-2}(\Lambda)$
as long as  
$ (\beta, \nu_{k+2},\nu_{k+3},\dots,\nu_{l-1}) \in W_{k,l}(x).$

%%%%%%%%%%%%%%%%%%%%%%%%%%%%%%%%%%%%%%%%%%%%%%%%%%%%%
Similarly  for
$x=(x_n)_{n\in \Z}  \in C_i^{k,l}$ and $\alpha \in \Sigma$
such that  
$ (\mu_{k+1},\mu_{k+2},\dots,\mu_{l-2},\alpha) \in W_{k,l}(x)$ 
for some 
$\mu =( \mu_{k+1},\mu_{k+2},\dots,\mu_{l-2}) \in B_{l-k-2}(\Lambda)$,
the bi-infinite sequence
\begin{equation*}
 x(\mu, \alpha) : = 
(\dots, x_{k-1}, x_k, \mu_{k+1},\mu_{k+2},\dots,\mu_{l-2},\alpha, x_l, x_{l+1},\dots ) \in \Lambda
\end{equation*}
belongs to $\Lambda.$
If $x(\mu, \alpha)$ belongs to $C_j^{k,l-1},$
 we write
$\alpha C_i^{k,l} \subset C_j^{k,l-1}.$
The notation $\alpha C_i^{k,l} \subset C_j^{k,l-1}$
is also well-defined, that is, it does not depend on the choice of 
$x=(x_n)_{n\in \Z}  \in C_i^{k,l}$ and
$\mu =( \mu_{k+1},\mu_{k+2},\dots,\mu_{l-2}) \in B_{l-k-2}(\Lambda)$
as long as  
$ (\mu_{k+1},\mu_{k+2},\dots,\mu_{l-2},\alpha) \in W_{k,l}(x).$

%%%%%%%%%%%%%%%%%%%%%%%%%%%%%%%%%%%%%%%%%%%%%%%%%%%%%%%%%%%%%%%%
Let us in particular specify the following equivalence classes
$\Omega^c_{-l,1}$ and $\Omega^c_{-1,l}$
and define the vertex sets $V_l^-$ and $V_l^+$ for $l=0,1,2,\dots$
by setting
\begin{gather*}
V_0^-:= \{ \Lambda\},\quad 
 V_1^-:= \Omega^c_{-1,1},\quad
 V_2^-:= \Omega^c_{-2,1}, \quad
 V_3^-:= \Omega^c_{-3,1}, \quad
\cdots,  \quad   
V_l^-:= \Omega^c_{-l,1}, \quad 
\cdots, \\
V_0^+:= \{ \Lambda\}, \quad
V_1^+:= \Omega^c_{-1,1}, \quad
V_2^+:= \Omega^c_{-1,2}, \quad
V_3^+:= \Omega^c_{-1,3}, \quad
\cdots, \quad
V_l^+:= \Omega^c_{-1,l}, \quad 
\cdots.
\end{gather*}
Write $V_0 = V_0^- = V_0^+.$
The each classes $ C_{i}^{-l,1}$ and $C_{i}^{-1,l}$ for $i=1,2,\dots,m(l)$ 
are identified  and written 
$C_i^l$, by
the bijective correspondence
$
x \in V_l^+ \longleftrightarrow \sigma^{l-1}(x) \in V_l^-,
$
and hence 
 $V_l^-$ and $V_l^+$ 
are identified  for each $l\in \Zp,$
that are  denoted by $V_l.$
We regard $C_i^l$ 
as a vertex denoted by $v_i^l$
and define an edge $e^-$ labeled $\beta\in \Sigma$ from
$v_j^{l+1}$ to $v_i^l$  
if $C_j^{l+1}\beta \subset C_i^l$.
We write $s(e^-) = v_j^{l+1}$, 
$t(e^-) = v_i^l$ %the terminal vertex of $e^-,$
and  $\lambda^-(e^-) = \beta.$
The set of such edges from
 $v_j^{l+1}$ to $v_i^l$  for some $i=1,\dots,m(l), \, j=1,\dots, m(l+1)$
is denoted by
$E_{l,l+1}^-.$
Similarly, we define 
an edge $e^+$ labeled $\alpha\in \Sigma$ from
$v_i^l$ to $v_j^{l+1}$
if $\alpha C_j^{l+1}\subset C_i^l$.
We write $s(e^+) = v_i^l$,
$t(e^+) = v_j^{l+1}$  % the terminal vertex of $e^+,$
and $\lambda^+(e^+) = \alpha.$
The set of such edges from
$v_i^l$ to $v_j^{l+1}$ for some $i=1,\dots,m(l), \, j=1,\dots, m(l+1)$
is denoted by
$E_{l,l+1}^+.$
%%%%%%%%%%%%%%%%
%%%%%%%%%%%%%%%
These are written
\begin{align*}
v_j^{l+1} \overset{\beta}{\underset{e^-}{\longrightarrow}} v_i^l
\qquad 
&\text{ if }\qquad C_j^{l+1}\beta \subset C_i^l, \\
v_i^l \overset{\alpha}{\underset{e^+}{\longrightarrow}} v_j^{l+1}
\qquad 
&\text{ if } \qquad \alpha C_j^{l+1}\subset C_i^l.
\end{align*}
We set 
$
E^- =\cup_{l=0}^{\infty} E_{l,l+1}^-, \,
E^+ =\cup_{l=0}^{\infty} E_{l,l+1}^+.
$ 
We now have  the pair 
$ {\frak L}^-_\Lambda =(V,E^-, \lambda^-)$ and
$ {\frak L}^+_\Lambda =(V,E^+, \lambda^+)$ of labeled Bratteli diagrams.
\begin{proposition}[{\cite[Proposition 4.4]{MaPre2019}}] \label{prop:4.4}
The pair $\LLGBS$ of labeled Bratteli diagrams
is a $\lambda$-graph bisystem satisfying FPCC and presenting the subshift $\Lambda.$ 
\end{proposition}
The $\lambda$-graph bisystem 
$({\frak L}_\Lambda^-, {\frak L}_\Lambda^+)$
for subshift $\Lambda$ is called the {\it canonical}\/  $\lambda$-graph bisystem 
for $\Lambda$. 
The $\lambda$-graph bisystems presented in Example \ref{ex:3.2}
(i), (ii)  are the canonical $\lambda$-graph bisystems for the full $N$-shift,
the golden mean shift, respectively.

\newpage

%%%%%%%%%%%%%%%%%%%%%%%%%%%%%%%%%%%%%%%%%
%\begin{figure}[htbp]
%\begin{center}
%\input{pictureiMP.tex}
%\end{center}
%\caption{ $\lambda$-graph bisystem 
%$({\frak L}^-,{\frak L}^+)$ of Example \ref{ex:3.2} (i) }
%\label{fig:(i)}
%\end{figure}
%%%%%%%%%%%%%%%%%%%%%%%%%%%%%%%%%%%%%%%%%%%%%%%%%%%
%where upward arrows $\longleftarrow$ 
%in ${\frak L}^-$ are labeled $\iota$,
%whereas downward arrows $\longleftarrow$ and {\mathversion{bold} $\longleftarrow$} (bold)
%in ${\frak L}^+$ are labeled $0$ and $1$, respectively.  

%\newpage

%%%%%%%%%%%%%%%%%%%%%%%%%%%%%%%%%%%%%%%%%
\begin{figure}[htbp]
\begin{center}
\input{pictureFiboMP2.tex}
\end{center}
\caption{  A standard $\lambda$-graph bisystem $({\frak L}_F^-,{\frak L}_F^+)$ of Example \ref{ex:3.2} (ii) }
\label{fig:fiboii}
\end{figure}
%%%%%%%%%%%%%%%%%%%%%%%%%%%%%%%%%%%%%%%%%%%%%%%%%%%
where upward arrows $\longleftarrow$ and {\mathversion{bold} $\longleftarrow$} (bold)
in ${\frak L}_F^-$ are labeled $\alpha$ and $\beta$, respectively,
whereas downward arrows $\longleftarrow$ and {\mathversion{bold} $\longleftarrow$} (bold)
in ${\frak L}_F^+$ are labeled $\alpha$ and $\beta$, respectively.  

\newpage

%%%%%%%%%%%%%%%%%%%%%%%%%%%%%%%%%%%%%%%%%
\begin{figure}[htbp]
\begin{center}
\input{picturefiboMP3.tex}
\end{center}
\caption{ A (non standard)$\lambda$-graph bisystem  $({\frak L}_F^-,{\frak L}_F^+)$
 of Example \ref{ex:3.2} (iii) }
\label{fig:fiboiii}
\end{figure}
%%%%%%%%%%%%%%%%%%%%%%%%%%%%%%%%%%%%%%%%%%%%%%%%%%%
where upward arrows $\longleftarrow$ and {\mathversion{bold} $\longleftarrow$} (bold)
in ${\frak L}_F^-$ are labeled $\alpha$ and $\beta$, respectively,
whereas downward arrows $\longleftarrow$ and {\mathversion{bold} $\longleftarrow$} (bold)
in ${\frak L}_F^+$ are labeled $\alpha$ and $\beta$, respectively.  

%%%%%%%%%%%%%%%%%%%%%%%%%%%%%%%%%%%%%%%%%%%%%%
%%%%%%%%%%%%%%%%%%%%%%%%%%%%%%%%%%%%%%

\newpage

%%%%%%%%%%%%%%%%%%%%%%%%%%%%%%%%%%%%%%%%%%%%%%%%%
%%%%%%%%%%%%%%%%%%%%%%%%%%%%%%
\section{Two-sided AF algebras } 
%%%%%%%%%%%%%%%%%%%%%%%%%%%%%%%%%%
Suppose that $\LGBS$ satisfy FPCC, that is,
$P(v_i^n) = F(v_i^n)$ for all $ i =1,2,\dots,m(n), \, n \in \Zp.$
Let us denote by $\Lambda$ the subshift presented by  $\LGBS$. 
The set 
$P(v_i^n) (= F(v_i^n))$ is denoted by 
$W_i^n$.
Put
$N_i^{n} = |W_i^{n}| \in \Zp$
and consider the
$N_i^{n}\times N_i^{n}$
 full matrix algebra
$M_{N_i^n}({\mathbb{C}})$ over ${\mathbb{C}}$.
Let
$E_i^{n}(\mu,\nu), \, \mu, \nu \in W_i^n$ 
be the set of matrix units of the algebra
$M_{N_i^{n}}({\mathbb{C}})$.   
They are partial isometries 
satisfying the identities
\begin{equation*}
E_i^{n}(\mu,\nu) E_i^{n}(\mu,\nu)^* = E_i^{n}(\mu,\mu),\qquad
E_i^{n}(\mu,\nu)^* =E_i^{n}(\nu,\mu), \qquad
\sum_{\mu \in W_i^{n}} E_i^{n}(\mu,\mu) =1.
\end{equation*}
For $k,l\in \Z$ with $k<l$,
put $n(k,l) = l-k-1 \in \Zp$
and
$m(k,l) = m(n(k,l)) = m(l-k-1)$.
We define partial isometries
$E_i^{k,l}(\mu,\nu)$ for 
$ \mu,\nu \in B_{n(k,l)}(\Lambda), 
i=1,2,\dots,m(k,l)
$
by setting
\begin{equation*}
E_i^{k,l}(\mu,\nu)
:= \begin{cases}
E_i^{n(k,l)}(\mu,\nu) & \text{ if } \mu, \nu \in W_i^{n(k,l)}, \\
0   & \text{ otherwise.} 
\end{cases}
\end{equation*}
%%%%%%%%%%%%%%%%%%%%%%%%%%%%
Define the full matrix algebra $\F_i^{k,l}$ for the vertex 
$v_i^{n(k,l)} \in V_{n(k,l)}$
by setting 
\begin{equation}
\F_i^{k,l} := C^*(E_i^{k,l}(\mu,\nu) \mid \mu, \nu \in W_i^{n(k,l)})
\end{equation} 
that is the $C^*$-algebra generated by the partial isometries
$E_i^{k,l}(\mu,\nu), \mu, \nu \in W_i^{n(k,l)}.$
It is isomorphic to $M_{N_i^{n(k,l)}}({\mathbb{C}})$.
Define also the finite dimensional $C^*$-algebra
$\F_{\frak L}^{k,l}$ by
\begin{equation}
\F_{\frak L}^{k,l} := \bigoplus_{i=1}^{m(nk,l}\F_i^{k,l} 
(= \bigoplus_{i=1}^{m(k,l)}M_{N_i^{n(k,l)}}({\mathbb{C}})).
\end{equation} 
The following lemma is obvious.
\begin{lemma}
For $k,l,k',l' \in \Z$ with $k<l$ and $k' <l'$, 
assume $l-k = l' -k'.$
Then the $C^*$-algebras 
$\F_i^{k,l}$ and $\F_i^{k',l'}$ are isomorphic for each $i=1,2, \dots,m(k,l)$
and hence 
$\F_{\frak L}^{k,l}$ and $\F_{\frak L}^{k',l'}$ are isomorphic.
\end{lemma}
Recall that  $(A^-, A^+) =(A^-_{n,n+1}, A^+_{n,n+1})_{n\in \Zp}$
denotes the transition matrix bisystem for the $\lambda$-graph bisystem $\LGBS$ defined by 
\eqref{eq:AM} and \eqref{eq:AP}.
We put
\begin{align*}
\iota_+(E_i^{k,l}(\mu,\nu))
= & \sum_{\alpha \in \Sigma}\sum_{j=1}^{m(k,l+1)}A^+_{n(k,l), n(k,l+1)}(i,\alpha,j)
E_j^{k,l+1}(\mu \alpha,\nu\alpha), \\
\iota_-(E_i^{k,l}(\mu,\nu))
= & \sum_{\beta \in \Sigma}\sum_{j=1}^{m(k-1,l)}A^-_{n(k,l), n(k-1,l)}(i,\beta,j)
E_j^{k-1,l}(\beta\mu,\beta\nu), 
\end{align*}
where $n(k,l) = l-k-1$ and $m(k,l) = m(n(k,l)) = m(l-k-1)$,
and hence $n(k,l+1) =n(k-1,l) = n(k,l) +1$.
\begin{lemma}
The above maps define unital embeddings  
\begin{equation*}
\iota_+: \F_{\frak L}^{k,l} \hookrightarrow \F_{\frak L}^{k,l+1}, \qquad 
\iota_-: \F_{\frak L}^{k,l} \hookrightarrow \F_{\frak L}^{k-1,l} 
\end{equation*}
of $C^*$-algebras.
\end{lemma}
\begin{proof}
We have
\begin{align*}
 & \iota_+(E_i^{k,l}(\mu,\nu))\iota_+(E_{i'}^{k,l}(\mu',\nu')) \\
= &
\left( \sum_{\alpha \in \Sigma}\sum_{j=1}^{m(k,l+1)}A^+_{n(k,l), n(k,l+1)}(i,\alpha,j)
E_j^{k,l+1}(\mu \alpha,\nu\alpha) \right) \\
&\cdot
\left( \sum_{\alpha' \in \Sigma}\sum_{j'=1}^{m(k,l+1)}A^+_{n(k,l), n(k,l+1)}(i',\alpha',j')
E_{j'}^{k,l+1}(\mu' \alpha',\nu'\alpha') \right) \\
= &
 \sum_{\alpha,\alpha' \in \Sigma}
 \sum_{j, j'=1}^{m(k,l+1)}
 A^+_{n(k,l), n(k,l+1)}(i,\alpha,j) A^+_{n(k,l), n(k,l+1)}(i',\alpha',j') \\
& E_j^{k,l+1}(\mu \alpha,\nu\alpha)
E_{j'}^{k,l+1}(\mu' \alpha',\nu'\alpha'). 
\end{align*}
Since
\begin{align*}
& E_j^{k,l+1}(\mu \alpha,\nu\alpha)E_{j'}^{k,l+1}(\mu' \alpha',\nu'\alpha') \\
=& 
{\begin{cases}
E_j^{k,l+1}(\mu \alpha,\nu'\alpha) & \text{ if } \
j= j' \text{ and } \nu = \mu', \alpha = \alpha',  \\
0 & \text{ otherwise},
\end{cases}
}
\end{align*}
and
$$
 A^+_{n(k,l), n(k,l+1)}(i,\alpha,j) A^+_{n(k,l), n(k,l+1)}(i',\alpha,j) =0
\quad\text{ for } i \ne i'
$$
because 
${\frak L}^+$ is left-resolving,
we have
\begin{align*}
 & \iota_+(E_i^{k,l}(\mu,\nu))\iota_+(E_{i'}^{k,l}(\mu',\nu')) \\
= &
{\begin{cases}
\sum_{\alpha \in \Sigma}\sum_{j=1}^{m(k,l+1)}A^+_{n(k,l), n(k,l+1)}(i,\alpha,j)
E_j^{k,l+1}(\mu \alpha,\nu'\alpha) & \text{ if } \nu = \mu', i = i',\\
0 & \text{ otherwise }
\end{cases}
} \\
= &
{\begin{cases}
\iota_+(E_i^{k,l}(\mu,\nu')) & \text{ if } \nu = \mu', i= i', \\
0 & \text{ otherwise }
\end{cases}
} \\
= & \iota_+(E_i^{k,l}(\mu,\nu) E_{i'}^{k,l}(\mu',\nu')). 
\end{align*}
Hence 
$\iota_+: \F_{\frak L}^{k,l} \longrightarrow \F_{\frak L}^{k,l+1}
$
may extend to a homomorphism from
$ \F_{\frak L}^{k,l} $ to $\F_{\frak L}^{k,l+1}$,
that is clearly injective.
Similarly we know that 
$ 
\iota_-: \F_{\frak L}^{k,l} \longrightarrow \F_{\frak L}^{k-1,l} 
$
yields an injective homomorphism of $C^*$-algebras.

We will next show that  
$ 
\iota_+: \F_{\frak L}^{k,l} \longrightarrow \F_{\frak L}^{k,l+1} 
$
and similarly 
$ 
\iota_-: \F_{\frak L}^{k,l} \longrightarrow \F_{\frak L}^{k-1,l} 
$
are unital.
Put the projections
$$
P_i^{k,l} := \sum_{\mu \in W_i^{n(k,l)}} E_i^{k,l}(\mu,\mu)
\quad
\text{ and }
\quad
P^{k,l}_{\frak L} :=\sum_{i=1}^{m(k,l)}P_i^{k,l}.  
$$
They are the unit of $\F_{i}^{k,l}$
and
that of $\F_{\frak L}^{k,l}$, respectively.
We have
\begin{align*}
\iota_+(P^{k,l}_{\frak L})
= & \sum_{i=1}^{m(k,l)} \sum_{\mu \in W_i^{n(k,l)}} 
    \sum_{\alpha \in \Sigma}\sum_{j=1}^{m(k,l+1)}
    A^+_{n(k,l), n(k,l+1)}(i,\alpha,j)E_j^{k,l+1}(\mu \alpha,\mu\alpha)\\
= & \sum_{j=1}^{m(k,l+1)}
\left( \sum_{i=1}^{m(k,l)} \sum_{\mu \in W_i^{n(k,l)}} 
    \sum_{\alpha \in \Sigma}
    A^+_{n(k,l), n(k,l+1)}(i,\alpha,j)E_j^{k,l+1}(\mu \alpha,\mu\alpha)\right).
\end{align*}
For a fixed $j=1,2,\dots, m(k,l+1)$, we see 
\begin{equation*}
\bigcup_{i=1}^{m(k,l)} \bigcup_{\mu \in W_i^{n(k,l)}}
\bigcup_{\alpha\in \Sigma}
\{\mu\alpha\mid A^+_{n(k,l), n(k,l+1)}(i,\alpha,j) =1\}
= W_j^{n(k,l+1)},
\end{equation*}
so that 
\begin{equation*}
\iota_+(P^{k,l}_{\frak L})=  \sum_{j=1}^{m(k,l+1)}P_j^{k,l+1} = P^{k,l+1}_{\frak L}.
\end{equation*}
The other equality
$\iota_-(P^{k,l}_{\frak L})=  P^{k-1,l}_{\frak L}
$ 
is similarly shown.
\end{proof}

\begin{lemma}
For $k,l\in \Z$ with $k <l,$
the following diagram is commutative:
\begin{equation*}
\begin{CD}
\F_{\frak L}^{k,l} @>\iota_+>> 
\F_{\frak L}^{k,l+1}  \\
@V{\iota_-}VV  @V{\iota_-}VV \\
 \F_{\frak L}^{k-1,l}
@>\iota_+>> 
 \F_{\frak L}^{k-1,l+1}. \\ 
\end{CD}
\end{equation*}
\end{lemma}
\begin{proof}
We have
\begin{align*}
  & (\iota_-\circ\iota_+)(E_i^{k,l}(\mu,\nu)) \\
= & \sum_{\alpha \in \Sigma}\sum_{h=1}^{m(k,l+1)}A^+_{n(k,l), n(k,l+1)}(i,\alpha,h)
\iota_-(E_h^{k,l+1}(\mu \alpha,\nu\alpha)) \\
%= & \sum_{\alpha \in \Sigma}\sum_{h=1}^{m(k,l+1)}A^+_{n(k,l), n(k,l+1)}(i,\alpha,h)
%\left(
%\sum_{\beta \in \Sigma}\sum_{j=1}^{m(k-1,l+1)}A^-_{n(k,l+1), n(k-1,l+1)}(h,\beta,j)
%E_j^{k-1,l+1}(\beta\mu\alpha,\beta\nu\alpha)
%\right) \\
= & \sum_{\alpha,\beta \in \Sigma}
\sum_{j=1}^{m(k-1,l+1)}
\left(
\sum_{h=1}^{m(k,l+1)}A^+_{n(k,l), n(k,l+1)}(i,\alpha,h)A^-_{n(k,l+1), n(k-1,l+1)}(h,\beta,j)
\right)
E_j^{k-1,l+1}(\beta\mu\alpha,\beta\nu\alpha). \\
\intertext{On the other hand}
   &(\iota_+\circ\iota_-)(E_i^{k,l}(\mu,\nu)) \\
= & \sum_{\beta \in \Sigma}\sum_{h=1}^{m(k-1,l)}A^-_{n(k,l), n(k-1,l)}(i,\beta,h)
\iota_+(E_h^{k-1,l}(\beta\mu ,\beta\nu))\\
%= & \sum_{\beta \in \Sigma}\sum_{h=1}^{m(k-1,l)}A^-_{n(k,l), n(k-1,l)}(i,\beta,h)
%\left(
%\sum_{\alpha \in \Sigma}\sum_{j=1}^{m(k-1,l+1)}A^+_{n(k-1,l), n(k-1,l+1)}(h,\alpha,j)
%E_j^{k-1,l+1}(\beta\mu\alpha,\beta\nu\alpha)
%\right) \\
= & \sum_{\alpha, \beta \in \Sigma}
\sum_{j=1}^{m(k-1,l+1)}
\left(
\sum_{h=1}^{m(k-1,l)}A^-_{n(k,l), n(k-1,l)}(i,\beta,h)
A^+_{n(k-1,l), n(k-1,l+1)}(h,\alpha,j)
\right)
E_j^{k-1,l+1}(\beta\mu\alpha,\beta\nu\alpha).
\end{align*}
The local property of $\lambda$-graph bisystem \eqref{eq:local}
tells us the identity
\begin{align*}
 & 
\sum_{h=1}^{m(k,l+1)}A^+_{n(k,l), n(k,l+1)}(i,\alpha,h)A^-_{n(k,l+1), n(k-1,l+1)}(h,\beta,j) \\
= &\sum_{h=1}^{m(k-1,l)}A^-_{n(k,l), n(k-1,l)}(i,\beta,h)
A^+_{n(k-1,l), n(k-1,l+1)}(h,\alpha,j),
\end{align*}
so that we conclude that 
$$
(\iota_-\circ\iota_+)(E_i^{k,l}(\mu,\nu))
=(\iota_+\circ\iota_-)(E_i^{k,l}(\mu,\nu)).
$$
\end{proof}
Hence we have an inductive system of finite dimensional $C^*$-algebras:
\begin{equation}
\{ \iota_+: \F_{\frak L}^{k,l} \longrightarrow \F_{\frak L}^{k,l+1}, \quad
\iota_-: \F_{\frak L}^{k,l} \longrightarrow \F_{\frak L}^{k-1,l} \mid
k,l\in \Z, \, k<l\}. \label{eq:indAF}
\end{equation}
\begin{definition}
Two-sided AF algebra $\F_{\frak L}$ 
associated with  a $\lambda$-graph bisystem 
$({\frak L}^-, {\frak L}^+)$
satisfying FPCC
is defined by the AF-algebra 
defined by the inductive system \eqref{eq:indAF}.
If in particular, 
$({\frak L}_\Lambda^-, {\frak L}_\Lambda^+)$
is the canonical $\lambda$-graph bisystem 
for a subshift $\Lambda$,
we write the AF-algebra
$\F_{{\frak L}_\Lambda}$ as
$\F_\Lambda$
and call it the AF-algebra defined by the two-sided subshift $\Lambda$.
%%%
%Two-sided AF algebra $\F_\Lambda$ 
%associated with  a two-sided subshift $\Lambda$is defined by 
%\begin{equation*}
%\F_\Lambda = \overline{\bigcup_{k,l\in \Z; k<l} \F_\Lambda^{k,l}}
%\end{equation*}
%the union of the finite dimensional $C^*$-algebras 
%$\bigcup_{k,l\in \Z; k<l} \F_\Lambda^{k,l}.$
%%%%%%%%%%%%%%%%%%%%%%%%%%
\end{definition}

Let $\LGBS$ be a $\lambda$-graph bisystem.
For the labeled Bratteli diagram
${\frak L}^-=(V,E^-,\lambda^-)$, 
let 
${}^t\!{\frak L}^{-} =(V, {}^t\!E^{-},\lambda^-)$ be the labeled Bratteli diagram
whose edges are reversed in its directions, that is the edge set 
${}^t\!E^{-}$ is defined by 
${}^t\!E^{-} = \cup_{l=0}^{\infty}{}^t\!E_{l,l+1}^{-}$
where ${}^t\!E_{l,l+1}^{-}$
is the set of all edges obtained by reversing its directions of all edges of $E_{l,l+1}^-$.
We then have two labeled Bratteli diagrams
${}^t\!{\frak L}^{-}$  and 
${\frak L}^+$ both of which have downward edges and the same vertex set $V$.
Let us 
${\frak B}_{({}^t\!{\frak L}^{-},{\frak L}^{+})}$ be the Bratteli diagram
$(V,{}^t\!E^{-}\cup E^+, \lambda^{-}\cup\lambda^+)$,
where
${}^t\!E^{-}\cup E^+ = \cup_{l=0}^\infty ({}^t\!E^{-}_{l,l+1}\cup E^+_{l,l+1})$. 
%and
%$$
%\lambda^{-}\cup\lambda^+(e) =
%\begin{cases}
%\lambda^-(e) & \text{ if } e\in E^{-t}, \\
%\lambda^+(e) & \text{ if } e\in E^{+}.
%\end{cases}
%$$
\begin{definition}
A $\lambda$-graph bisystem $\LGBS$  is said to be {\it irreducible}\/
if for any vertex $v_i^l \in V_l$, there exists $L \in \N$ with $L >l$
such that 
for any vertex $v_j^L \in V_L$ there exists a labeled path of 
${\frak B}_{({}^t\!{\frak L}^{-},{\frak L}^{+})}$
from
$v_i^l $ to $v_j^L. $  
\end{definition}
We thus reach the following proposition
\begin{proposition}\label{prop:simplicity}
Let $\LGBS$ be a $\lambda$-graph bisystem satisfying FPCC.
If  $\LGBS$ is irreducible,
then
the AF-algebra $\F_{\frak L}$ is simple.
\end{proposition}
\begin{proof}
It is well-known that an ideal of an AF-algebra bijectively corresponds to
a hereditary subset of the Bratteli diagram which defines the AF-algebra
(\cite{Bratteli}).
 The irreducibility ensures us that there is no nontrivial hereditary subset of 
the Bratteli diagram ${\frak B}_{({}^t\!{\frak L}^{-},{\frak L}^{+})}$, so that the algebra
$\F_{\frak L}$ is simple.
 \end{proof}

%\newpage

%%%%%%%%%%%%%%%%%%%%%%%%%%%%%%%%%%%%%%%
\section{$C^*$-algebras associated with two-sided subshifts}
%%%%%%%%%%%%%%%%%%%%%%%%%%%%%%%%%%%%%%%%%
Let $\LGBS$ be a $\lambda$-graph bisystem satisfying FPCC.
Let $\Lambda$ be its presenting subshift.
In the first half of this section, we will construct an automorphism $\rho_{\frak L}$ on the AF-algebra
$\F_{\frak L}$ arising from the shift $\sigma$ 
on the subshift $\Lambda$, 
and then define the $C^*$-algebra $\R_{\frak L}$
as the crossed product
$\F_{\frak L}\rtimes_{\rho_{\frak L}}\Z$.
In the second half of the section,
we will give a condition on $\LGBS$ under which the $C^*$-algebra 
$\F_{\frak L}\rtimes_{\rho_{\frak L}}\Z$
becomes simple.
 
For the transition matrix bisystem 
$(A^-, A^+)$ for $\LGBS$, 
we define matrices $A^-_{n_1, n_2}, A^+_{n_1, n_2}$,
for $n_1, n_2 \in \Zp$ with $n_1< n_2$
and $\gamma = (\gamma_1,\dots,\gamma_{n_2-n_1}) \in B_{n_2-n_1}(\Lambda)$
by setting,
\begin{align*}
A^-_{n_1, n_2}(i,\gamma,j) 
=
\sum_{j_{n_2-n_1-1}=1}^{m(n_2-1)}\cdots \sum_{j_1=1}^{m(n_1+1)} 
A^-_{n_1,n_1+1}(i,\gamma_1,j_1)
& A^-_{n_1+1,n_1+2}(j_1,\gamma_2,j_2)
\cdots \\
\cdots & A^-_{n_2-1,n_2}(j_{n_2-n_1-1},\gamma_{n_2-n_1}, j), \\
A^+_{n_1, n_2}(i,\gamma,j) 
=
\sum_{j_{n_2-n_1-1}=1}^{m(n_2-1)} \cdots \sum_{j_1=1}^{m(n_1+1)} 
A^+_{n_1,n_1+1}(i,\gamma_1,j_1)
 & A^+_{n_1+1,n_1+2}(j_1,\gamma_2,j_2)
   \cdots \\
\cdots & A^+_{n_2-1,n_2}(j_{n_2-n_1-1},\gamma_{n_2-n_1}, j).
\end{align*}

\begin{lemma}\label{lem:commute}
For $k_1, l_1,k_2, l_2\in \Z$ such that 
$k_1 < l_1 \le k_2 < l_2,$ 
the two projections
$$
\iota_+^{l_2-l_1}(E_{i_1}^{k_1,l_1}(\mu, \mu)), \qquad
\iota_-^{k_2-k_1}(E_{i_2}^{k_2,l_2}(\nu, \nu))
$$
for $\mu \in W_{i_1}^{n(k_1, l_1)}, \nu \in W_{i_2}^{n(k_2, l_2)}
$ commute to each other  in $\F_{\frak L}^{k_1, l_2}.$
 \end{lemma}
\begin{proof}
We have
\begin{align*}
& \iota_-^{k_2-k_1}(E_{i_2}^{k_2,l_2}(\nu, \nu)) \\
= & \sum_{\eta\in B_{k_2-k_1}(\Lambda)}
\sum_{j_2=1}^{m(k_1,l_2)} A^-_{n(k_2,l_2), n(k_1,l_2)}(i_2,\eta, j_2)
E_{j_2}^{k_1,l_2}(\eta\nu,\eta\nu) 
\quad \text{ in } \F_{\frak L}^{k_1,l_2}\\
\intertext{and }
& \iota_+^{l_2-l_1}(E_{i_1}^{k_1,l_1}(\mu, \mu)) \\
= & \sum_{\gamma\in B_{l_2-l_1}(\Lambda)}
\sum_{j_1=1}^{m(k_1,l_2)} A^+_{n(k_1,l_1), n(k_1,l_2)}(i_1,\gamma, j_1)
E_{j_1}^{k_1,l_2}(\mu\gamma, \mu\gamma) 
\quad \text{ in } \F_{\frak L}^{k_1,l_2}.
\end{align*}
We note that
\begin{equation}
 E_{j_1}^{k_1,l_2}(\mu\gamma, \mu\gamma) 
  E_{j_2}^{k_1,l_2}(\eta\nu,\eta\nu) 
=
{\begin{cases}
E_{j_1}^{k_1,l_2}(\mu\gamma, \eta\nu) & 
\text{ if } j_1= j_2, \,  \mu\gamma =\eta\nu,\\
0 & \text{ otherwise},
\end{cases}
} \label{eq:Ej1Ej2}
\end{equation}
and
\begin{equation}
E_{j_2}^{k_1,l_2}(\eta\nu,\eta\nu)
  E_{j_1}^{k_1,l_2}(\mu\gamma, \mu\gamma) 
=
{\begin{cases}
E_{j_1}^{k_1,l_2}(\eta\nu, \mu\gamma) & 
\text{ if } j_1= j_2, \,  \eta\nu =\mu\gamma,\\
0 & \text{ otherwise}.
\end{cases}
} \label{eq:Ej2Ej1}
\end{equation}
As 
$\mu\gamma =\eta\nu$
 if and only if 
$\gamma = \xi \nu, \eta = \mu \xi$ for some 
$\xi \in B_{k_2-l_1}(\Lambda)$,
%For the both cases \eqref{eq:Ej1Ej2} and \eqref{eq:Ej2Ej1}, we have
%\begin{equation*}
%E_{j_1}^{k_1,l_2}(\mu\gamma, \eta\nu) 
%=E_{j_1}^{k_1,l_2}(\mu\xi\nu, \mu\xi\nu)
%\quad \text{ if }
%j_1= j_2, \, \mu \gamma = \eta \nu.
%\end{equation*}
%Hence 
we have
\begin{align*}
& \iota_+^{l_2-l_1}(E_{i_1}^{k_1,l_1}(\mu, \mu)) \iota_-^{k_2-k_1}(E_{i_2}^{k_2,l_2}(\nu, \nu)) \\
= & \sum_{\xi\in B_{k_2-l_1}(\Lambda)}
\sum_{j=1}^{m(k_1,l_2)} 
A^+_{n(k_1,l_1), n(k_1,l_2)}(i_1,\xi\nu, j) A^-_{n(k_2,l_2), n(k_1,l_2)}(i_2,\mu\xi, j)
E_{j}^{k_1,l_2}(\mu\xi\nu, \mu\xi\nu) \\
\intertext{and}
&  \iota_-^{k_2-k_1}(E_{i_2}^{k_2,l_2}(\nu, \nu)) \iota_+^{l_2-l_1}(E_{i_1}^{k_1,l_1}(\mu, \mu)) \\
= & \sum_{\xi\in B_{k_2-l_1}(\Lambda)}
\sum_{j=1}^{m(k_1,l_2)} 
 A^-_{n(k_2,l_2), n(k_1,l_2)}(i_2,\mu\xi, j) A^+_{n(k_1,l_1), n(k_1,l_2)}(i_1,\xi\nu, j)
E_{j}^{k_1,l_2}(\mu\xi\nu, \mu\xi\nu),
\end{align*}
so that 
$\iota_+^{l_2-l_1}(E_{i_1}^{k_1,l_1}(\mu, \mu))
$ commutes to
$
\iota_-^{k_2-k_1}(E_{i_2}^{k_2,l_2}(\nu, \nu)).
$
\end{proof}
Define
$\rho_{\frak L}: \F_{\frak L}^{k,l} \longrightarrow 
\F_{\frak L}^{k+1,l+1}$
by setting
\begin{equation}
\rho_{\frak L}(E_i^{k,l}(\mu,\nu) = E_i^{k+1,l+1}(\mu,\nu) 
\quad \text{ for }
\mu,\nu \in W_i^{n(k,l)}.
\end{equation}
As 
$W_i^{n(k,l)} = W_i^{n(k+1,l+1)}$,
the partial isometry 
$E_i^{k+1,l+1}(\mu,\nu) 
$ is defined  for
$\mu,\nu \in W_i^{n(k,l)}.$
Since 
$\rho_{\frak L} \circ \iota_+ = \iota_+\circ \rho_{\frak L}$
and
$\rho_{\frak L} \circ \iota_- = \iota_-\circ \rho_{\frak L}$,
it is easy to see that 
$\rho_{\frak L}: \F_{\frak L} \longrightarrow \F_{\frak L}$
gives rise rise to an automorphism
on the $C^*$-algebra $\F_{\frak L}$, that will be written as $\rho_{\frak L}.$
For $k,l\in \Z$ with $k<l$ and $\mu \in B_{n(k,l)}(\Lambda)$, 
we put the projection
\begin{equation*}
E^{k,l}(\mu) := \sum_{i=1}^{m(k,l)} E_i^{k,l}(\mu,\mu).
\end{equation*}

Let us denote by
${\mathcal{D}}_{\frak L}$
the diagonal algebra of the AF-algebra
$\F_{\frak L}$, that is generated by the projections of the form
$E_i^{k,l}(\mu,\mu)$ for $\mu \in W_i^{n(k,l)}, i=1,2,\dots, m(k,l),\, k,l \in \Z$ with
$k<l$.
Let us denote by $C(\Lambda)$ 
the commutative $C^*$-algebra of all complex valued continuous functions on $\Lambda$.
\begin{proposition}\label{prop:commutative}
There exists an embedding 
$\iota_\Lambda: C(\Lambda) \hookrightarrow {\mathcal{D}}_{\frak L}$ 
of the commutative $C^*$-algebra
$C(\Lambda)$  into the diagonal algebra ${\mathcal{D}}_{\frak L}$
such that 
$\iota_\Lambda\circ\sigma_\Lambda^* = \rho_{\frak L}\circ\iota_\Lambda$,
where
$\sigma_\Lambda^*: C(\Lambda) \longrightarrow C(\Lambda)$
is the automorphism defined by $\sigma_\Lambda^*(f) = f\circ \sigma_\Lambda^{-1}$
for $f \in C(\Lambda)$. 
\end{proposition}
\begin{proof}
For $k, l\in \Z$ 
with $k<l$ and an admissible word 
$\mu=(\mu_1,\dots, \mu_{n(k,l)}) \in B_{n(k,l)}(\Lambda)$, 
we denote by 
$U^{k,l}_\mu$ the cylinder set on $\Lambda$ 
$$
U^{k,l}_\mu =\{ (x_n)_{n\in \Z} \in \Lambda \mid x_{k+1} =\mu_1,x_{k+2} =\mu_2,
\dots, x_{l-1} =\mu_{n(k,l)} \}.
$$
Let $\chi_{U^{k,l}_\mu}$ be the characteristic function on $\Lambda$.
It is easy to see that the  correspondence
\begin{equation*}
\chi_{U^{k,l}_\mu} \in C(\Lambda) \longrightarrow E^{k,l}(\mu) \in 
{\mathcal{D}}_{\frak L}
\end{equation*}
yields an injective homomorphism
$\iota_\Lambda: C(\Lambda) \hookrightarrow {\mathcal{D}}_{\frak L}$ 
satisfying  
$\iota_\Lambda\circ\sigma_\Lambda^* = \rho_{\frak L}\circ\iota_\Lambda$.
\end{proof}

We will now define the $C^*$-algebra $\R_{\frak L}$
associated with $\lambda$-graph bisystem
satisfying FPCC.
\begin{definition}
$\R_{\frak L} :=\F_{\frak L}\rtimes_{\rho_{\frak L}}\Z$
\end{definition}
If, in particular,  
$({\frak L}_\Lambda^-, {\frak L}_\Lambda^+)$
is the canonical $\lambda$-graph bisystem for a subshift 
$\Lambda$, the $C^*$-algebra
 $\R_{{\frak L}_\Lambda}$ is called the $C^*$-{\it algebra
associated with two-sided subshift}\/ $\Lambda$,
and written $\R_\Lambda$.

We remark that a subshift can not be any Smale space unless it is a shift of finite type
(cf. \cite{Putnam1}, \cite{Putnam2}, \cite{Ruelle2}).
The above definition of the $C^*$-algebras associated with two-sided subshifts is a 
generalization of the asymptotic  Ruelle algebras for shifts of finite type to general subshifts.
%%%%%%%%%%%%%%%%

\bigskip

%%%%%%%%%%%%%%%%%%%%%%%%%%%%%%%%%%%%%%%%%%%%%%%%%%%
%%%%%%%%%%%%%%%%%%%%%%%%%%%%%%%%%%%%%%%%%%%%%%%%
%\section{Simplicity condition}
%%%%%%%%%%%%%%%%%%%%%%%%%%%%%%%%%%%%%%%%%%
We will next give a condition on $\LGBS$ under which the $C^*$-algebra 
$\R_{\frak L}$ becomes simple.
In what follows, we assume  that a $\lambda$-graph bisystem $\LGBS$ satisfies FPCC.
Let us denote by $\Lambda$ the presented subshift by $\LGBS$.
We provide a lemma similar to Lemma \ref{lem:commute}.

\begin{lemma}\label{lem:comm2}
For $k_1, l_1,k_2, l_2\in \Z$ with 
$k_1 < l_1 \le k_2 < l_2$ 
and
$i_2 =1,2, \dots, m(n(k_2,l_2))$,
the projection
$
\iota_+^{l_2-l_1}(E^{k_1,l_1}(\mu ))
$
commutes to the partial isometry
$
\iota_-^{k_2-k_1}(E_{i_2}^{k_2,l_2}(\mu', \nu'))
$
for $\mu \in B_{n(k_1,l_1)}(\Lambda), \, \mu', \nu' \in W_{i_2}^{n(k_2, l_2)}
$ 
in $\F_{\frak L}^{k_1, l_2}.$
Similarly, 
$\iota_-^{k_2-k_1}(E^{k_2,l_2}(\mu))$ 
commutes to $\iota_+^{l_2-l_1}(E_{i_1}^{k_1,l_1}(\mu', \nu'))$
for $\mu \in B_{n(k_2, l_2)}(\Lambda)$ and 
$\mu', \nu' \in W_{i_1}^{n(k_1,l_1)}$
in $\F_{\frak L}^{k_1, l_2}.$
 \end{lemma}
\begin{proof}
As in the proof of
Lemma \ref{lem:commute},
we have
\begin{align*}
& \iota_+^{l_2-l_1}(E^{k_1,l_1}(\mu)) \iota_-^{k_2-k_1}(E_{i_2}^{k_2,l_2}(\mu', \nu'))\\
= & \sum_{\gamma\in B_{l_2-l_1}(\Lambda)}
\sum_{j_1=1}^{m(k_1,l_2)} \sum_{i_1=1}^{m(k_1,l_1)}
 A^+_{n(k_1,l_1), n(k_1,l_2)}(i_1,\gamma, j_1)
E_{j_1}^{k_1,l_2}(\mu\gamma, \mu\gamma) \\
 &\cdot  \sum_{\eta\in B_{k_2-k_1}(\Lambda)}
\sum_{j_2=1}^{m(k_1,l_2)} A^-_{n(k_2,l_2), n(k_1,l_2)}(i_2,\eta, j_2)
E_{j_2}^{k_1,l_2}(\eta\mu',\eta\nu') \\ 
= & \sum_{\xi\in B_{k_2-l_1}(\Lambda)}
\sum_{j=1}^{m(k_1,l_2)} \sum_{i_1=1}^{m(k_1,l_1)}
A^+_{n(k_1,l_1), n(k_1,l_2)}(i_1,\xi\mu', j) A^-_{n(k_2,l_2), n(k_1,l_2)}(i_2,\mu\xi, j)
E_{j}^{k_1,l_2}(\mu\xi\mu', \mu\xi\nu'),
\end{align*}
and similarly we have
\begin{align*}
& \iota_-^{k_2-k_1}(E_{i_2}^{k_2,l_2}(\mu', \nu')) \iota_+^{l_2-l_1}(E^{k_1,l_1}(\mu)) \\
= &  \sum_{\eta\in B_{k_2-k_1}(\Lambda)}
\sum_{j_2=1}^{m(k_1,l_2)} A^-_{n(k_2,l_2), n(k_1,l_2)}(i_2,\eta, j_2)
E_{j_2}^{k_1,l_2}(\eta\mu',\eta\nu') \\ 
 &\cdot  \sum_{\gamma\in B_{l_2-l_1}(\Lambda)}
\sum_{j_1=1}^{m(k_1,l_2)} \sum_{i'_1=1}^{m(k_1,l_1)}
 A^+_{n(k_1,l_1), n(k_1,l_2)}(i'_1,\gamma, j_1)
E_{j_1}^{k_1,l_2}(\mu\gamma, \mu\gamma) \\
= & \sum_{\xi\in B_{k_2-l_1}(\Lambda)}
\sum_{j=1}^{m(k_1,l_2)} \sum_{i'_1=1}^{m(k_1,l_1)}
A^+_{n(k_1,l_1), n(k_1,l_2)}(i'_1,\xi\nu', j) A^-_{n(k_2,l_2), n(k_1,l_2)}(i_2,\mu\xi, j)
E_{j}^{k_1,l_2}(\mu\xi\mu', \mu\xi\nu').
\end{align*}
Now suppose that 
$$
A^+_{n(k_1,l_1), n(k_1,l_2)}(i_1,\xi\mu', j) A^-_{n(k_2,l_2), n(k_1,l_2)}(i_2,\mu\xi, j)=1.
$$
Since $\mu', \nu' \in W_{i_2}^{n(k_2,l_2)}$ and 
$A^-_{n(k_2,l_2), n(k_1,l_2)}(i_2,\mu\xi, j)=1$,
we have
$\mu \xi \nu' \in F(v_j^{n(k_1,l_2)})$ in ${\frak L}^-$
and hence 
$\mu \xi \nu' \in P(v_j^{n(k_1,l_2)})$ in ${\frak L}^+$
because $\LGBS$ satisfies FPCC.
Hence there exists $i'_1 \in V_{n(k_1,l_1)}$ such that 
$A^+_{n(k_1,l_1), n(k_1,l_2)}(i'_1,\xi\mu', j)=1$,
so that 
$$
A^+_{n(k_1,l_1), n(k_1,l_2)}(i'_1,\xi\nu', j) A^-_{n(k_2,l_2), n(k_1,l_2)}(i_2,\mu\xi, j)=1.
$$
Such $i'_1$ satisfying
$
A^+_{n(k_1,l_1), n(k_1,l_2)}(i'_1,\xi\nu', j)=1
$
 is uniquely determined by the word $\xi \nu'$ and the vertex $v_j^{n(k_1,l_2)}$, because ${\frak L}^+$ is left-resolving.  

Conversely, 
suppose that 
$$
A^+_{n(k_1,l_1), n(k_1,l_2)}(i'_1,\xi\nu', j) A^-_{n(k_2,l_2), n(k_1,l_2)}(i_2,\mu\xi, j) =1.
$$
By a similar argument to the above one, 
we know that there exists a unique vertex 
$v_{i_1}^{n(k_1,l_1)} \in V_{n(k_1,l_1)}$  such that 
$A^+_{n(k_1,l_1), n(k_1,l_2)}(i_1,\xi\nu', j) =1$,
and hence 
$$
A^+_{n(k_1,l_1), n(k_1,l_2)}(i_1,\xi\nu', j) A^-_{n(k_2,l_2), n(k_1,l_2)}(i_2,\mu\xi, j) =1.
$$
Therefore the equality
\begin{align*}
& \sum_{i_1=1}^{m(k_1,l_1)}
   A^+_{n(k_1,l_1), n(k_1,l_2)}(i_1,\xi\mu', j) A^-_{n(k_2,l_2), n(k_1,l_2)}(i_2,\mu\xi, j) \\
=&
\sum_{i'_1=1}^{m(k_1,l_1)}
A^+_{n(k_1,l_1), n(k_1,l_2)}(i_1,\xi\nu', j) A^-_{n(k_2,l_2), n(k_1,l_2)}(i_2,\mu\xi, j)
\end{align*}
holds.
We thus  see that 
$\iota_+^{l_2-l_1}(E^{k_1,l_1}(\mu))
$ commutes to
$
\iota_-^{k_2-k_1}(E_{i_2}^{k_2,l_2}(\mu', \nu')).
$

Similarly we may prove that 
$\iota_{-}^{k_2-k_1}(E^{k_2,l_2}(\mu))$ 
commutes to $\iota_{+}^{k_2-l_1}(E_{i_1}^{k_1,l_1}(\mu', \nu'))$
for $\mu \in B_{n(k_2, l_2)}(\Lambda)$ and $\mu', \nu' \in W_{i_1}^{n(k_1,l_1)}$.
\end{proof}

For $k,K\in \N$ with $k<K$,
a word $\mu =(\mu_1, \dots,\mu_K) \in B_K(\Lambda)$ 
is said to be $k$-{\it aperiodic}\/ if the words with length $K-k$
\begin{align*}
\mu_{[1,K-k]} & =(\mu_1,\mu_2,\dots,\mu_{K-k}), \quad \\
\mu_{[2,K-k+1]}&  =(\mu_2,\mu_3,\dots,\mu_{K-k+1}), \quad \\
& \dots, \quad \\ 
\mu_{[1+k,K]} & =(\mu_{1+k},\mu_{2+k},\dots,\mu_{K})
\end{align*}
are all distinct.
\begin{definition}\label{def:conditionI}
A $\lambda$-graph bisystem $\LGBS$ satisfying FPCC
is said to satisfy {\it condition}\/ $(I^+)$ (resp. $(I^-)$)
if  for $n, k \in \N$  
there exists $K\in \N$ with $k <K$ such  that 
there exists a $k$-aperiodic word $\mu(i) \in B_K(\Lambda)$ 
for each $i=1,2,\dots,m(n)$
such that $\mu(i)$ starts (resp. ends) with $v_i^n \in V_n$ in 
${\frak L}^+$ (resp. ${\frak L}^-$). 
If $\LGBS$ satisfies at least one of condition $(I^+)$ or $(I^-)$,
it is said to satisfy {\it condition}\/ (I). 
\end{definition}
Let $M$ be an $N \times N$ square matrix with entries in $\{0,1\}$.
It is easy to see that the $\lambda$-graph bisystem 
$({\frak L}_M^-, {\frak L}_M^+)$ defined in Example \ref{ex:3.2} (iv)
satisfies condition (I) in the above sense if and only if
the matrix $M$ satisfies condition (I) in the sense of Cuntz--Krieger \cite{CK}.  

Let us denote by
$\F_{\frak L}^l$ the $C^*$-subalgebra  
$\F_{\frak L}^{-l,l}$
of the AF-algebra $\F_{\frak L}$ for $l \in \Zp$. 
Let us identify the commutative $C^*$-subalgebra
generated by the projections $E^{k,l}(\mu), \, k,l\in Z, k<l,\, \mu\in B_{n(k,l)}(\Lambda)$
 with the commutative $C^*$-algebra
$C(\Lambda)$ by Proposition \ref{prop:commutative}.
\begin{lemma}
Assume that $\LGBS$ satisfies condition $(I)$.
For $k,l\in \N$ with $k\le l$, there exists a projection
$q_k^l \in C(\Lambda) \cap {\F_{\frak L}^l}'
=\{ x \in C(\Lambda) \mid a x = x a \text{ for all } a \in \F_{\frak L}^l\}$
such that 
\begin{enumerate}
\renewcommand{\theenumi}{\arabic{enumi}}
\renewcommand{\labelenumi}{\textup{(\theenumi)}}
\item $q_k^l a \ne 0$ for all $a \in \F_{\frak L}^l$ with $a \ne 0$.
\item $q_k^l \rho_{\frak L}^m(q_k^l) =0$ for $m=1,2,\dots,k$.
\end{enumerate}
\end{lemma} 
\begin{proof}
Assume that  $\LGBS$ satisfies condition $(I^+)$.
For $k,l\in \N$ with $k\le l$, let $n = n(-l,l) = 2l-1$.
For the $n, k\in \N$,  the condition $(I^+)$ ensures us that 
there exists $K >k$ such that 
there exists a $k$-aperiodic word $\mu(i) \in B_k(\Lambda)$ 
for each $i=1,2,\dots,m(n)$
such that $\mu(i)$ starts with $v_i^n$ in 
${\frak L}^+$. 
We set a projection
\begin{equation}
q_k^l = \sum_{i=1}^{m(n)} E^{n,n+K+1}(\mu(i)) \in C(\Lambda). \label{eq:qkl}
\end{equation}
Let $-l < l < n < n+K+1$  correspond to 
$k_1 < l_1 < k_2 < l_2$ in Lemma \ref{lem:commute}, respectively.
We know that 
$q_k^l$ commutes to $\F_{\frak L}^l$.
As the words $\mu(i)$ starts at $v_i^n$ for all $i=1,2,\dots,m(n)$,
we have $q_k^l a \ne 0$ for all $a \in \F_{\frak L}^l$ with $a \ne 0$.
As the words $\mu(i)$ are $k$-aperiodic, we have
$q_k^l \rho_{\frak L}^m(q_k^l) =0$ for $m=1,2,\dots,k$.

In case  $\LGBS$ satisfies condition $(I^-)$,
we easily see that the projection 
$q_k^l$ defined by \eqref{eq:qkl}
satisfies 
$q_k^l \rho_{\frak L}^m(q_k^l) =0$ for $m=-1,-2,\dots,-k$.
We then have
\begin{equation*}
q_k^l \rho_{\frak L}^{|m|}(q_k^l) 
%= \rho_{\frak L}^{|m|}(\rho_{\frak L}^m(q_k^l)q_k^l )
= \rho_{\frak L}^{|m|}(q_k^l \rho_{\frak L}^m(q_k^l))
 =0.
\end{equation*}
Therefore the projection $q_k^l$ satisfies the desired properties.
\end{proof}
Let $U$ be the generating unitary in 
$\R_{\frak L} =\F_{\frak L}\rtimes_{\rho_{\frak L}}\Z$ 
corresponding to the positive generator of $\Z$ satisfying the relation
\begin{equation}
\rho_{\frak L}(x)  = U  x U^*, \qquad x \in \F_{\frak L}.
\end{equation}
The following proposition shows a uniqueness of the crossed product 
$C^*$-algebra $\F_{\frak L}\rtimes_{\rho_{\frak L}}\Z$
under condition (I).
\begin{proposition} \label{prop:unicity}
Assume that $\LGBS$ satisfies condition $(I)$.
If there exists an injective homomorphism
$\pi : \F_{\frak L} \longrightarrow \mathcal{B}$
from $\F_{\frak L}$ to a $C^*$-algebra $\mathcal{B}$
and a unitary $u$ in $\mathcal{B}$ satisfying the relation
\begin{equation}
\pi (\rho_{\frak L}(x))  = u \pi(x) u^*, \qquad x \in \F_{\frak L},
\end{equation}
then  the correspondence
\begin{equation*}
x \in \F_{\frak L} \longrightarrow \pi(x) \in \mathcal{B},
\qquad
U \longrightarrow u \in \mathcal{B}
\end{equation*}
extends to an isomorphism from $\F_{\frak L}\rtimes_{\rho_{\frak L}}\Z$
onto the $C^*$-subalgebra of $\mathcal{B}$ generated by 
$\pi(x), x \in \F_{\frak L}$ and $u$. 
%Then any nontrivial ideal $\mathcal{I}$ of the crossed product 
%$\R_{\frak L} =\F_{\frak L}\rtimes_{\sigma_{\frak L}}\Z$ 
%has nontrivial intersection with the AF-algebra $\F_{\frak L}$.
\end{proposition}
\begin{proof}
Assume that $\LGBS$ satisfies condition $(I)$.
In general, a pair $(\A,\rho)$ of a unital $C^*$-algebra $\A$ and an automorphism $\rho$ of $\A$
is  an example of a $C^*$-symbolic dynamical system introduced in 
\cite{MaCrelle} (cf. \cite{MaCM2009}, \cite{MaMZ2010}).
 By the previous lemma, the pair $(\F_{\frak L}, \rho_{\frak L})$
satisfies condition $(I)$ as a $C^*$-symbolic dynamical system 
in the sense of \cite[Section 3]{MaMZ2010},
so that by \cite[Theorem 3.9]{MaMZ2010}, we know the desired assertion.
\end{proof}
Therefore we have the following simplicity criterion for the $C^*$-algebra  
$\R_{\frak L}$.
\begin{theorem}\label{thm:simplicity}
Assume that $\LGBS$ satisfies condition $(I)$.
If $\LGBS$ is irreducible,
the $C^*$-algebra $\R_{\frak L}$ is simple. 
\end{theorem}
\begin{proof}
 Suppose that $\LGBS$ satisfies condition (I).
By Proposition \ref{prop:unicity},
any nontrivial ideal $\mathcal{I}$ of the crossed product 
$\R_{\frak L} =\F_{\frak L}\rtimes_{\rho_{\frak L}}\Z$ 
has nontrivial intersection with the AF-algebra $\F_{\frak L}$.
As the AF-algebra $\F_{\frak L}$ is simple by Proposition \ref{prop:simplicity},
 we may conclude that 
the crossed product $\F_{\frak L}\rtimes_{\rho_{\frak L}} \Z$
is simple.
 \end{proof}

%%%%%%%%%%%%%%%%%%%%%%%%%%%%%
%%%%%%%%%%%%%%%%%%%%%%
\section{Invariance under topological conjugacy}
%%%%%%%%%%%%%%%%%%%%%%%%%%%%%%%%%%%%%%%%%%%%%%%%%%%%%%%%%%%%%%%%%%%%%%%%
%%%%%%%%%%%%%%%%%%%%%%%%%%%%%%%%%%%%%%%%%%%%%%%%
In this section, we will prove that the pair $(\F_{\frak L},\rho_{\frak L})$
for a $\lambda$-graph bisystem $\LGBS$ satisfying FPCC 
is invariant under properly strong shift equivalence 
of the associated symbolic matrix bisystems, so that 
 the crossed product $\R_{\frak L}$ is also invariant
under the equivalence relation in the symbolic matrix bisystems.
As a corollary, if two subshifts $\Lambda_1, \Lambda_2$ 
are topologically conjugate, then
the pairs $(\F_{{\frak L}_{\Lambda_1}},\rho_{{\frak L}_{\Lambda_1}})$
and $(\F_{{\frak L}_{\Lambda_2}},\rho_{{\frak L}_{\Lambda_2}})$
for their canonical $\lambda$-graph bisystems 
${\frak L}_{\Lambda_1}, {\frak L}_{\Lambda_2}$
are invariant.
Hence the K-groups
$(K_0(\F_{\frak L}), K_0(\F_{\frak L})_+, \rho_{{\frak L}*})$
and the $C^*$-algebra 
$\R_{\frak L}$ yield topological conjugacy invariants for subshifts.
In order to prove such invariance, we recall the notion of
symbolic matrix bisystem
 and its properly strong shift equivalence introduced in \cite{MaPre2019}. 
We first recall the definition of symbolic matrix bisystem,
that is a matrix presentation of a $\lambda$-graph bisystem.

Let $\Sigma$ be a finite alphabet and 
${\frak S}_\Sigma$ be the set of finite formal sums of elements of $\Sigma.$
 By a symbolic matrix $\A$ over $\Sigma$
 we mean a rectangular finite matrix $\A = [\A(i,j)]_{i,j}$
 whose entries in ${\frak S}_\Sigma.$ 
We write the empty word $\emptyset$ as $0$ in ${\frak S}_\Sigma.$
For the symbolic matrix $\A,$
if $\A(i,j) = \alpha_1+\cdots +\alpha_n$,
then  we write an edge $e_k$ labeled $\alpha_k$
for $k=1,\dots,n$ from a vertex $v_i$ to a vertex $v_j$,
so that we have a finite labeled directed graph.

For two alphabets $\Sigma, \Sigma',$
the notation $\Sigma\cdot\Sigma'$ denotes the set
$\{ a\cdot b \mid a \in \Sigma, \, b \in \Sigma' \}.$
The following notion of specified equivalence 
between symbolic matrices due to M. Nasu in
\cite{Nasu}, \cite{NasuMemoir}.
For two symbolic matrices
$\A$ over alphabet $\Sigma$
 and  $\A'$ over alphabet $\Sigma'$
 and  a bijection
 $\varphi$ from a subset of $\Sigma$ onto a subset of $\Sigma'$,
 we call $\A$ and $\A'$ are  specified equivalence under specification
 $\varphi$ if $\A'$ can be obtained from $\A$
  by replacing every symbol 
 $\alpha$ appearing in components of $\A$ by $\varphi(\alpha)$.
 We write it as
 $\A \overset{\varphi}{\simeq} \A'$, and call 
$\varphi$ a specification from $\Sigma$ to $\Sigma'$.
For two alphabet $\Sigma_1, \Sigma_2,$ 
the bijection 
$\alpha\cdot\beta\in \Sigma_1\cdot\Sigma_2 \longrightarrow 
\beta\cdot\alpha\in \Sigma_2\cdot\Sigma_1$
naturally yields a bijection from 
${\frak S}_{\Sigma_1\cdot\Sigma_2}$
to 
${\frak S}_{\Sigma_2\cdot\Sigma_1}$
that we denote by $\kappa$ 
and call the exchanging specification  
between $\Sigma_1$ and $\Sigma_2.$
\begin{definition}[{\cite{MaPre2019}}]
%Let $\M_{l,l+1}^-$ (resp. $\M_{l,l+1}^+$) be a rectangular symbolic matrix over alphabet 
%$\Sigma^-$ (resp. $\Sigma^+$).  
{\it A symbolic matrix bisystem}\/
$(\M_{l,l+1}^-,\M_{l,l+1}^+), l\in \Zp$ is a pair of
sequences of rectangular symbolic matrices
$\M_{l,l+1}^-$ over $\Sigma^-$ and
$\M_{l,l+1}^+$ over $\Sigma^+$  
 satisfying  
the following five conditions:
\begin{enumerate}
\renewcommand{\theenumi}{\roman{enumi}}
\renewcommand{\labelenumi}{\textup{(\theenumi)}}
\item
Both $\M_{l,l+1}^-$ and $\M_{l,l+1}^+$ are $m(l)\times m(l+1)$ 
rectangular  symbolic matrices such that $m(l) \le m(l+1)$ for $l \in \Zp.$
\item
(1) For $i,$ there exists $j$ such that $\M_{l,l+1}^-(i,j)\ne 0$, 
and for $i,$ there exists $j$ such that $\M_{l,l+1}^+(i,j)\ne 0$.

(2) For $j,$ there exists $i$ such that $\M_{l,l+1}^-(i,j)\ne 0$, 
and for $j,$ there exists $i$ such that $\M_{l,l+1}^+(i,j)\ne 0$.

%\hspace{6mm} For $j,$ there exists $i$ such that $\M_{l,l+1}^+(i,j)\ne 0$.

\item Each column of both $\M_{l,l+1}^-$ and $\M_{l,l+1}^+$ 
does not have multiple labeling.
This means that 
if $\M_{l,l+1}^-(i,j) = \alpha_1+\cdots+\alpha_n,$
then the symbols $\alpha_1,\dots,\alpha_n$ do not appear in any other rows of the $j$th column.
The same  condition is required for  $\M_{l,l+1}^+.$
\item
 %$\M_{l-1,l}^-$ is right-resolving symbolic matrix and 
%$ \M_{l,l+1}^+$ is left-resolving symbolic matrix: that mean, 
Both the $j$th columns 
$[\M_{l,l+1}^-(i,j)]_{i=1}^{m(l)}$
and
$[\M_{l,l+1}^+(i,j)]_{i=1}^{m(l)}$
do not have multiple symbols for each $j=1,2,\dots,m(l+1)$.
Namely, if a symbol $\alpha$ appears in $\M_{l,l+1}^-(i,j)$ for some 
$i\in \{1,2,\dots, m(l)\},$
then it does not appear in  any other row $\M_{l,l+1}^-(i',j)$ for $i' \ne i,$
and $\M_{l,l+1}^+$ has the same property.
\item The specified equivalence
 $\M_{l,l+1}^- \M_{l+1,l+2}^+ \overset{\kappa}{\simeq}
  \M_{l,l+1}^+ \M_{l+1,l+2}^-$ 
 for $l \in \Zp$ holds,
that means for $i=1,2,\dots,m(l),\, j=1,2,\dots,m(l+2)$
\begin{equation}
\sum_{k=1}^{m(l+1)}\M_{l,l+1}^-(i,k) \M_{l+1,l+2}^+(k,j)
 =\sum_{k=1}^{m(l+1)} \kappa\left( \M_{l,l+1}^+(i,k)\M_{l+1,l+2}^-(k,j) \right)
\label{eq:MLocal}
\end{equation}
holds, where $\kappa$ is the exchanging specification between 
$\Sigma$ and $\Sigma'.$
\end{enumerate}
%The matrix $\M_{l,l+1}^-$ (resp.  $\M_{l,l+1}^+$) satisfying
%the condition (iv) is said to be right-resolving (resp. left-resolving). 
The condition (v) exactly expresses the local property of $\lambda$-graph bisystem (v)
in Definition \ref{def:lambdabisystem}.
The pair $(\M^-,\M^+)$ is called a symbolic matrix bisystem over $\Sigma^\pm.$
\end{definition}
It is easy to see that symbolic matrix bisystem is exactly 
a matrix presentation of $\lambda$-graph bisystem.
A symbolic matrix bisystem $(\M^-, \M^+)$ is said to be {\it standard}\/
if  $m(0) =1$, that is its row sizes of the matrices $\M^-_{0,1}$ and $\M^+_{0,1}$ are one.
A symbolic matrix bisystem  $(\M^-, \M^+)$ is said to {\it have a common alphabet}\/
if  $\Sigma^- = \Sigma^+.$
In this case, write the alphabet $\Sigma^- = \Sigma^+$
as $\Sigma,$
and say that $(\M^-, \M^+)$ is a symbolic matrix bisystem over common alphabet
 $\Sigma.$
It s said to satisfy
{\it Follower-Predecessor Compatibility Condition}\/, FPCC for brevity, 
if for every $l \in \N$ and $j = 1,2,\dots,m(l),$
the set of words  appearing in $[\M_{0,1}^-\M_{1,2}^-\cdots\M^-_{l-1,l}](i,j)$
coincides with the set of transposed words appearing in 
$[\M_{0,1}^+\M_{1,2}^+\cdots\M^+_{l-1,l}](i,j).$
Hence we may recognize that symbolic matrix bisystems satisfying FPCC
and $\lambda$-graph bisystems satisfying FPCC are the same objects
(see \cite{MaPre2019} for detail).

%%%%%%%%%%%%%%%%%%%%%%%%%%%%%%%%%%%%%%%%%%%%%%%%%%%%
%Two symbolic matrix bisystems $(\M^-,\M^+)$ over  $\Sigma_\M^\pm$
%and $(\SN^-, \M^+)$ over  $\Sigma_\M^\pm$
%are said to be isomorphic if the sizes $m(l)\times m(l+1)$  and $n(l) \times n(l+1)$ 
%of the matrices $\M^{\pm}_{l,l+1}$ and $\SN^\pm_{l,l+1}$ 
%coincide, that is $m(l) = n(l)$,   for each $l \in \Zp$ and
%there exists a specification $\phi$ from $\Sigma_\M$ to $\Sigma_\M$ 
%and an $m(l) \times m(l)$-square permutation matrix
%$P_l$ for each $l \in \Zp$ such that
%$$P_l \M_{l,l+1}^- \overset{\phi}{\simeq} \SN_{l,l+1}^-P_{l+1},\qquad
%P_l \M_{l,l+1}^+ \overset{\phi}{\simeq} \SN_{l,l+1}^+P_{l+1} \qquad
%\text{ for }\quad l \in \Zp.$$
%%%%%%%%%%%%%%%%%%%%%%%%%%%%%%%%%%%%%%%%%%%%%%%%%%

As seen in Section 2, any subshift is presented by a $\lambda$-graph bisystem
satisfying FPCC,
and hence by a symbolic matrix bisystem satisfying FPCC.
In \cite{MaPre2019}, 
we introduced a notion of properly strong shift equivalence in 
symbolic matrix bisystems satisfying FPCC, and proved that 
 two subshifts are topologically conjugate 
if and only if their canonical symbolic matrix bisystems are properly strong shift equivalent.

Let $(\M_1^-,\M_1^+)$ and $(\M_2^{-}, \M_2^{+})$ be symbolic matrix bisystems over alphabets
$\Sigma_1$ and $\Sigma_2$, respectively,
both of them satisfy FPCC,
where
${\M^-_1}_{l,l+1}, {\M^+_1}_{l,l+1}$ are $m_1(l)\times m_1(l+1)$ matrices and
${\M^-_2}_{l,l+1}, {\M^+_2}_{l,l+1}$ are $m_2(l)\times m_2(l+1)$ matrices.

\begin{definition}[{\cite{MaPre2019}}] \label{def:PSSE}
Two symbolic matrix bisystems 
 $(\M_1^-,\M_1^+)$ and $(\M_2^-, \M_2^+)$
 are said to be {\it properly strong shift equivalent in}\/ $1$-{\it step}\/  
 if there exist alphabets 
$C,D$ and specifications 
$$
 \varphi_1: \Sigma_1 \rightarrow C\cdot D,
 \qquad
 \varphi_2: \Sigma_2 \rightarrow D \cdot C
$$
 and  sequences $c(l),d(l)$ on $l \in \Zp$
 such that  there exist for each $l\in \Zp$
\begin{enumerate}
\renewcommand{\theenumi}{\arabic{enumi}}
\renewcommand{\labelenumi}{\textup{(\theenumi)}}
\item a $c(l)\times d(l+1)$ matrix ${\P}_l$ over $C,$ 
\item a $d(l)\times c(l+1)$ matrix ${\Q}_l$ over $D,$
\item a $d(2l+1)\times d(2l+2)$ matrix $\X_{2l+1} $ over $D$,
\item a $c(2l)\times c(2l+1)$ matrix $\X_{2l} $ over $D$,
\item a $c(2l+1)\times c(2l+2)$ matrix $\Y_{2l+1} $ over $C$,
\item a $d(2l)\times d(2l+2)$ matrix $\Y_{2l} $ over $C$
\end{enumerate}
 satisfying the following equations: 
\begin{gather}
{\M_1^+}_{l,l+1} 
\overset{\varphi_1}{\simeq} {\P}_{2l}{\Q}_{2l+1},\qquad
{\M^+_2}_{l,l+1} 
\overset{\varphi_2}{\simeq} {\Q}_{2l}{\P}_{2l+1}, \label{eq:PSSE1} \\
{\M^-_1}_{l,l+1} \overset{\kappa\varphi_1}{\simeq} \X_{2l}\Y_{2l+1}, \qquad 
{\M^-_2}_{l,l+1} \overset{\kappa\varphi_2}{\simeq} \Y_{2l} \X_{2l+1} \\ 
\intertext{and}
\Y_{2l+1}{\P}_{2l+2} \overset{\kappa}{\simeq} {\P}_{2l+1}\Y_{2l+2},\qquad
\X_{2l+1}{\Q}_{2l+2}\overset{\kappa}{\simeq} {\Q}_{2l+1}\X_{2l+2}, \label{eq:PSSE3}\\ 
\X_{2l}{\P}_{2l+1} \overset{\kappa}{\simeq} {\P}_{2l}\X_{2l+1},\qquad
\Y_{2l}{\Q}_{2l+1}\overset{\kappa}{\simeq} {\Q}_{2l}\Y_{2l+1}, \label{eq:PSSE4}
\end{gather}
where $\kappa$ is the exchanging specification defined by
$\kappa(a\cdot b) = b\cdot a.$

\medskip
We write this situation as
$
(\M_1^-,\M_1^+)\underset{1-\pr}{\approx} (\M_2^-, \M_2^+).
$
Two symbolic matrix bisystems 
 $(\M_1^-,\M_1^+)$ and $(\M_2^-,\M_2^+)$ 
 are said to be {\it properly  strong shift equivalent in }\/ $\ell$-{\it step}\/ 
 if there exists a sequence of  symbolic matrix bisystems
$(\M_{(i)}^-,\M_{(i)}^+),\, i=1,2,\dots, \ell-1$
such that
\begin{equation*}
(\M_1^-,\M_1^+) \underset{1-\pr}{\approx} (\M_{(1)}^-,\M_{(1)}^+)
       \underset{1-\pr}{\approx} 
       \cdots 
       \underset{1-\pr}{\approx} 
(\M_{(\ell-1)}^-,\M_{(\ell-1)}^+) \underset{1-\pr}{\approx} 
( \M_2^-, \M_2^+).
\end{equation*}
We write this situation as
$
(\M_1^-, \M_1^+) \underset{\ell-\pr}{\approx} (\M_2^-,\M_2^+)
$
and simply call it a {\it properly strong shift equivalence}.
\end{definition}
Properly strong shift equivalence in symbolic matrix bisystems is an equivalence relation
(\cite[Proosition 6.2]{MaPre2019}).
For a subshift $\Lambda$,
the symbolic matrix bisystem
$(\M_\Lambda^-,\M_\Lambda^+)$
associated to the canonical $\lambda$-graph bisystem
$({\frak L}^-,{\frak L}^+)$ for $\Lambda$
is called the canonical symbolic matrix bisystem for $\Lambda$.
We have proved the following theorem.
\begin{theorem}[{\cite[Theorem 6.3]{MaPre2019}}] \label{thm:topconj}
Two  subshifts are topologically conjugate
if and only if  
their canonical symbolic matrix bisystems
are properly strong shift equivalent.
\end{theorem}
Thanks to Theorem \ref{thm:topconj},
the following theorem will show that 
the automorphism $\rho_{{\frak L}_{\Lambda}}$
on the AF-algebra $\F_{{\frak L}_\Lambda}$ 
is invariant under topological conjugacy of subshifts.

\medskip

%%%%%%%%%%%%%%%%%%%%%%%%%

\begin{theorem}\label{thm:psseC}
Let 
$({\frak L}_1^-, {\frak L}_1^+)$
and
$({\frak L}_2^-, {\frak L}_2^+)$ 
be $\lambda$-graph bisystems satisfying FPCC.
If their symbolic matrix bisystems
$(\M_1^-, \M_1^+)$
and
$(\M_2^-, \M_2^+)$ 
are properly strong shift equivalent, then
there exists an isomorphism 
$\Phi: \F_{{\frak L}_1}\longrightarrow \F_{{\frak L}_2}$
of $C^*$-algebras such that 
$\Phi\circ\rho_{{\frak L}_1} =\rho_{{\frak L}_2}\circ\Phi$.
Hence it induces an isomorphism 
$\widehat{\Phi}: \R_{{\frak L}_1}\longrightarrow \R_{{\frak L}_2}$
between their  crossed products
$\F_{{\frak L}_1}\rtimes_{\rho_{{\frak L}_1}}\Z$
and
$\F_{{\frak L}_2}\rtimes_{\rho_{{\frak L}_2}}\Z$.
\end{theorem}
\begin{proof}
Assume that the symbolic matrix bisystems
$(\M_1^-, \M_1^+)$
and
$(\M_2^-, \M_2^+)$ 
are properly strong shift equivalent in $1$-step.
Let $\P_l, \Q_l, \X_l, \Y_l, l\in \Zp$ be symbolic matrices 
as in Definition \ref{def:PSSE}.
As in \cite[Section 6]{MaPre2019},
there exist a bipartite symbolic matrix bisystem
$(\widehat{\M}^-,\widehat{\M}^+)$
with
alphabets $C, D$ and specifications
$\varphi_1: \Sigma_1\longrightarrow C\cdot D, 
\varphi_2: \Sigma_2\longrightarrow D\cdot C
$
such that 
\begin{equation}
 (\M_1^-, \M_1^+) \overset{\varphi_1}{\simeq}(\widehat{\M}^{CD-},\widehat{\M}^{CD+}),
\qquad
 (\M_2^-, \M_2^+) \overset{\varphi_2}{\simeq}(\widehat{\M}^{DC-},\widehat{\M}^{DC+})  \label{eq:bipartitesms}
\end{equation}
where 
$(\widehat{\M}^{CD-},\widehat{\M}^{CD+})$
and
$(\widehat{\M}^{DC-},\widehat{\M}^{DC+})
$
are symbolic matrix bisystems defined by the following way (cf. \cite[Lemma 6.5]{MaPre2019})
\begin{gather*}
\widehat{M}^{CD-}_{l,l+1}:= (\X_{2l} \Y_{2l+1})^\kappa, \qquad
\widehat{M}^{DC-}_{l,l+1}:= (\Y_{2l} \X_{2l+1})^\kappa, \\
\widehat{M}^{CD+}_{l,l+1}:= \P_{2l} \Q_{2l+1}, \qquad
\widehat{M}^{DC+}_{l,l+1}:= \Q_{2l} \P_{2l+1} 
\end{gather*}
where for a symbolic matrix $\A$, the matrix $\A^\kappa$
is denoted by
$A^\kappa(i,j) =\kappa(\A(i,j))$ for the exchanging spacification $\kappa$.
Let 
$(\widehat{\frak L}^{CD-},\widehat{\frak L}^{CD+}),
(\widehat{\frak L}^{DC-},\widehat{\frak L}^{DC+})
$
and 
$(\widehat{\frak L}^-,\widehat{\frak L}^+)$
be the corresponding 
$\lambda$-graph bisystems to the  symbolic matrix bisystems
$(\widehat{\M}^{CD-},\widehat{\M}^{CD+}),
(\widehat{\M}^{DC-},\widehat{\M}^{DC+})
$
and
$(\widehat{\M}^-,\widehat{\M}^+)$, respectively.
The relations \eqref{eq:bipartitesms}
tells us that the $\lambda$-graph bisystems  
$({\frak L}_1^-, {\frak L}_1^+)$
and
$({\frak L}_2^-, {\frak L}_2^+)$
are identified with
$(\widehat{\frak L}^{CD-},\widehat{\frak L}^{CD+})$
and
$(\widehat{\frak L}^{DC-},\widehat{\frak L}^{DC+})$,
through 
the specifications
$\varphi_1: \Sigma_1\longrightarrow C\cdot D
$
and
$ 
\varphi_2: \Sigma_2\longrightarrow D\cdot C$,
respectively.
We identify symbols in $\Sigma_1$ (resp. $\Sigma_2$)
with symbols in $C\cdot D$ (resp. $D\cdot C$) through $\varphi_1$ (resp. $\varphi_2$).

Let us denote by
$E_{1,{i_1}}^{k,l}(\mu,\nu)$ for $\mu,\nu\in W_{i_1}^{n(k,l)}$
the generating partial isometries in 
$\F_{{\frak L}_1}^{k,l}$.
We use similar notation for 
$E_{2,{i_2}}^{k,l}(\xi,\eta)$.
Let $(\widehat{A}_{l,l+1}^-, \widehat{A}_{l,l+1}^+)$ be the transition matrix bisystem
for the bipartite $\lambda$-graph bisystem $(\widehat{\frak L}^-,\widehat{\frak L}^+)$.
We define homomorphisms
\begin{equation}
j_+: \F_{{\frak L}_1}^{k,l}\longrightarrow \F_{{\frak L}_2}^{k,l+1}, \qquad
j_-: \F_{{\frak L}_2}^{k,l}\longrightarrow \F_{{\frak L}_1}^{k-1,l} \label{eq:jjpm}
\end{equation}
by setting
\begin{align*}
 & j_+(E_{1,{i_1}}^{k,l}(\mu,\nu))\\
= & \sum_{c \in C, d \in D}
\sum_{h=1}^{m_1(\hat{n}(k,l)+1)}\sum_{j_2=1}^{m_2(k,l+1)}
\widehat{A}_{\hat{n}(k,l),\hat{n}(k,l)+1}^{+}(i_1,c,h)
\widehat{A}_{\hat{n}(k,l)+1,\hat{n}(k,l+1)}^{-}(h,d,j_2)
E_{2,j_2}^{k,l+1}(d\mu c, d\nu c), \\
& j_-(E_{2,i_2}^{k,l}(\xi,\eta)) \\
= & \sum_{c \in C, d \in D}
\sum_{g=1}^{m_2(\hat{n}(k,l)+1)} \sum_{j_1=1}^{m_1(k-1,l)}
\widehat{A}_{\hat{n}(k,l),\hat{n}(k,l)+1}^{+}(i_2,d,g)
\widehat{A}_{\hat{n}(k,l)+1,\hat{n}(k,l+1)}^{-}(g,d,j_1)
E_{1,j_1}^{k-1,l}(c\xi d, c\eta d),
\end{align*}
where $\hat{n}(k,l) = 2(l-k-1)$. 
We then have 
\begin{gather*}
j_{-}\circ j_+: \F_{{\frak L}_1}^{k,l}\longrightarrow \F_{{\frak L}_1}^{k-1,l+1},\qquad
j_{+}\circ j_-: \F_{{\frak L}_2}^{k,l}\longrightarrow \F_{{\frak L}_2}^{k-1,l+1} \\
\intertext{and}
j_{-}\circ j_+ = \iota_-\circ\iota_+ \quad \text{ in }\F_{{\frak L}_1}, \qquad
j_{+}\circ j_- = \iota_+\circ\iota_- \quad \text{ in }\F_{{\frak L}_2},
\end{gather*}
so that the homomorphisms 
$j_+: \F_{{\frak L}_1}^{k,l}\longrightarrow \F_{{\frak L}_2}^{k,l+1}, 
j_-: \F_{{\frak L}_1}^{k,l}\longrightarrow \F_{{\frak L}_2}^{k-1,l} 
$
yield isomorphisms between 
$\F_{{\frak L}_1}$ and $\F_{{\frak L}_2}$.
We denote the isomorphism 
induced by 
$j_+: \F_{{\frak L}_1}^{k,l}\longrightarrow \F_{{\frak L}_2}^{k,l+1}$
by $\Phi:\F_{{\frak L}_1}\longrightarrow \F_{{\frak L}_2}$, 
so that 
its inverse  $\Phi^{-1}:\F_{{\frak L}_2}\longrightarrow \F_{{\frak L}_1}$
is induced by 
$j_-: \F_{{\frak L}_2}^{k,l}\longrightarrow \F_{{\frak L}_1}^{k,l+1}$.
As the equalities
\begin{equation*}
j_+\circ\rho_{{\frak L}_1} =\rho_{{\frak L}_2}\circ j_+, \qquad
j_-\circ\rho_{{\frak L}_2} =\rho_{{\frak L}_1}\circ j_- 
\end{equation*}
hold, 
we know that $\Phi\circ \rho_{{\frak L}_1} =\rho_{{\frak L}_2}\circ \Phi$.
\end{proof}

%%%%%%%%%%%%%%%%%%%%%%%%%%%%%%%%%%%%%%%%%%%%%%%%%%%%

For subshifts $\Lambda_1$ and $\Lambda_2$,
we consider their canonical $\lambda$-graph bisystems
 $({\frak L}_1^-,{\frak L}_1^+)$ and $({\frak L}_2^-, {\frak L}_2^+)$
 for the subshifts
$\Lambda_1$ and $\Lambda_2$, respectively.
As a corollary, we  have
\begin{corollary}\label{cor:conjugacy}
Suppose that two subshifts 
$\Lambda_1$ and $ \Lambda_2$ are topologically conjugate.
Then there exists an isomorphism 
$\Phi: \F_{\Lambda_1}\longrightarrow \F_{\Lambda_2}$
of $C^*$-algebras such that 
$\Phi\circ\rho_{\Lambda_1} =\rho_{\Lambda_2}\circ\Phi$.
Hence it induces an isomorphism 
$\widehat{\Phi}: \R_{\Lambda_1}\longrightarrow \R_{\Lambda_2}$
between their  crossed products
$\F_{\Lambda_1}\rtimes_{\rho_{\Lambda_1}}\Z$
and
$\F_{\Lambda_2}\rtimes_{\rho_{\Lambda_2}}\Z$.
\end{corollary}
For a subshift $(\Lambda,\sigma)$,
let us denote by
$({}^t\!\Lambda, \sigma)$
the transposed subshift
$(\Lambda,\sigma^{-1})$
of $(\Lambda,\sigma)$.
The $C^*$-algebra  
$\F_{{}^t\!\Lambda}\rtimes_{\rho_{{}^t\!\Lambda}}\Z$
is canonically isomorphic to
$\F_{\Lambda}\rtimes_{{\rho_{\Lambda}}^{-1}}\Z$
that is also isomorphic to
$\F_{\Lambda}\rtimes_{{\rho_{\Lambda}}}\Z$.
Two subshifts $\Lambda_1$ and $\Lambda_2$
are said to be flip conjugate if 
$\Lambda_1$ is conjugate to $\Lambda_2$ or its transpose
${}^t\!\Lambda_2$.
Hence we have

\begin{corollary}\label{cor:flip}
Suppose that two subshifts 
$\Lambda_1$ are $\Lambda_2$ are flip conjugate.
Then there exists an isomorphism 
$\Psi:  \R_{\Lambda_1}\longrightarrow \R_{\Lambda_2}$
of $C^*$-algebras.
\end{corollary}
We remark that 
the isomorphism $\Phi:\F_{{\frak L}_1}\longrightarrow \F_{{\frak L}_2}$
defined in the proof of Theorem \ref{thm:psseC}
satisfies $\Phi({\mathcal{D}}_{{\frak L}_1}) ={\mathcal{D}}_{{\frak L}_2}$
and $\Phi(C(\Lambda_1)) =C(\Lambda_2)$
where $\Lambda_1, \Lambda_2$ are presented subshifts by 
${\frak L}_1,{\frak L}_2$, respectively.
The condition
$\Phi\circ\rho_{{\frak L}_1} =\rho_{{\frak L}_2}\circ\Phi$
implies the induced isomorphism
$\widehat{\Phi}: \R_{{\frak L}_1}\longrightarrow \R_{{\frak L}_2}$
satisfies
$\widehat{\Phi} \circ \hat{\rho}_{{\frak L}_1,t} 
=\hat{\rho}_{{\frak L}_2,t}\circ \widehat{\Phi}, \, t \in \T$,
where
$\hat{\rho}_{{\frak L}_i,t}$
denotes the dual action on
$\F_{{\frak L}_i}\rtimes_{{\rho_{{\frak L}_i}}}\Z$.

We remark also recent preprints \cite{MaPre2018a}, \cite{MaPre2018b}.
If in particular a subshift is a   topological Markov shift $(\Lambda_A,\sigma_A)$,
then the $C^*$-algebra $\R_{\Lambda_A}$ denoted by $\R_A$
is nothing but the 
asymptotic Ruelle algebra written $\R^a_{\sigma_A}$ in \cite{Putnam1}.
In this case
we know much more than Corollary \ref{cor:conjugacy} and 
Corollary \ref{cor:flip}.
\begin{proposition}[{\cite{MaPre2018a}, \cite{MaPre2018b}}]
Let $(\Lambda_A, \sigma_A)$ and
 $(\Lambda_B, \sigma_B)$ be the topological Markov shifts defined by 
 irreducible non-permutation matrices $A$ and $B$, respectively.
\begin{enumerate}
%\hspace{6cm}
\renewcommand{\theenumi}{\roman{enumi}}
\renewcommand{\labelenumi}{\textup{(\theenumi)}}
\item
 $(\Lambda_A, \sigma_A)$ and
 $(\Lambda_B, \sigma_B)$   
are topologically conjugate if and only if
there exists an isomorphism
$\Phi: \R_A\longrightarrow \R_B$ such that 
$\Phi(C(\Lambda_A)) = C(\Lambda_B)$
and
$\Phi\circ \hat{\rho}_{A,t} =\hat{\rho}_{B,t} \circ \Phi, \, t \in \T$.
\item
 $(\Lambda_A, \sigma_A)$ and
 $(\Lambda_B, \sigma_B)$   
are flip conjugate if and only if
there exists an isomorphism
$\Phi: \R_A\longrightarrow \R_B$ such that 
$\Phi(C(\Lambda_A)) = C(\Lambda_B)$
and
$\Phi\circ \hat{\rho}_{A, t} =\hat{\rho}_{B,\epsilon t} \circ \Phi, \, t \in \T$,
where $\epsilon = \pm 1$,
\end{enumerate}
where $\hat{\rho}_A$ (resp. $\hat{\rho}_B$) is the dual action on
$\R_A = \F_{\Lambda_A}\rtimes_{\rho_A}\Z$ 
(resp. $\R_B = \F_{\Lambda_B}\rtimes_{\rho_B}\Z$).  
\end{proposition}

%%%%%%%%%%%%%%%%%%%%%%%%%%%%%%%%%%%%%%%%%%%%%%%%%%%
%%%%%%%%%%%%%%%%%%%%%%%%%%%%%%%%%%%%%%%%%%%%%%%%
\section{Didimension groups and K-theory formulas}\label{sec:Dim}
%%%%%%%%%%%%%%%%%%%%%%%%%%%%%%%%%%%%%%%%%%
Let $\LGBS$ be a $\lambda$-graph bisystem over alphabets $\Sigma^\pm$.
Let $(A^-, A^+) =(A^-_{l,l+1}, A^+_{l,l+1})_{l\in \Zp} $ be its transition matrix bisystem.
We define the sequence $(M^-, M^+)$ of $m(l) \times m(l+1)$ matrices
by 
\begin{align}
M^-_{l,l+1}(i,j) & =\sum_{\beta\in \Sigma^-}A^-_{l,l+1}(i,\beta,j), \\
M^+_{l,l+1}(i,j) & =\sum_{\alpha\in \Sigma^+}A^+_{l,l+1}(i,\alpha,j), 
\end{align}
for $i=1,2,\dots,m(l), \, j=1,2,\dots,m(l+1)$.
By the local property of $\lambda$-graph bisystem, 
the commutation relations
\begin{equation}
M^-_{l,l+1}M^+_{l+1,l+2} =M^+_{l,l+1}M^-_{l+1,l+2}, \qquad l \in \Zp \label{eq:nnmbs}
\end{equation}
hold.
The sequence 
of pairs of nonnegative matrices is called a nonnegative matrix bisystem.
In this section, 
we will introduce the notions of dimension group, K-groups  
for nonnegative matrix bisystems.
The dimension group defined in this section 
is a generalization of the  
 dimension group for  nonnegative matrices defined by W. Krieger in 
\cite{Kr80Invent}, \cite{Kr80MathAnn},
and
for nonnegative matrix systems defined by the author in \cite{MaDocMath1999}.
For a $\lambda$-graph bisystem $\LGBS$ satisfying FPCC, 
the dimension group will be isomorphic to the $K_0$-group of the AF-algebra
$\F_{\frak L}$ and K-groups are isomorphic to the K-groups for the $C^*$-algebra
$\R_{\frak L}$.
Hence if $\LGBS$ is the canonical $\lambda$-graph bisystem for a subshift $\Lambda$,
the dimension group and K-groups are invariant under topological conjugacy of subshifts.

We  first formulate the  dimension group
and the dimension triple
for nonnegative matrix bisystems.
Let $(M^-, M^+)$
be a nonnegative matrix bisystem.
The transpose ${}^t\!M^{-}_{l,l+1}$ of the matrix $M_{l,l+1}^-$
naturally induces an order preserving homomorphism from
${\Z}^{m(l)}$ to
${\Z}^{m(l+1)}$,
where the positive cone
${\Z}^{m(l)}_{+}$
of the group
${\Z}^{m(l)}$
is defined by
$$
{\Z}^{m(l)}_+ = \{ 
(n_1,n_2,\dots,n_{m(l)}) \in {\Z}^{m(l)} |
n_i \in \Zp, i=1,2,\dots,m(l) \}.
$$
We put the inductive limits:
\begin{align*}
{\Z}_{M^-} & = 
\varinjlim \{ {}^t\!M^{-}_{l,l+1}: {\Z}^{m(l)} 
            \longrightarrow {\Z}^{m(l+1)} \}, \\ %\label{eq:indM1}
{\Z}_{M^{-}}^{+} & = 
\varinjlim \{ {}^t\!M^{-}_{l,l+1}: {\Z}^{m(l)}_+ 
            \longrightarrow {\Z}^{m(l+1)}_+ \} %\label{eq:indM2}
\end{align*}
of the abelian group and its positive cone.
By the relation \eqref{eq:nnmbs},
the sequence of the transposed matrices  ${}^t\!M^{+}_{l,l+1}$
of $M^+_{l,l+1}$ 
 naturally induces an order preserving  endomorphism
on the ordered group 
${\Z}_{M^-}$ 
by 
$x_l \in \Z^{m(l)}\longrightarrow {}^t\!M^{+}_{l,l+1}x_l \in \Z^{m(l+1)}
$
that is 
denoted by
$\lambda_{M^+}$. 
We set 
$
{\Z}_{M^-}(k) = {\Z}_{M^-}
$
and
$
{\Z}^+_{M^-}(k) = {\Z}^+_{M^-}
$
for
$ k \in \N$, 
and define an abelian group and its positive cone by the following inductive 
limits:
\begin{align}
\Delta_{(M^-,M^+)} & = 
\underset{k}{\varinjlim} \{ \lambda_{M^+}:{\Z}_{M^-}(k) 
                    \longrightarrow {\Z}_{M^-}(k+1) \}, \label{eq:dim}\\
\Delta^+_{(M^-,M^+)} & = 
\underset{k}{\varinjlim}\{ \lambda_{M^+}:{\Z}^+_{M^-}(k) 
                    \longrightarrow {\Z}^+_{M^-}(k+1) \}.\label{eq:dimpo}
\end{align} 
We call the ordered group
$(\Delta_{(M^-,M^+)},\Delta^+_{(M^-,M^+)})$
the {\it dimension group  for}\/ $(M^-,M^+)$.
Let us denote by $[X,k]$ for $X \in {\Z}_{M^-}(k), k \in \N$ 
the element $[X,k] $ in $\Delta_{(M^-,M^+)}.$
Since the map
$
\delta_{(M^-, M^+)} :{\Z}_{M^-}(k) 
                    \longrightarrow {\Z}_{M^-}(k+1) 
$
defined by
$\delta_{(M^-, M^+)}([X,k]) = [X,k+1]$
for
$X \in {\Z}_{M^-}
$ 
yields an automorphism on 
$\Delta_{(M^-,M^+)}$
that preserves the positive cone
$\Delta^+_{(M^-,M^+)}$.
We still denote it by $\delta_{(M^-,M^+)}$
and call it  
the {\it dimension automorphism}.
We call 
the triple
$(\Delta_{(M^-,M^+)},\Delta^+_{(M^-,M^+)},\delta_{(M^-,M^+)})$
the {\it dimension triple for}\/ $(M^-,M^+)$.
Let
$\Lambda$ be a subshift and
$(M^-_\Lambda,M^+_\Lambda)$ its associated nonnegative matrix bisystem 
for $\Lambda$.
Then the {\it bidimension triple}
$(\Delta_{\Lambda},\Delta^+_{\Lambda},\delta_{\Lambda})$
 for subshift $\Lambda$
is defined to be the dimension triple
$(\Delta_{(M^-_\Lambda,M^+_\Lambda)},\Delta^+_{(M^-_\Lambda, M^+_\Lambda)},
\delta_{(M^-_\Lambda, M^+_\Lambda)})$.
%The {\it past dimension triple} for $\Lambda$
%is defined as the future dimension triple for the transposed subshift
%$\Lambda^T$ for $\Lambda$.

Suppose that $\LGBS$ is a $\lambda$-graph bisystem satisfying FPCC.
Consider the sequence $\F_{\frak L}^{k,l}, k,l \in \N$ with $k<l$ of
finite dimensional $C^*$-subalgebras of the AF-algebra $\F_{\frak L}$.
 Put
$\F_{\frak L}^l = \F_{\frak L}^{-l,l}$ for $l \in \N$.
The following lemma is direct.
\begin{lemma}
We have an increasing sequence
$$
\F_{\frak L}^1 \hookrightarrow \F_{\frak L}^{2}\hookrightarrow \F_{\frak L}^{3}
\hookrightarrow \cdots 
$$
such that the union $\cup_{l=1}^\infty \F_{\frak L}^l$ 
generates the AF algebra $\F_{\frak L}$. 
\end{lemma}
We may write the AF algebra $\F_{\frak L}$ as
$
\F_{\frak L} = \lim_{l\to\infty}\F_{\frak L}^l.
$
We define the matrix component
$$
A^2_{l,l+1}(i,\beta,\alpha,j) \quad \text{ for }
i=1,2,\dots,m(2l-1),\,
j=1,2,\dots,m(2l+1),\,
\alpha,\beta \in \Sigma
$$
by
\begin{equation}
A^2_{l,l+1}(i,\beta,\alpha,j) 
= \sum_{k=1}^{m(2l)} A^-_{2l-1, 2l}(i, \beta,k)A^+_{2l, 2l+1}(k,\alpha,j).
\end{equation}
Then the inclusion map
$\iota_{l,l+1}: \F_{\frak L}^l \hookrightarrow \F_{\frak L}^{l+1}
$
defined by 
$$
\iota_{l,l+1}(E_i^{-l,l}(\mu,\nu)) := \iota_{+}\circ \iota_-(E_i^{-l,l}(\mu,\nu)),
\qquad l \in \N
$$
satisfies
\begin{equation}
\iota_{l,l+1}(E_i^{-l,l}(\mu,\nu)) =
\sum_{\alpha,\beta\in\Sigma}\sum_{j=1}^{m(2l+1)}
A^2_{l,l+1}(i,\beta,\alpha,j) E_j^{-l-1, l+1}(\beta\mu\alpha, \beta\nu\alpha). \label{eq:iotaA2}
\end{equation}
Put
$$
M^2_{l,l+1}(i,j) =
\sum_{\alpha,\beta\in\Sigma}
A^2_{l,l+1}(i,\beta,\alpha,j)
$$
so that
$M^2_{l,l+1}(i,j) = [M^-_{2l-1, 2l} M^+_{2l,2l+1}](i,j)
$
for $i=1,2,\dots, m(2l-1), \, j=1,2,\dots,m(2l+1)$.
Denote by 
${}^t\!M^2_{l,l+1}$ the transpose of
$M^2_{l,l+1}$.
We thus have
\begin{lemma}
For $l\in \N$, we have the commutative diagram:
\begin{equation*}
\begin{CD}
K_0(\F_{\frak L}^{l})  @>\iota_{l,l+1*}>> 
K_0(\F_{\frak L}^{l+1})  \\
@|   @| \\
 \Z^{m(2l-1)}
@>{}^t\!M^2_{l,l+1}>> 
 \Z^{m(2l+1)}. \\ 
\end{CD}
\end{equation*}
\end{lemma}
\begin{proof}
As
$
\F_{\frak L}^l= \oplus_{i=1}^{m(2l-1)}\F_{i}^{-l,l} 
=\oplus_{i=1}^{m(2l-1)} M_{N_i^{2l-1}}({\mathbb{C}}),
$
we have
\begin{equation}
K_0(\F_{\frak L}^l) \cong 
\bigoplus_{i=1}^{m(2l-1)}K_0(M_{N_i^{2l-1}}({\mathbb{C}}))\cong 
\bigoplus_{i=1}^{m(2l-1)}\Z = \Z^{m(2l-1)}. \label{eq:KZ}
\end{equation}
%The group $K_0(M_{N_i^{2l-1}}({\mathbb{C}}))$ and hence
%.
By \eqref{eq:iotaA2}, the equality 
\begin{equation}
\iota_{l,l+1*}([E_i^{-l,l}(\mu,\mu)]) 
=\sum_{j=1}^{m(2l+1)}
M^2_{l,l+1}(i,j) [E_j^{-l-1, l+1}(\beta\mu\alpha, \beta\mu\alpha)]
\end{equation}
holds.
Since the group $K_0(\F_{\frak L}^l)$ is generated by the projections
$E_i^{-l,l}(\mu,\mu), \, \mu \in W_i^{2l-1}, i=1,2,\dots,m(2l-1)$,
we obtain the desired commutative diagram.
\end{proof}
Since
$$
K_0(\F_{\frak L}) \cong
{\displaystyle\lim_{l\to\infty}}
\{ \iota_{l,l+1*}:  K_0(\F_{\frak L}^l) \longrightarrow K_0(\F_{\frak L}^{l+1}) \},
$$
 we have
\begin{proposition}
$K_0(\F_{\frak L}) \cong
{\displaystyle\lim_{l\to\infty}}\{ {}^t\!M^2_{l,l+1}: \Z^{m(2l-1)} \longrightarrow \Z^{m(2l+1)} \}.
$
\end{proposition}
Define the ordered abelian group $(D_{(M^-, M^+)}, D_{(M^-, M^+)}^+)$ 
 by setting
\begin{align*}
D_{(M^-, M^+)} := &
\lim_{l\to\infty}
\{ {}^t\!M^2_{l,l+1}:  \Z^{m(2l-1)} \longrightarrow \Z^{m(2l+1)} \}, \\
D_{(M^-, M^+)}^+ := &
\lim_{l\to\infty}
\{ {}^t\!M^2_{l,l+1}:  \Z^{m(2l-1)}_+ \longrightarrow \Z^{m(2l+1)}_+ \}
\end{align*}
and the map
$d_{(M^-,M^+)}: \Z^{m(2l-1)}\longrightarrow \Z^{m(2l+1)}
$
by
\begin{equation*}
d_{(M^-,M^+)}(x):= {}^t\!M^2_{2l,2l+1}{}^t\!M^2_{2l-1,2l}x \quad \in \Z^{m(2l+1)}
\quad \text{ for } x \in \Z^{m(2l-1)}.
\end{equation*}
Since
${}^t\!M^2_{2l+1,2l+2}\circ d_{(M^-,M^+)} = d_{(M^-,M^+)}\circ {}^t\!M^2_{2l-2,2l-1}
$ for $l \in \Zp$,
the map
$d_{(M^-,M^+)}$ yields an endomorphism on $D_{(M^-, M^+)}$,
that we still denote by $d_{(M^-,M^+)}$.
The identification between $K_0(\F_{\frak L}^l)$ and 
$\Z^{m(2l-1)}$ in \eqref{eq:KZ} gives rise to the following
proposition.
\begin{proposition}\label{prop:KD}
Let ${\frak L}$ be a $\lambda$-graph bisystem satisfying FPCC and $(M^-, M^+)$ 
its nonnegative matrix bisystem. 
Then we have a natural identification
\begin{equation*}
(K_0(\F_{\frak L}),K_0(\F_{\frak L})_+, \rho_{{\frak L}*})  
\cong (D_{(M^-, M^+)},D_{(M^-, M^+)}^+, d_{(M^-, M^+)}).
\end{equation*}
\end{proposition}
We have the following theorem.
\begin{theorem}\label{thm:DDelta}
There exists an isomorphism 
$\Psi:K_0(\F_{\frak L}) \longrightarrow \Delta_{(M^-, M^+)}$
of abelian groups such that 
$\Psi(K_0(\F_{\frak L})_+)=\Delta_{(M^-, M^+)}^+$
and
$\Psi\circ \rho_{{\frak L}*} = \delta_{(M^-, M^+)}\circ \Psi$, that is
\begin{equation*}
(K_0(\F_{\frak L}),K_0(\F_{\frak L})_+, \rho_{{\frak L}*})
 \cong
(\Delta_{(M^-, M^+)}, \Delta_{(M^-, M^+)}^+, \delta_{(M^-, M^+)}).
\end{equation*}
\end{theorem}
\begin{proof}
Since $\F_{\frak L}$
is the AF algebra of the inductive system
${\displaystyle\lim_{k,l\to\infty}}\F_{\frak L}^{k,l}$,
we have
$K_0(\F_{\frak L}) ={\displaystyle\lim_{k,l\to\infty}} K_0(\F_{\frak L}^{k,l})$.
Hence for $X \in K_0(\F_{\frak L})$,
we may assume that
$ X =[X_i^{k,l}]_{i=1}^{m(k,l)} \in K_0(\F_{\frak L}^{k,l}) = \Z^{m(k,l)}$
for some $k,l\in \Z$ with $k<l$,
where $m(k,l) = m(l-k-1).$  
Since
$$
\Delta_{(M^-, M^+)} 
= \lim_{l\to\infty}\{\lambda_M^+: \Z_{M^-}(l) \longrightarrow \Z_{M^-}(l+1)\}
$$
where
$\Z_{M^-}(l) =\Z_{M^-}= 
 \lim_{n\to\infty}\{ {}^t\!M^{-}_{n,n+1}: \Z^{m(n)} \longrightarrow \Z^{m(n+1)}\}, l\in \N$,
we may define 
$
\Psi(X_i^{k,l}) \in \Z_{M^-}(l)
$
by
$$
\Psi(X_i^{k,l}) = [X_i^{k,l}]_{i=1}^{m(l-k-1)} \in \Z^{m(l-k-1)} \quad
\text{ in } \Z_{M^-}(l). 
$$
We will show that $\Psi$ yields a well-defined homomorphism
from
$
K_0(\F_{\frak L})
$ to
$\Delta_{(M^-, M^+)}$.

Suppose that 
$[X] =[Y]$ in 
$
K_0(\F_{\frak L})
$
with 
$X =[X_i^{k_1,l_1}]_{i=1}^{m(l_1-k_1-1)} \in \Z^{m(l_1-k_1-1)}$
and
$Y =[Y_j^{k_2,l_2}]_{j=1}^{m(l_2-k_2-1)} \in \Z^{m(l_2-k_2-1)}$.
Take $k, l \in \Z$ such that $k<k_i, \, l_i<l$ for $i=1,2$
and
\begin{equation}
\iota_+^{l- l_1}\circ\iota_-^{k_1 -k}(X)
= 
\iota_+^{l- l_2}\circ\iota_-^{k_2 -k}(Y). \label{eq:iotapm}
\end{equation}
We then have 
\begin{align*}
\Psi(X) 
= & [X_i^{k_1,l_1}]_{i=1}^{m(l_1-k_1-1)} \in \Z^{m(l_1-k_1-1)} \quad
\text{ in } \Z_{M^-}(l_1), \\
\Psi(Y) 
= & [Y_j^{k_2,l_2}]_{j=1}^{m(l_2-k_2-1)} \in \Z^{m(l_2-k_2-1)} \quad
\text{ in } \Z_{M^-}(l_2).
\end{align*}
As 
\begin{equation*}
\lambda_{M^+}^{l-l_1}([X_i^{k_1,l_1}]_{i=1}^{m(l_1-k_1-1)}), \quad
\lambda_{M^+}^{l-l_2}([Y_j^{k_2,l_2}]_{j=1}^{m(l_2-k_2-1)}) \in \Z_{M^-}(l)
 \end{equation*}
such that 
\begin{gather*}
\lambda_{M^+}^{l-l_1}([X_i^{k_1,l_1}]_{i=1}^{m(l_1-k_1-1)}) \in 
\Z^{m((l_1-k_1-1) + (l-l_1))} =  \Z^{m(l-k_1-1)} , \\
\lambda_{M^+}^{l-l_2}([Y_j^{k_2,l_2}]_{j=1}^{m(l_2-k_2-1)}) 
\in \Z^{m((l_2-k_2-1) + (l-l_2))} = \Z^{m(l-k_2-1) },
\end{gather*}
by \eqref{eq:iotapm}, we have
\begin{align}
& {}^t\!M^{-}_{l-k_1-1, (l-k_1-1) + (k_1-k)}
(\lambda_{M^+}^{l-l_1}([X_i^{k_1,l_1}]_{i=1}^{m(l_1-k_1-1)}) 
) \\
= &
{}^t\!M^{-}_{l-k_2-1, (l-k_2-1) + (k_2-k)}
(\lambda_{M^+}^{l-l_2}([Y_j^{k_2,l_2}]_{j=1}^{m(l_2-k_2-1)}))
\quad \text{ in }
\Z^{m(l-k-1)}. \label{eq:MlMl} 
\end{align}
We know that
\begin{equation}
[[X_i^{k_1,l_1}]_{i=1}^{m(l_1-k_1-1)} \text{ in }  \Z_{M^-}(l_1)]
=
[\lambda_{M^+}^{l-l_1}([X_i^{k_1,l_1}]_{i=1}^{m(l_1-k_1-1)} \text{ in } 
\Z_{M^-}(l)] \label{eq:XMX}
\end{equation}
as elements of 
$
\Delta_{(M^-, M^+)} 
$
%(= \lim_{l\to\infty}\{\lambda_M^+: \Z_{M^-}(l) \longrightarrow \Z_{M^-}(l+1)\})
and
\begin{align}
& [\lambda_{M^+}^{l-l_1}([X_i^{k_1,l_1}]_{i=1}^{m(l_1-k_1-1)}) \text{ in } \Z^{m(l-k_1-1)}] \\  
= & 
[{}^t\!M^{-}_{l-k_1-1, (l-k_1-1) + (k_1-k)}
(\lambda_{M^+}^{l-l_1}([X_i^{k_1,l_1}]_{i=1}^{m(l_1-k_1-1)}) 
) \text{ in } \Z^{m(l-k-1)}] \label{eq:lambdaMXM} 
\end{align}
as elements of 
$\Z_{M^-}$. 
%(= \lim_{n\to\infty}\{ M^{-t}_{n,n+1}: \Z^{m(n)} \longrightarrow \Z^{m(n+1)} \}).
Similarly,
\begin{equation}
[[Y_j^{k_2,l_2}]_{j=1}^{m(l_2-k_2-1)} \text{ in } \Z_{M^-}(l_2) ]
=
[\lambda_{M^+}^{l-l_2}([Y_j^{k_2,l_2}]_{j=1}^{m(l_2-k_2-1)}) \text{ in } 
\Z_{M^-}(l) ] \label{eq:YMY}
\end{equation}
as elements of 
$
\Delta_{(M^-, M^+)} 
$
and
\begin{align}
& [\lambda_{M^+}^{l-l_2}([Y_j^{k_2,l_2}]_{j=1}^{m(l_2-k_2-1)}) \text{ in } \Z^{m(l-k_2-1)}] \\  
= & 
[{}^t\!M^{-}_{l-k_2-1, (l-k_2-1) + (k_2-k)}
(\lambda_{M^+}^{l-l_2}([Y_j^{k_2,l_2}]_{j=1}^{m(l_2-k_2-1)}) 
) \text{ in }\Z^{m(l-k-1)}] \label{eq:lambdaMYM} 
\end{align}
as elements of 
$\Z_{M^-}$.
Therefore by
\eqref{eq:MlMl},
\eqref{eq:XMX},
\eqref{eq:lambdaMXM}, 
\eqref{eq:YMY}
and
\eqref{eq:lambdaMYM} 
we know that 
$\Psi(X) = \Psi(Y)$
as elements of 
$\Delta_{(M^-, M^+)}$
so that 
 $\Psi$ yields a well-defined homomorphism
from
$
K_0(\F_{\frak L})
$ to
$\Delta_{(M^-, M^+)}$.

Conversely,
for $[Z] \in \Z_{M^-}(l)$ in 
$\Delta_{(M^-,M^+)}$,
we my assume that 
$Z \in \Z^{m(n)}$ for some $n \in \N$
so that $Z = [Z_i]_{i=1}^{m(n)}.$
Put $ k= l-1-n$ and hence $n = l-k-1$
so that
$Z^{k,l}:= [Z_i]_{i=1}^{m(l-k-1)} \in \Z^{m(k,l)}.$
Define 
$\Phi(Z) := [Z^{k,l}] $ in $\Z^{m(k,l)} = K_0(\F_\Lambda^{k,l}).$
It is routine to check that $\Phi$ yields a well-defined homomorphism
from
$\Delta_{(M^-, M^+)}$
to
$K_0(\F_{\frak L})$
such that 
$$
\Phi\circ \Psi = \id_{K_0(\F_{\frak L})}, \qquad
\Psi\circ \Phi = \id_{\Delta_{(M^-, M^+)}}.
$$
Therefore
the group $K_0(\F_{\frak L})$
is isomorphic to 
the group 
$
\Delta_{(M^-, M^+)}.
$
The construction of $\Phi$ and $\Psi$ shows that 
$
\Psi(K_0(\F_{\frak L})_+) \subset 
\Delta_{(M^-, M^+)}^+
$
and
$
\Phi(\Delta_{(M^-, M^+)}^+)
\subset
K_0(\F_{\frak L})_+) 
$ so that 
 $$
\Psi(K_0(\F_{\frak L})_+) = \Delta_{(M^-, M^+)}^+,
\qquad
\Phi(\Delta_{(M^-, M^+)}^+)
=
K_0(\F_{\frak L})_+
$$
and hence  
$\Psi$ and $\Phi$
give rise to order preserving isomorphisms inverses to each other.

For  $X =[X_i^{k,l}]_{i=1}^{m(k,l)} \in K_0(\F_{\frak L}^{k,l})$,
we have 
$$
\rho_{{\frak L} *}(X) = [X_i^{k+1,l+1}]_{i=1}^{m(k+1,l+1)} \in K_0(\F_{\frak L}^{k+1,l+1})
$$
where 
$X_i^{k,l}=  X_i^{k+1,l+1}$ and $m(k,l) = m(k+1,l+1)$
so that 
\begin{equation}
(\Psi\circ\rho_{{\frak L} *})(X) 
= [X_i^{k+1,l+1}]_{i=1}^{m(k+1,l+1)} \quad \in \Z^{m(l-k-1)} \text{ in }
\Z_{M^-} (l+1). \label{eq:Psirho}
\end{equation}
Since the right hand side of \eqref{eq:Psirho} is nothing but 
$\delta_{(M^-, M^+)}\circ \Psi(X)$,
we have 
$\Psi\circ\rho_{{\frak L} *} = \delta_{(M^-, M^+)}\circ \Psi$.
\end{proof}

%%%%%%%%%%%%%%%%%%%%%%%%%%%%%%%%%%%%%%%%%%%%%%%%%%%
%%%%%%%%%%%%%%%%%%%%%%%%%%%%%%%%%%%%%%%%%%%%%%%%
We will next present  K-theory formulas for the $C^*$-algebra
$\R_{\frak L}$.
Let
$(M^-,M^+)$ be a nonnegative matrix bisystem.
For $l \in \Zp$, 
we set the abelian groups
\begin{align*}
K_0^l(M^-,M^+) & = 
{\Z}^{m(l+1)} / ({}^t\!M^{-}_{l,l+1} - {}^t\!M^{+}_{l,l+1}){\Z}^{m(l)},\\
K_1^l(M^-,M^+) & =
 \text{Ker}({}^t\!M^{-}_{l,l+1} - {}^t\!M^{+}_{l,l+1}) 
 \text{ in } {\Z}^{m(l)}. 
\end{align*}
The following lemma is straightforward by using the commutation relations
\eqref{eq:nnmbs}. 
%\begin{equation}
%M^-_{l,l+1}M^+_{l+1,l+2} = M^+_{l,l+1} M^-_{l+1,l+2},
%\qquad l \in \Zp. \label{eq:MM}
%\end{equation}   
\begin{lemma}\label{lem:DM9.1}
The map
${}^t\!M^{-}_{l,l+1} : {\Z}^{m(l)} \longrightarrow {\Z}^{m(l+1)}$
 naturally induces homomorphisms between the following groups:
$$
 {}^t\!M_*^{-l} : K_*^l(M^-,M^+) \rightarrow K_*^{l+1}(M^-,M^+)
 \qquad  \text{ for } \quad  *=0,1.
$$
\end{lemma}

We now define the K-groups for nonnegative matrix bisystem $(M^-,M^+)$.
\begin{definition}
The {\it K-groups for } 
$(M^-,M^+)$ are defined by the inductive limits of the abelian groups:
\begin{align}
K_0(M^-,M^+) & = \underset{l}{\varinjlim} \{ {}^t\!M_0^{-l}:
 K_0^l(M^-,M^+) \longrightarrow K_0^{l+1}(M^-,M^+) \},\\
K_1(M^-,M^+) & = \underset{l}{\varinjlim} \{ {}^t\!M_1^{-l}:
 K_1^l(M^-,M^+) \longrightarrow K_1^{l+1}(M^-,M^+) \}.
\end{align}
\end{definition}
It is routine to show the following lemma by definition. 
\begin{lemma}\label{prop:DM9.2}
\hspace{6cm}
\begin{enumerate}
\renewcommand{\theenumi}{\roman{enumi}}
\renewcommand{\labelenumi}{\textup{(\theenumi)}}
\item
$
K_0(M^-,M^+) = {\Z}_{M^-} / (\id - \lambda_{M^+}){\Z}_{M^-},
$
\item
$
K_1(M^-,M^+) = {\Ker}(\id - \lambda_{M^+}) \text{ in } {\Z}_{M^-}.
$
\end{enumerate}
\end{lemma}
The following formulas show that 
the groups
$K_*(M^-,M^+)$ are determined by its dimension triple.
\begin{lemma}\label{lem:DM9.3}
\hspace{6cm}
\begin{enumerate}
\renewcommand{\theenumi}{\roman{enumi}}
\renewcommand{\labelenumi}{\textup{(\theenumi)}}
\item
$
K_0(M^-,M^+) = \Delta_{(M^-,M^+)} / (\id - \delta_{(M^-,M^+)})\Delta_{(M^-,M^+)},
$
\item
$
K_1(M^-,M^+) = {\Ker}(\id - \delta_{(M^-,M^+)} ) 
\text{ in } \Delta_{(M^-,M^+)}.
$
\end{enumerate}
\end{lemma}
\begin{proof}
As the automorphism
$\delta_{(M^-,M^+)}$ is given by
$\lambda_{M^+} = \{ {}^t\!M^{+}_{l,l+1} \}$ 
on  
$\Delta_{(M^-,M^+)}$,
the assertions are straightforward
by Lemma \ref{prop:DM9.2}.
\end{proof}

Now the $C^*$-algebra $\R_{\frak L}$ is the crossed product
$\F_{\frak L}\rtimes_{\rho_{\frak L}}\Z$, so that the six term
exact sequence of K-theory (\cite{PV})
\begin{equation*}
\begin{CD}
 K_0(\F_{\frak L}) @>\id - {\rho_{{\frak L}*}}>> 
 K_0(\F_{\frak L}) @>\iota_{*}>> 
 K_0(\R_{\frak L}) \\
@AAA @. @VVV \\
 K_1(\R_{\frak L})
@<<\iota_{*}<
 K_1(\F_{\frak L}) @<<\id - {\rho_{{\frak L}*}}< 
 K_1(\F_{\frak L}) \\ 
\end{CD}
\end{equation*}
holds.
As $\F_{\frak L}$ is an AF-algebra, 
one sees that $K_1(\F_{\frak L}) = 0$, 
and hence  we have the following formulas.
\begin{lemma}\label{lem:Ktheory1}
\begin{align*}
K_0(\R_{\frak L})  \, 
& \cong \, K_0(\F_{\frak L}) / (\id - {\rho}_{{\frak L}*}) K_0(\F_{\frak L}), \\
K_1(\R_{\frak L}) \,  
& \cong \, \Ker(\id - {\rho}_{{\frak L}*}) \text{ in } K_0(\F_{\frak L}). \\
\end{align*}
\end{lemma}
%%%%%%%%%%%%%%%%%%%%%%%%%%%%%%%%%%%%%%%%
By Theorem \ref{thm:DDelta},  
we conclude the following theorem.
\begin{theorem}\label{thm:RLKgroup}
$
K_i(\R_{\frak L}) \cong K_i(M^-,M^+), \, i=0,1.$
\end{theorem}

%\newpage

%%%%%%%%%%%%%%%%%%%%%%%%%%%%%%%%%%%%%%%%%%%
%%%%%%%%%%%%%%%%%%%%%%%%%%%%%%%%%%%%%%%
\section{Topological Markov shifts}
%%%%%%%%%%%%%%%%%%%%%%%%%%%%%%%%%%%%%%%%%
In this section, we will compute the dimension triple 
$(\Delta_{(M^-,M^+)},\Delta_{(M^-,M^+)}^+,\delta_{(M^-,M^+)})$
and the K-groups $K_i(\R_{\frak L}), i=0,1$ for some examples.

Let $G =(V,E)$ be a finite directed graph.
We assume that $G$ is primitive (see \cite[p. 127]{LM}). 
Let $V = \{ v_1,\dots,v_N\}$
and put
$\Sigma =E$ the edge set of $G$.
We further assume that $G$ does not have multiple edges.
Put $\widetilde{V} = V \times V$.
We write each element 
$(v_p, v_i) $ of $\widetilde{V}$
as $(p,i)$, so that 
$\widetilde{V} = \{(p,i) \mid p,i=1,2,\dots,N\}$.
We also write an edge from $v_i$ to $v_j$ as $\alpha_{ij}$.
If there is no edge from $v_i$ to $v_j$,
we write $\alpha_{ij} =0$. 
Hence the adjacency matrix of the graph is written
$\A = [\alpha_{ij}]_{i,j=1}^N$.
Let $\Lambda_G$ be the topological Markov shift
over $\Sigma$
defined by the graph $G$ (see \cite{LM}).
\begin{lemma}
Let $(\M_{\Lambda_G}^-,\M_{\Lambda_G}^+)$ be the canonical symbolic matrix bisystem for the topological Markov shift $\Lambda_G$ defined by a primitive finite directed graph $G$ as an edge shift.
Define $N^2 \times N^2$ matrices $\M_G^-, \M_G^+$
over $\Sigma$ by setting
\begin{align*}
\M_G^-((p,i), (q,j))
=& 
{\begin{cases}
\alpha_{qp} & \text{ if } i=j, \\
0 & \text{ if } i\ne j, 
\end{cases}} \\
\M_G^+((p,i), (q,j))
=& 
{\begin{cases}
\alpha_{ij} & \text{ if } p=q, \\
0 & \text{ if }  p \ne q 
\end{cases}} \\
\end{align*}
for $(p,i), (q,j) \in N\times N =N^2$, that is 
\begin{equation}
\M_G^- := 
\begin{bmatrix}
\alpha_{11} I_N & \alpha_{21} I_N  & \cdots & \alpha_{N1} I_N  \\
\alpha_{12} I_N  & \alpha_{22} I_N  & \cdots  & \alpha_{N2} I_N  \\
\vdots              & \vdots              & \ddots & \vdots    \\
\alpha_{1N} I_N  & \alpha_{2N} I_N &  \cdots& \alpha_{NN} I_N 
\end{bmatrix},
\qquad
\M_G^+ := 
\begin{bmatrix}
\A  & 0    & \dots & 0 \\
0     & \A & \ddots & \vdots \\
\vdots     & \ddots  & \ddots & 0 \\
0     & \dots     &  0 & \A
\end{bmatrix}. \label{eq:symbma2}
\end{equation}
Then there exists $L \in \N$ such that  
\begin{equation*}
{\M_{\Lambda_G}^-}_{l,l+1} = {\M_G^-}_{l,l+1}, \qquad
{\M_{\Lambda_G}^+}_{l,l+1} = {\M_G^+}_{l,l+1} \quad \text{ for all }\,\, l \ge L.
\end{equation*}
\end{lemma}
\begin{proof}
Since the matrix $\A$ is primitive,
there exists $L$ such that $\A^l(i,j) \ne 0$ for all
$i,j =1,2,\dots, N, \, l \ge L.$
Take and fix $l \ge L.$
The equivalence class
$\Omega_{-l, 1}^c$ is determined by the two vertices
$v_q$ and $v_j$ such that 
$y \in \Lambda$ belongs to $\Omega_{-l,1}^c$
if and only if  $v_q = t(y_{-l})$ and $v_j = s(y_1)$.  
If there exists $\beta \in \Sigma =E$ such that 
$s(\beta) = v_q, t(\beta) = v_p$,
then the class $\Omega_{-l+1,1}^c$
defined by $y_{(-\infty,l]} \beta$ and $y_{[1,\infty)}$
determines the two vertices $v_p$ and $v_i$.
Hence if there exists an edge $\beta$ in $G$ such that 
$s(\beta) = v_q, t(\beta) = v_p$, then
we have
$\M_G^-((p,j), (q,j)) = \beta (= \alpha_{qp}) $.  
It is clear that 
$\M_G^-((p,i), (q,j)) = 0$ if $i\ne j.$

On the other hand, 
the equivalence class
$\Omega_{-1,l}^c$ is determined by the two vertices
$v_p$ and $v_j$ such that 
$x \in \Lambda$ belongs to $\Omega_{-1,l}^c$
if and only if  $v_p = t(x_{-1})$ and $v_j = s(x_l)$.  
If there exists $\alpha \in \Sigma =E$ such that 
$s(\alpha) = v_i, t(\alpha) = v_j$,
then the class $\Omega_{-1,l-1}^c$
defined by $x_{(-\infty,-1]}$ and $\alpha x_{[l,\infty)}$
determines the two vertices $v_p$ and $v_i$.
Hence if there exists an edge $\alpha$ in $G$ such that 
$s(\alpha) = v_i, t(\alpha) = v_j$, then
we have
$\M_G^+((p,i), (p,j)) = \alpha (=\alpha_{ij})$.  
It is clear that 
$\M_G^+((p,i), (q,j)) = 0$ if $p\ne q.$
\end{proof}

For the symbolic matrix $\A$ above, let us denote by $A$ the $N \times N$ matrix 
over $\{0,1\}$ obtained from $\A$ by setting all the symbols $\alpha_{ij}$ equal to $1$.
Let us denote by $M_N(\Z)$
the abelian group of $N\times N$ matrix algebra over $\Z$.
For the $N\times N$ matrix $A$,
let $\Psi_A: M_N(\Z) \longrightarrow M_N(\Z) $
be the homomorphism defined by
$\Psi_A(X) = {}^t\!A X {}^t\!A$ for $X \in M_N(\Z). $
In our situation, we note that $m(2l-1) = m(2l+1) = N^2$ for all $l \ge L.$
Let $M_G^-, M_G^+$ be 
the $N^2 \times N^2$ matrix 
over $\{0,1\}$ obtained from $\M_G^-, \M_G^+$ by setting all the symbols $\alpha_{ij}$ equal to $1$, respectively.
We put $M_G^2 = M_G^- M_G^+$.
\begin{lemma}
For $l\in \N$ with $l \ge L$, we have the commutative diagram:
\begin{equation*}
\begin{CD}
\Z^{m(2l-1)} @>M_G^2>>  
\Z^{m(2l+1)}  \\
@|   @| \\
 M_N(\Z)
@>\Psi_A>> 
 M_N(\Z). \\ 
\end{CD}
\end{equation*}
\end{lemma}
\begin{proof}
We assume  $l \ge L$.
%Define the nonnegative matrix$M^2 = M^- M^+$.
For $X =[X_{(q,i)}]_{(q,i)\in N\times N}\in M_N(\Z)$,
we have
\begin{align*}
 [{}^t\!A X {}^t\!A]_{(p,j)}
=& \sum_{(q,i)\in N\times N}\alpha_{qp}\alpha_{ji} X_{(q,i)} \\
=& \sum_{(q,i)\in N\times N}M_G^- M_G^+((p,j), (q,i)) X_{(q,i)} \\
=& \sum_{(q,i)\in N\times N}M_G^2((p,j), (q,i)) X_{(q,i)} 
\end{align*} 
so that we obtain the desired commutative diagram.
\end{proof}
Therefore we have the following proposition.
\begin{proposition}
Let $(\Lambda_G,\sigma)$ be the topological Markov shift
defined by a primitive directed graph $G =(V,E)$.
Let $A$ be its transition matrix.
Then we have
\begin{equation}
K_0(\F_{\Lambda_G}) 
=\lim_{l\to\infty}\{\Psi_A: M_N(\Z)\longrightarrow M_N(\Z) \} \label{eq:K0markov}
\end{equation}
where
$\Psi_A(X) = {}^t\!A X {}^t\!A$ for $X \in M_N(\Z).$
\end{proposition}

The above formula 
\eqref{eq:K0markov} has appeared in Putnam's paper \cite{Putnam1} 
(cf. \cite{Holton}, \cite{MaCJM}).
In his paper, the right hand of \eqref{eq:K0markov} is written 
as $H({}^t\!A)$.  
For the full $N$-shift, the group $K_0(\F_{\Lambda_G})$ 
has been already computed as in the following way.
\begin{proposition}[{\cite[Lemma 10.2]{MaPre2019}}] \label{prop:7.4}
Let $\Lambda_N$ be the full $N$-shift for $1<N\in \N.$
Then the bidimension triple  
$(\Delta_{\Lambda_N},\Delta_{\Lambda_N}^+, \delta_{\Lambda_N})$
is isomorphic to
$(\Z[\frac{1}{N^2}], \Z[\frac{1}{N^2}]^+, N\times )$,
where $N\times$ means the multiplication by $N$ on $\Z[\frac{1}{N^2}].$ 
\end{proposition}

We remark that the ordinary dimension triple for the full $N$-shift 
$\Lambda_N$ is
$(\Z[\frac{1}{N}], \Z[\frac{1}{N}]^+, N\times )$.
Although the ordered groups
$(\Z[\frac{1}{N^2}], \Z[\frac{1}{N^2}]^+ )$
and
$(\Z[\frac{1}{N}], \Z[\frac{1}{N}]^+)$
are isomorphic,
their dimension triples are different.

We will now present examples of 
$K_i(\R_{\Lambda_G}), i=1,2$
for topological Markov shifts $\Lambda_G$.

\medskip

{\bf Examples}.

{\bf 1. Full $2$-shift.} 
Let 
$\A = 
\begin{bmatrix}
a & b \\
c & d 
\end{bmatrix}
$
be a symbolic matrix over alphabet $\Sigma =\{ a,b,c,d\}.$ 
Let $G_2$ be the finite directed graph having two vertices and four directed edges
associated with the symbolic matrix $\A$.
Let us consider the symbolic matrix bisystem
$(\M^-_{G_2}, \M^+_{G_2})$ defined by 
\eqref{eq:symbma2}.
It satisfies FPCC.
Let $({\frak L}^-_{G_2}, {\frak L}^+_{G_2})$ be the associated $\lambda$-graph bisystem.
 Its presenting subshift is topologically conjugate to the full $2$-shift.  
Put 
$
M_2
=
\begin{bmatrix}
1 & 1 \\
1 & 1
\end{bmatrix}.
$
By \eqref{eq:symbma2},
 the associated nonnegative matrix bisystem
$(M_2^-, M_2^+)$ is given by  
\begin{equation}
M_2^- = 
\begin{bmatrix}
1 & 0   & 1 & 0 \\
0   & 1 & 0   & 1 \\
1 & 0   & 1 & 0  \\
0   & 1 & 0   & 1 
\end{bmatrix}, \qquad
M_2^+ =
\begin{bmatrix}
1 & 1 & 0 & 0 \\
1 & 1 & 0 & 0 \\
0 & 0 & 1 & 1 \\
0 & 0 & 1 & 1 \\ 
\end{bmatrix}.  \label{eq:nnmbma2}
\end{equation}
We will show the following
\begin{proposition}\label{prop:7.5}
$
K_i(\R_{\Lambda_{G_2}}) 
\cong  \Z[\frac{1}{2}], \, i=0,1.
$
\end{proposition}
\begin{proof}
Since
$$
{}^t\!M^{-}_{2} - {}^t\!M^{+}_{2} 
=M^{-}_{2} - M^{+}_{2} 
=
\begin{bmatrix}
0  &-1 & 1  & 0 \\
-1 & 0 & 0  & 1 \\
 1 & 0 & 0   &-1  \\
0  & 1 & -1 & 0 
\end{bmatrix}, 
$$
we easily have
$
({}^t\!M^{-}_{2} - {}^t\!M^{+}_{2})\Z^4
=A_2\Z^4
$
where
$
A_{2} 
=
\begin{bmatrix}
-1 &0  & 0  & 0 \\
 0 & -1 & 0 &0 \\
 0 & 1 & 0  & 0 \\
1  & 0 & 0 & 0
\end{bmatrix}. 
$
We set
$$
R 
=
\begin{bmatrix}
1  & 0 & 0  & 0 \\
0 & 1 & 0  & 0 \\
0 & 1 & 1   &0  \\
1  & 0 & 0 & 1
\end{bmatrix}, \qquad
B_2 = RA_2=
\begin{bmatrix}
-1  & 0 & 0  & 0 \\
0 & -1 & 0  & 0 \\
0 & 0 &0   &0  \\
0  & 0 & 0 & 0
\end{bmatrix}. 
$$
The diagram 
\begin{equation*}
\begin{CD}
\Z^{4}/ 
(M^{-}_{2} - M^{+}_{2} ) \Z^{4}  @>\overline{M_2^-}>> 
 \Z^{4}/ 
(M^{-}_{2} - M^{+}_{2} ) \Z^{4}  \\
@\vert @\vert \\
\Z^{4}/ 
A_2 \Z^{4}  @. 
 \Z^{4}/ 
A_2 \Z^{4}  \\
@V{R}VV  @V{R}VV \\
 \Z^{4}/ 
B_2 \Z^{4}  @>\overline{R M_2^- R^{-1}}>> 
 \Z^{4}/ 
B_2 \Z^{4} \\ 
\end{CD}
\end{equation*}
commutes,
where $\overline{M_{2}^{-}}$
is the endomorphism on the abelian group 
$\Z^{4}/ 
(M^{-}_{2} - M^{+}_{2} ) \Z^{4} 
$
induced by the matrix 
$M^{-}_{2}.$
As 
$
R M_2^- R^{-1} 
\begin{bmatrix}
0\\
0\\
z\\
w
\end{bmatrix}
=
\begin{bmatrix}
1  & -1 & 1  & 0 \\
-1 & 1 & 0  & 1 \\
0 & 0 &1   &1  \\
0  & 0 & 1 & 1 
\end{bmatrix}
\begin{bmatrix}
0\\
0\\
z\\
w
\end{bmatrix}
=
\begin{bmatrix}
z\\
w\\
z+w\\
z+w
\end{bmatrix},
$
we have the commutative diagram 
\begin{equation*}
\begin{CD}
\Z^{4}/ 
B_2 \Z^{4}  @>\overline{R M_2^-R^{-1}}>> 
 \Z^{4}/ 
B_2 \Z^{4} \\
@\vert @\vert \\ 
\Z^2 @>M_2>> 
 \Z^2
\end{CD}.
\end{equation*}
As
$ \displaystyle{\lim_{l\to{\infty}}} \{
M_{2}:
\Z^{2} 
 \longrightarrow 
\Z^2
\} 
\cong 
\Z[\frac{1}{2}],
$
we have 
$ K_0(\R_{\Lambda_{G_2}}) \cong \Z[\frac{1}{2}].
$

We will next compute
$ K_1(\R_{\Lambda_{G_2}}).$
A vector 
$
\begin{bmatrix}
x\\
y\\
z\\
w
\end{bmatrix}
$
belongs to
$\Ker({}^t\!M^{-}_{2} - {}^t\!M^{+}_{2} )$
if and only if
$x =w, y = z.$
By putting
$
P=
\begin{bmatrix}
1  & 0 & 0  & 0 \\
0 & 1 & 0  & 0 \\
0 & 0 &0   &0  \\
0  & 0 & 0 & 0 
\end{bmatrix},
$
we see that
\begin{equation*}
\Ker({}^t\!M^{-}_{2} - {}^t\!M^{+}_{2} ) \text{ in } \Z^4
\cong P\Z^4 \cong \Z^2,
\end{equation*}
so that the diagram
\begin{equation*}
\begin{CD}
\Ker({}^t\!M^{-}_{2} - {}^t\!M^{+}_{2} )
  @> M_2^- >> 
\Ker({}^t\!M^{-}_{2} - {}^t\!M^{+}_{2} ) \\
@V{P}VV    @VV{P}V \\ 
\Z^2 @>M_2>> 
 \Z^2
\end{CD}
\end{equation*}
commutes and hence  
$ K_1(\R_{\Lambda_{G_2}}) \cong \Z[\frac{1}{2}].
$
\end{proof}
%%%%%%%%%%%%%%%%%%%%%%%%%%%%%%%%%%%%%%%%%%%%%%%%%
%%%%%%%%%%%%%%%%%%%%%%%%%%%%%%%%%%%%%%%%%%%%%%%%%%%%%%

\medskip

{\bf 2. Golden mean shift.} 
Let 
$
\F = 
\begin{bmatrix}
a & b \\
c & 0 
\end{bmatrix}
$
be a symbolic matrix over alphabet $\Sigma =\{ a,b,c \}.$ 
Let $G_\F$ be the finite directed graph having two vertices and three  directed edges
associated with the symbolic matrix $\F$.
Let us consider the symbolic matrix bisystem
$(\M_{G_\F}^-, \M_{G_\F}^+)$ defined by 
\eqref{eq:symbma2}.
It satisfies FPCC.
Let $({\frak L}^-_{G_\F}, {\frak L}^+_{G_\F})$ 
be the associated $\lambda$-graph bisystem.
 Its presenting subshift is topologically conjugate to the topological Markov shift
defined by the matrix
$F = 
\begin{bmatrix}
1 & 1 \\
1 & 0 
\end{bmatrix},
$
that is called the golden mean shift written $\Lambda_F$ (cf. \cite{LM}).
The matrix $F$ has the 
unique positive eigenvalue  
$\frac{1 + \sqrt{5}}{2}$, denoted by
$\beta$.
By \eqref{eq:symbma2},   
the associated nonnegative matrix bisystem
 $(F^-, F^+)$  is given by  
\begin{equation}
F^- = 
\begin{bmatrix}
1 & 0   & 1 & 0 \\
0   & 1 & 0   & 1 \\
1 & 0   & 0 & 0  \\
0   & 1 & 0   & 0 
\end{bmatrix}, \qquad
F^+ =
\begin{bmatrix}
1 & 1 & 0 & 0 \\
1 & 0 & 0 & 0 \\
0 & 0 & 1 & 1 \\
0 & 0 & 1 & 0 \\ 
\end{bmatrix}.  \label{eq:nnmbgm}
\end{equation}
We will show the following.
\begin{proposition}\label{prop:7.6}
$
   K_i(\R_{\Lambda_F}) \cong
\Z^2, \, i~0,1.
$
\end{proposition}
\begin{proof}
Since
$$
{}^t\!F^{-} - {}^t\!F^{+}
=F^{-} - F^{+} 
=
\begin{bmatrix}
0  &-1 & 1  & 0 \\
-1 & 1 & 0  & 1 \\
 1 & 0 & -1  &-1  \\
0  & 1 & -1 & 0
\end{bmatrix}, 
$$
we easily have
$
({}^t\!F^{-} - {}^t\!F^{+})\Z^4
=F'\Z^4
$
where
$
F' 
=
\begin{bmatrix}
-1 &0  & 0  & 0 \\
 1 & -1 & 0 &0 \\
 0 & 1 & 0  & 0 \\
1  & 0 & 0 & 0 
\end{bmatrix}. 
$
We set
$$
S
=
\begin{bmatrix}
1  & 0 & 0  & 0 \\
1 & 1 & 0  & 0 \\
1 & 1 & 1   &0  \\
1  & 0 & 0 & 1 
\end{bmatrix}, \qquad
B_2 = SF'=
\begin{bmatrix}
-1  & 0 & 0  & 0 \\
0 & -1 & 0  & 0 \\
0 & 0 &0   &0  \\
0  & 0 & 0 & 0 
\end{bmatrix}.
$$
The diagram 
\begin{equation*}
\begin{CD}
\Z^{4}/ 
(F^{-} - F^{+} ) \Z^{4}  @>\overline{F^-}>> 
 \Z^{4}/ 
(F^{-} - F^{+} ) \Z^{4}  \\
@\vert @\vert \\
\Z^{4}/ 
F' \Z^{4}  @. 
 \Z^{4}/ 
F' \Z^{4}  \\
@V{S}VV  @V{S}VV \\
 \Z^{4}/ 
B_2 \Z^{4}  @>\overline{S F S^{-1}}>> 
 \Z^{4}/ 
B_2 \Z^{4} \\ 
\end{CD}
\end{equation*}
commutes,
where $\overline{F^{-}}$ denotes  the endomorphism on the abelian group
$\Z^{4}/ 
(F^{-} - F^{+} ) \Z^{4} 
$
induced by the matrix $F^-$. 
As 
$
S F^- S^{-1}
\begin{bmatrix}
0\\
0\\
z\\
w
\end{bmatrix}
=
\begin{bmatrix}
1 & -1 & 1  & 0 \\
-1 & 0 & 1  & 1 \\
0 & 0 & 1   &1  \\
0  & 0 & 1 & 0 
\end{bmatrix}
\begin{bmatrix}
0\\
0\\
z\\
w
\end{bmatrix}
=
\begin{bmatrix}
z\\
z+w\\
z+w\\
z
\end{bmatrix},
$
the diagram
\begin{equation*}
\begin{CD}
\Z^{4}/ 
B_2 \Z^{4}  @>\overline{S F^- S^{-1}}>> 
 \Z^{4}/ 
B_2 \Z^{4} \\
@\vert @\vert \\ 
\Z^2 @>F>> 
 \Z^2
\end{CD}
\end{equation*}
commutes. 
As
$\displaystyle{
\lim_{l\to{\infty}}} \{
F:
\Z^{2} 
 \longrightarrow 
\Z^2
\} 
\cong 
\Z[\frac{1}{\beta}],
$
we have 
$ K_0(\R_{\Lambda_F}) \cong \Z[\frac{1}{\beta}].
$
Since $\frac{1}{\beta} = 1-\beta$, the group
$\Z[\frac{1}{\beta}]$ is isomorphic to $\Z + \Z\beta$ and hence to 
$\Z^2$.

We will next compute 
$K_1(\R_{\Lambda_F}).$
A vector 
$
\begin{bmatrix}
x\\
y\\
z\\
w
\end{bmatrix}
$
belongs to
$\Ker({}^t\!F^{-} - {}^t\!F^{+} )$
if and only if
$z =y,\, w =x -y.$
By putting
$
P=
\begin{bmatrix}
1  & 0 & 0  & 0 \\
0 & 1 & 0  & 0 \\
0 & 0 &0   &0  \\
0  & 0 & 0 & 0 
\end{bmatrix},
$
we see that
\begin{equation*}
\Ker({}^t\!F^{-} - {}^t\!F^{+} ) \text{ in } \Z^4
\cong P\Z^4 \cong \Z^2,
\end{equation*}
so that the diagram
\begin{equation*}
\begin{CD}
\Ker({}^t\!F^{-} - {}^t\!F^{+} )
  @> F_2^- >> 
\Ker({}^t\!F^{-} - {}^t\!F^{+} ) \\
@V{P}VV    @VV{P}V \\ 
\Z^2 @>F>> 
 \Z^2
\end{CD}
\end{equation*}
commutes
and hence we  
$ K_1(\R_{\Lambda_F}) \cong \Z[\frac{1}{\beta}] \cong \Z^2.
$
\end{proof}
\begin{remark}
As in \cite[Lemma 10.4]{MaPre2019}, we know that the $C^*$-algebra $\R_{\Lambda_F}$
is a simple A$\T$-algebra with real rank zero having a unique tracial state $\tau$
such that  $\tau_*(K_0(\R_{\Lambda_F}))= \Z + \Z\beta$. 
Hence by classification of theory 
of $C^*$-algebra (\cite{Elliott}, cf. \cite{Ro3}), 
$\R_{\Lambda_F}$ is isomorphic to the irrational rotation $C^*$-algebra
$A_\beta$ with irrational angle $\beta$. 
\end{remark}

%\newpage

%%%%%%%%%%%%%%%%%%%%%%%%%%%%%%%%%%
%%%%%%%%%%%%%%%%%%%%%%%%%%%%%%%%%%%%%%%%%%%%%%%%
\section{Even shift}
%%%%%%%%%%%%%%%%%%%%%%%%%%%%%%%%%%%%
The even shift written $\Lambda_{\ev}$ 
is a sofic subshift over alphabet $\{\alpha,\beta\}$
whose forbidden words are the set of  words
 $$
\alpha\overbrace{\beta\cdots\beta}^{\operatorname{odd}}\alpha.   
$$
It lives outside topological Markov shifts (cf. \cite[Example 1.2.4]{LM}).
In this section, 
we will compute the K-groups 
for two 
$C^*$-algebras 
$\R_{\Lambda_\ev^c}$ and 
$\R_{\Lambda_\ev^\lambda}$
associated with
the $\lambda$-graph bisystems 
$({\frak L}^-_{\Lambda_{\ev}^c}, {\frak L}^+_{\Lambda_{\ev}^c})$
and
$({\frak L}^-_{\Lambda_{\ev}^\lambda}, {\frak L}^+_{\Lambda_{\ev}^\lambda})$
both of which present the even shift, respectively.
The former one 
$({\frak L}^-_{\Lambda_{\ev}^c}, {\frak L}^+_{\Lambda_{\ev}^c})$
is the canonical symbolic matrix bisystem for $\Lambda_\ev$.
The latter  one 
$({\frak L}^-_{\Lambda_{\ev}^\lambda}, {\frak L}^+_{\Lambda_{\ev}^\lambda})$
is an irreducible component of the former one,
so that the associated $C^*$-algebra
$\R_{\Lambda_\ev^\lambda}$ is simple,
whereas 
the other one 
$\R_{\Lambda_\ev^c}$ is not simple.

We will first visualize the canonical $\lambda$-graph bisystem 
$({\frak L}^-_{\Lambda_{\ev}^c}, {\frak L}^+_{\Lambda_{\ev}^c})$ for $\Lambda_\ev$.
For $l\in \N$, we put
\begin{align*}
\Omega_{-1,l}^c(1,1) 
& :=\{ x_{[0,l-1]}\in B_l(\Lambda_\ev) 
\mid (x_n)_{n\in \N} \in\Lambda_\ev, \, x_{-1} =\alpha, x_l = \alpha \},\\ 
\Omega_{-1,l}^c(1,2) 
& :=\{ x_{[0,l-1]}\in B_l(\Lambda_\ev) 
\mid (x_n)_{n\in \N} \in\Lambda_\ev, \, x_{-1} =\alpha, x_l = \beta, x_{l+1} =\alpha \},\\ 
\Omega_{-1,l}^c(2,1) 
& :=\{ x_{[0,l-1]}\in B_l(\Lambda_\ev) 
\mid (x_n)_{n\in \N} \in\Lambda_\ev, \, x_{-2} = \alpha, x_{-1} =\beta, x_l = \alpha \},\\ 
\Omega_{-1,l}^c(2,2) 
& :=\{ x_{[0,l-1]}\in B_l(\Lambda_\ev) 
\mid (x_n)_{n\in \N} \in\Lambda_\ev, \, x_{-2} = \alpha, x_{-1} =\beta, x_l = \beta, x_{l+1} =\alpha \},\\ 
\Omega_{-1,l}^c(1,\infty) 
& :=\{ x_{[0,l-1]}\in B_l(\Lambda_\ev) 
\mid (x_n)_{n\in \N} \in\Lambda_\ev, \, x_{-1} =\alpha, x_{[l,\infty)} = \beta^{+\infty} \},\\
\Omega_{-1,l}^c(2,\infty) 
& :=\{ x_{[0,l-1]}\in B_l(\Lambda_\ev) 
\mid (x_n)_{n\in \N} \in\Lambda_\ev, \, x_{-2} =\alpha, x_{-1} =\beta, x_{[l,\infty)} = \beta^{+\infty} \},\\ 
\Omega_{-1,l}^c(\infty,1) 
& :=\{ x_{[0,l-1]}\in B_l(\Lambda_\ev) 
\mid (x_n)_{n\in \N} \in\Lambda_\ev, \, x_{(-\infty,-1]} =\beta^{-\infty}, x_l = \alpha \},\\
\Omega_{-1,l}^c(\infty,2) 
& :=\{ x_{[0,l-1]}\in B_l(\Lambda_\ev) 
\mid (x_n)_{n\in \N} \in\Lambda_\ev, \, x_{(-\infty,-1]} =\beta^{-\infty}, x_l=\beta,
x_{l+1} = \alpha \},\\ 
\Omega_{-1,l}^c(\infty,\infty) 
& :=\{ x_{[0,l-1]}\in B_l(\Lambda_\ev) 
\mid (x_n)_{n\in \N} \in\Lambda_\ev, \, x_{(-\infty,-1]} =\beta^{-\infty}, x_{[l,\infty)} = \beta^{+\infty} \},
\end{align*}
where $\beta^{+\infty} =(b,b,b,\dots), \beta^{-\infty} =(\dots,b,b,b)$ all symbols are 
$\beta$s. 

\begin{lemma}
\hspace{6cm}
\begin{enumerate}
\renewcommand{\theenumi}{\roman{enumi}}
\renewcommand{\labelenumi}{\textup{(\theenumi)}}
\item For $l=1$, we have
\begin{gather*}
\Omega_{-1,1}^c(1,1)  = \{ \alpha\}, \quad
\Omega_{-1,1}^c(1,2)  = \Omega_{-1,1}^c(2,1)  = \{ \beta\}, \quad
\Omega_{-1,1}^c(2,2)  = \emptyset, \\
\Omega_{-1,1}^c(1,\infty) = \Omega_{-1,1}^c(2,\infty) 
=\Omega_{-1,1}^c(\infty,1)=\Omega_{-1,1}^c(\infty,2) 
=\Omega_{-1,1}^c(\infty,\infty) = \{\alpha,\beta\} \\
\intertext{ and }
\Omega_{-1,1}^c = 
\Omega_{-1,1}^c(1,1) \sqcup \Omega_{-1,1}^c(1,2)\sqcup\Omega_{-1,1}^c(\infty,\infty).
\end{gather*}
Hence the sets
$\{
\Omega_{-1,1}^c(1,1),\, \Omega_{-1,1}^c(1,2), \, \Omega_{-1,1}^c(\infty,\infty)
\}$
are $(-1,1)$-centrally equivalence classes of $\Lambda_\ev$.
\item For $l \ge 2$, each of the sets
$\Omega_{-1,l}^c(i,j), \, i,j=1,2,\infty$
is non-empty, and they are disjoint such that 
\begin{equation*}
\Omega_{-1,l}^c = \bigsqcup_{i,j=1,2,\infty}\Omega_{-1,l}^c(i,j).
\end{equation*}
Hence the sets
$\{\Omega_{-1,l}^c(i,j) \mid  i,j=1,2,\infty \}$
are $(-1,l)$-centrally equivalence classes of $\Lambda_\ev$.
 \end{enumerate}
\end{lemma}
Let us represent the canonical $\lambda$-graph bisystem
$({\frak L}_{\Lambda_{\ev}^{c}}^-,
{\frak L}_{\Lambda_{\ev}^{c}}^+)$ 
of the even shift $\Lambda_{\ev}$.
Since 
$\{\Omega_{-1,l}^c(i,j) \mid  i,j=1,2,\infty \}$
are the $(-1,l)$-centrally equivalence classes 
$\{C_i^{-1,l} \mid i=1,2,3\}$ for $l=1$ 
and 
$\{C_i^{-1,l} \mid i=1,2,\dots, 9\}$ for $l \ge 2$, 
we define its vertices by setting 
\begin{gather*}
v^1_1 :=\Omega_{-1,1}^c(1,1), \quad
v^1_2 :=\Omega_{-1,1}^c(1,2), \quad
v^1_3 :=\Omega_{-1,1}^c(\infty,\infty)  \\
\intertext{ and }
v^l_1 :=\Omega_{-1,l}^c(1,1), \quad
v^l_2 :=\Omega_{-1,l}^c(1,2), \quad
v^l_3 :=\Omega_{-1,l}^c(2,1)  \quad
v^l_4 :=\Omega_{-1,l}^c(2,2), \\
v^l_5 :=\Omega_{-1,l}^c(1,\infty), \quad
v^l_6 :=\Omega_{-1,l}^c(2,\infty), \quad
v^l_7 :=\Omega_{-1,l}^c(\infty,1)  \quad
v^l_8 :=\Omega_{-1,l}^c(\infty,2), \\
v^l_9 :=\Omega_{-1,l}^c(\infty,\infty) \qquad \text{ for } l \ge 2.
\end{gather*}
Define the vertex sets $V_l$ by 
\begin{equation*}
V^c_0: = \{ v^0_1\},\qquad
V^c_1 :=\{ v^1_1, v^1_2, v^1_3\}, \qquad 
V^c_l :=\{ v^l_i \}_{i=1}^9 \quad \text{ for } l\ge 2. 
\end{equation*}
%We note that $C^{-1,l}_i = v^l_i, i=1,2,\dots, 9$.
Following the construction of the canonical $\lambda$-graph bisystem for subshift
as in Section 2, we have the canonical  $\lambda$-graph bisystem 
$({\frak L}_{\Lambda_{ev}^{c}}^-,{\frak L}_{\Lambda_{ev}^{c}}^+)$ 
for $\Lambda_\ev$ with its vertex  $V_l^c = \cup_{l=0}^\infty V_l$. 
It is figured as Figure \ref{fig:evencm} and
Figure \ref{fig:evencp}
in the end of this section. 
Let us denote by $(\M^{c-}, \M^{c+})$ the associated symbolic matrix bisystem
for 
$({\frak L}_{\Lambda_{ev}^{c}}^-,{\frak L}_{\Lambda_{ev}^{c}}^+)$.
The following lemma is straightforward by its construction of 
$({\frak L}_{\Lambda_{ev}^{c}}^-,{\frak L}_{\Lambda_{ev}^{c}}^+)$
and by viewing Figure \ref{fig:evencm} and
Figure \ref{fig:evencp}.
\begin{lemma}
\begin{gather*}
\M^{c-}_{0,1} = [\alpha, \beta, \alpha+\beta], \qquad
\M^{c+}_{0,1} = [\alpha, \beta, \alpha+\beta], \\
\M^{c-}_{1,2} = 
\begin{bmatrix}
\alpha&    0   &\beta&0       &0      &0       &\alpha&0                &0                 \\
\beta &\alpha&0       &\beta&\beta &0       & 0      &\alpha+\beta& 0               \\
 0      &0       &0       &0       &\alpha&\beta&\beta & 0               &\alpha+\beta
\end{bmatrix}, \\
\M^{c+}_{1,2} = 
\begin{bmatrix}
\alpha&\beta&0      &0      &\alpha&0                &0      &0       & 0       \\
\beta &0      &\alpha&\beta&0      &\alpha+\beta&\beta&0       &0         \\
    0   &0      &0       &0      &\beta&0                &\alpha&\beta&\alpha+\beta
\end{bmatrix}, \\
\M^{c-}_{l,l+1} = 
\begin{bmatrix}
\alpha&    0   &\beta&0       &0      &0       &\alpha&0       &0            \\
 0      &\alpha&0       &\beta&0      &0       & 0      &\alpha& 0           \\
 \beta&0       &0       &0      &0      &0       &0       & 0      &0            \\
 0      &\beta &0       &0      &0      &0       &0       & 0      &0            \\
0      &0       &0       &0      &\alpha&\beta&0       & 0      &\alpha      \\
0      &0       &0       &0      &\beta &0      &0       & 0      &0            \\
0      &0       &0       &0      &0      &0       &\beta & 0      &0            \\
0      &0       &0       &0      &0      &0       &0       & \beta&0            \\
0      &0       &0       &0      &0      &0       &0       & 0      &\beta
\end{bmatrix}, \quad
\M^{c+}_{l,l+1} = 
\begin{bmatrix}
\alpha&\beta &0       &0       &\alpha&0     &0       &0       &0            \\
\beta &0       &0       &0       &0      &0       & 0      &0      & 0           \\
0       &0       &\alpha&\beta &0      &\alpha&0       & 0      &0            \\
0      &0        &\beta &0      &0      &0       &0       & 0      &0            \\
0      &0       &0       &0      &\beta &0       &0       & 0      &0            \\
0      &0       &0       &0      &0       &\beta &0      &0       & 0            \\
0      &0       &0       &0      &0      &0       &\alpha& \beta&\alpha      \\
0      &0       &0       &0      &0      &0       &\beta &0       &0            \\
0      &0       &0       &0      &0      &0       &0       & 0      &\beta
\end{bmatrix}\,
\text{for  }l\ge 2.
\end{gather*}
\end{lemma}
Let $M^{c-}, M^{c+}$ be the $9\times 9$ matrices obtained from 
$\M^{c-}_{l,l+1}, \M^{c+}_{l,l+1}, l\ge 2$ by setting all the symbols equal to one, so that  
\begin{equation*}
M^{c-} = 
\begin{bmatrix}
1       &    0  &1       &0      &0      &0       &1       &0       &0      \\
 0      &1      &0       &1      &0      &0       & 0      &1       &0      \\
1       &0      &0       &0      &0      &0       &0       & 0      &0       \\
 0      &1      &0       &0      &0      &0       &0       & 0      &0       \\
0      &0       &0       &0      &1      &1       &0       & 0      &1      \\
0      &0       &0       &0      &1      &0       &0       & 0      &0       \\
0      &0       &0       &0      &0      &0       &1       & 0      &0       \\
0      &0       &0       &0      &0      &0       &0       & 1      &0       \\
0      &0       &0       &0      &0      &0       &0       & 0      &1
\end{bmatrix}, \quad
M^{c+} = 
\begin{bmatrix}
1      &1       &0       &0      &1      &0       &0       &0       &0      \\
1      &0       &0       &0      &0      &0       & 0      &0       &0      \\
0      &0       &1       &1      &0      &1       &0       & 0      &0       \\
0      &0       &1       &0      &0      &0       &0       & 0      &0       \\
0      &0       &0       &0      &1      &0       &0       & 0      &0       \\
0      &0       &0       &0      &0      &1       &0       &0       & 0       \\
0      &0       &0       &0      &0      &0       &1       & 1      &1     \\
0      &0       &0       &0      &0      &0       &1       &0       &0       \\
0      &0       &0       &0      &0      &0       &0       & 0      &1
\end{bmatrix}
\end{equation*}
and hence
\begin{equation*}
{}^t\!M^{c-} = 
\begin{bmatrix}
1      &    0   &1       &0      &0      &0       &0       &0       &0  \\
0      & 1      &0       &1      &0      &0       & 0      &0       &0 \\
1      & 0      &0       &0      &0      &0       &0       & 0      &0   \\
0      &1       &0       &0      &0      &0       &0       & 0      &0   \\
0      &0       &0       &0      &1      &1       &0       & 0      &0  \\
0      &0       &0       &0      &1      &0       &0       & 0      &0   \\
1      &0       &0       &0      &0      &0       &1       & 0      &0   \\
0      &1       &0       &0      &0      &0       &0       & 1      &0  \\
0      &0       &0       &0      &1      &0       &0       & 0      &1
\end{bmatrix}, \quad
{}^t\!M^{c+} = 
\begin{bmatrix}
1      &1       &0       &0      &0      &0       &0      &0      &0   \\
1      &0       &0       &0      &0      &0       &0      &0      & 0  \\
0      &0       &1       &1      &0      &0       &0      &0      &0  \\
0      &0       &1       &0      &0      &0       &0      &0      &0            \\
1      &0       &0       &0      &1      &0       &0      &0      &0            \\
0      &0       &1       &0      &0      &1       &0      &0      & 0            \\
0      &0       &0       &0      &0      &0       &1      &1      &0      \\
0      &0       &0       &0      &0      &0       &1      &0      &0            \\
0      &0       &0       &0      &0      &0       &1      &0      &1
\end{bmatrix}.
\end{equation*}
We put
\begin{equation*}
B := {}^t\!M^{c-}- {}^t\!M^{c+} = 
\begin{bmatrix}
0      &-1     &1       &0      &0      &0       &0       &0       &0  \\
-1    & 1      &0       &1      &0      &0       &0       &0       &0  \\
1      & 0      &-1     &-1     &0      &0      &0       & 0      &0   \\
0      &1       &-1     &0      &0      &0       &0       & 0      &0   \\
-1     &0       &0      &0      &0      &1       &0       & 0      &0  \\
0      &0       &-1     &0      &1      &-1      &0       & 0      &0   \\
1      &0       &0       &0      &0      &0       &0       & -1     &0   \\
0      &1       &0       &0      &0      &0       &-1      & 1      &0  \\
0      &0       &0       &0      &1      &0       &-1      & 0      &0
\end{bmatrix}.
\end{equation*}
Hence we have 
\begin{lemma}
$K_0(\R_{\Lambda_\ev^c}) \cong \varinjlim\{{}^t\!M^{c-}: \Z^9/B\Z^9 \longrightarrow \Z^9/B\Z^9 \}$.
\end{lemma}
To compute the above group of the inductive limit, 
we provide some notation of basic operations on matrices from elementary linear algebras.    
For $i,j =1,2,\dots,9$, define the column operations to be: 
\begin{align*}
C(i\rightarrow j)& = \text{Add the $i$th column to the $j$th column},\\
C(-i\rightarrow j)& = \text{Add the minus of the $i$th column to the $j$th column},\\
C(i\leftrightarrow j)& = \text{Exchange  the $i$th column and the $j$th column}.
\end{align*}
The row operations
$R(i\rightarrow j), R(-i\rightarrow j), R(i\leftrightarrow j)
$
are similarly defined.
The matrix $B$ can be transformed by the sequence of the following column operations:
\begin{align*}
&  B \\ 
& \downarrow \quad 
  C(4\rightarrow 1)C(5\rightarrow 7)C(8\rightarrow 7)C(-8\rightarrow 2)C(7\rightarrow 1) \\  
& 
\begin{bmatrix}
0      &-1      &1       &0      &0      &0       &0       & 0      &0  \\
0      & 1      &0       &1      &0      &0       &0       & 0      &0  \\ 
0      & 0      &-1      &-1     &0      &0       &0       & 0      &0  \\
0      &1       &-1      &0      &0      &0       &0       & 0      &0  \\
-1     &0       &0       &0      &0      &1       &0       & 0      &0  \\
1      &0       &-1      &0      &1      &-1      &1       & 0      &0  \\
0      &1       &0       &0      &0      &0       &-1      & -1     &0  \\
0      &0       &0       &0      &0      &0       &0       & 1      &0  \\
0      &0       &0       &0      &1      &0       &0       & 0      &0
\end{bmatrix} \\
& \downarrow \quad C(6\rightarrow 1)C(7\rightarrow 2)C(-6\rightarrow 3)
C(-7\rightarrow 5) \\  
& 
\begin{bmatrix}
0      &-1      &1       &0      &0      &0       &0       & 0      &0  \\
0      & 1      &0       &1      &0      &0       &0       & 0      &0  \\ 
0      & 0      &-1      &-1     &0      &0       &0       & 0      &0  \\
0      &1       &-1      &0      &0      &0       &0       & 0      &0  \\
0      &0       &-1      &0      &0      &1       &0       & 0      &0  \\
0      &1       &0       &0      &0      &-1      &1       & 0      &0  \\
0      &0       &0       &0      &1      &0       &-1      & -1     &0  \\
0      &0       &0       &0      &0      &0       &0       & 1      &0  \\
0      &0       &0       &0      &1      &0       &0       & 0      &0
\end{bmatrix} \\
& \downarrow \quad C(6\rightarrow 2)C(3\rightarrow 2)C(-4\rightarrow 2)C(8\rightarrow 5) \\  
& 
\begin{bmatrix}
0      &0       &1       &0      &0      &0       &0       & 0      &0  \\
0      &0       &0       &1      &0      &0       &0       & 0      &0  \\ 
0      &0       &-1      &-1     &0      &0       &0       & 0      &0  \\
0      &0       &-1      &0      &0      &0       &0       & 0      &0  \\
0      &0       &-1      &0      &0      &1       &0       & 0      &0  \\
0      &0       &0       &0      &0      &-1      &1       & 0      &0  \\
0      &0       &0       &0      &0      &0       &-1      & -1     &0  \\
0      &0       &0       &0      &1      &0       &0       & 1      &0  \\
0      &0       &0       &0      &1      &0       &0       & 0      &0
\end{bmatrix}, 
\end{align*}
where the column operations
$C(i_1\rightarrow i_2)\cdots C(i_{n-1}\rightarrow i_n)$
are taken from the leftmost $C(i_1\rightarrow i_2)$ in order. 
The last matrix goes to the following matrix denoted by $C$ by the operations
$C(1\leftrightarrow 3)C(2\leftrightarrow 4)C(5\leftrightarrow 9)$
\begin{equation*}
C :=
\begin{bmatrix}
1      &0       &0       &0      &0      &0       &0       & 0      &0  \\
0      &1       &0       &0      &0      &0       &0       & 0      &0  \\ 
-1     &-1      &0       &0      &0      &0       &0       & 0      &0  \\
-1     &0       &0       &0      &0      &0       &0       & 0      &0  \\
-1     &0       &0       &0      &0      &1       &0       & 0      &0  \\
0      &0       &0       &0      &0      &-1      &1       & 0      &0  \\
0      &0       &0       &0      &0      &0       &-1      & -1     &0  \\
0      &0       &0       &0      &0      &0       &0       & 1      &1  \\
0      &0       &0       &0      &0      &0       &0       & 0      &1
\end{bmatrix}. 
\end{equation*}
We note that the equality $B\Z^9 = C \Z^9$ and hence
$\Z^9/ B\Z^9 = \Z^9 /C\Z^9$ hold.
By the row operations
$R(-8\rightarrow 9) R(7\rightarrow 8) R(6\rightarrow 7) R(5\rightarrow 6) R(2\rightarrow 3) R(1\rightarrow 5) R(1\rightarrow 4)R(1\rightarrow 3), 
$
where the row operations above
are taken from the rightmost $R(1\rightarrow 3)$ in order,
the matrices $C$ and ${}^t\!M^-$
go to the following matrices respectively:
\begin{equation*}
\begin{bmatrix}
1      &0       &0       &0      &0      &0       &0       & 0      &0  \\
0      &1       &0       &0      &0      &0       &0       & 0      &0  \\ 
0      &0       &0       &0      &0      &0       &0       & 0      &0  \\
0      &0       &0       &0      &0      &0       &0       & 0      &0  \\
0      &0       &0       &0      &1      &0       &0       & 0      &0  \\
0      &0       &0       &0      &0      &1       &0       & 0      &0  \\
0      &0       &0       &0      &0      &0       &-1      & 0      &0  \\
0      &0       &0       &0      &0      &0       &0       & 1      &0  \\
0      &0       &0       &0      &0      &0       &0       & 0      &0
\end{bmatrix}, \quad 
\begin{bmatrix}
1      &0       &1       &0      &0      &0       &0       &0       &0  \\
0      &1       &0       &1      &0      &0       &0       &0       &0  \\
2      &1       &1       &1      &0      &0       &0       & 0      &0  \\
1      &1       &1       &0      &0      &0       &0       & 0      &0  \\
1      &0       &1       &0      &1      &1       &0       & 0      &0  \\
1      &0       &1       &0      &2      &1       &0       & 0      &0  \\
2      &0       &1       &0      &2      &1       &1       & 0      &0  \\
2      &1       &1       &0      &2      &1       &1       & 1      &0  \\
-2     &-1      &-1      &0      &-1     &-1      &-1      &-1      &1
\end{bmatrix}, 
\end{equation*}
that are written $C^1$ and $M^1$, respectively. 
We thus have
\begin{lemma}\label{lem:evenmainc0}
$
K_0(\R_{\Lambda_{\operatorname{ev}}^{c}}) \cong \Z^3.
$
\end{lemma}
\begin{proof}
As ${}^t\!M^{c-}: \Z^9 /C\Z^9\longrightarrow \Z^9 /C\Z^9$ is conjugate to
$M^1: \Z^9 /C^1\Z^9\longrightarrow \Z^9 /C^1\Z^9$,
through the preceding row operations,
and the latter one is conjugate  to
$F\oplus\id: \Z^2 \oplus \Z \longrightarrow \Z^2 \oplus \Z$
where
$
F =
\begin{bmatrix}
1 & 1 \\
1 & 0 
\end{bmatrix},
$
we have
\begin{equation*}
 \varinjlim \{{}^t\!M^{c-}: \Z^9 /C\Z^9\longrightarrow \Z^9 /C\Z^9 \} 
\cong
 \varinjlim \{F\oplus\id: \Z^2 \oplus \Z \longrightarrow \Z^2 \oplus \Z \} 
\cong
 \Z[\frac{1}{\beta}] \oplus \Z,
\end{equation*}
where 
$\beta = \frac{1 +\sqrt{5}}{2}$.
Therefore we  know that
$
K_0(\R_{\Lambda_{\operatorname{ev}}^{c}})
\cong
\Z[\frac{1}{\beta}] \oplus \Z
 \cong \Z^3.
$
\end{proof}
\medskip
We will next compute 
$
K_1(\R_{\Lambda_{\operatorname{ev}}^{c}})
$
that is isomorphic to the group of the inductive limit
\begin{equation*}
\varinjlim\{{}^t\!M^{c-}: \Ker(B)\text{ in } \Z^9 \longrightarrow \Ker(B)\text{ in } \Z^9\}.
\end{equation*}
By the following row operations from the rightmost in order, the matrix $B$ goes to:
\begin{align*}
&  B \\ 
& \downarrow \quad R(6\rightarrow 8)R(-9\rightarrow 8)R(1\rightarrow 8)R(5\rightarrow 6)
                   R(1\rightarrow 3)R(2\rightarrow 3)R(1\rightarrow 4) \\  
& 
\begin{bmatrix}
0      &-1      &1       &0      &0      &0       &0       & 0      &0  \\
-1     & 1      &0       &1      &0      &0       &0       & 0      &0  \\ 
0      & 0      &0       &0      &0      &0       &0       & 0      &0  \\
0      &0       &0       &0      &0      &0       &0       & 0      &0  \\
-1     &0       &0       &0      &0      &1       &0       & 0      &0  \\
-1     &0       &-1      &0      &1      &0       &0       & 0      &0  \\
1      &0       &0       &0      &0      &0       &0       & -1     &0  \\
0      &0       &0       &0      &0      &0       &0       & 0      &0  \\
0      &0       &0       &0      &1      &0       &-1      & 0      &0
\end{bmatrix} \\
& \downarrow \quad R(9\leftrightarrow 5)R(7\leftrightarrow 6)R(5\leftrightarrow 4)R(6\leftrightarrow 3) \\  
&
\begin{bmatrix}
0      &-1      &1       &0      &0      &0       &0       & 0      &0  \\
-1     & 1      &0       &1      &0      &0       &0       & 0      &0  \\ 
-1     & 0      &-1      &0      &1      &0       &0       & 0      &0  \\
-1     &0       &0       &0      &0      &1       &0       & 0      &0  \\
0      &0       &0       &0      &1      &0       &-1      & 0      &0  \\
1      &0       &0       &0      &0      &0       &0       & -1     &0  \\
0      &0       &0       &0      &0      &0       &0       & 0      &0  \\
0      &0       &0       &0      &0      &0       &0       & 0      &0  \\
0      &0       &0       &0      &0      &0       &0       & 0      &0
\end{bmatrix}.  
\end{align*}
We write the last matrix as $D$, so that 
$\Ker(B) = \Ker(D)$ and
\begin{equation}
\Ker(D) =\{ (x_i)_{i=1}^9 \in \Z^9 \mid
x_1 = x_2 + x_4 = -x_3 + x_5 = x_6 = x_8, \,x_2 = x_3,\, x_5 = x_7 \}.\label{eq:kerD} 
\end{equation}
We provide a lemma
\begin{lemma}
Put the matrix
$
M= 
\begin{bmatrix}
1 & 1 & 0 \\
1 & 0 & 0 \\
1 & 0& 1
\end{bmatrix}.
$
The correspondence
$\Phi:\Ker(D) \longrightarrow \Z^3$ defined by 
$\Phi((x_i)_{i=1}^9) = (x_7, x_8,  x_9)$
gives rise to an isomorphism such that the diagram
\begin{equation*}
\begin{CD}
\Ker(D)
  @> {}^t\!M^{c-} >> 
\Ker(D) \\
@V{\Phi}VV    @VV{\Phi}V \\ 
\Z^3 @>M >> 
 \Z^3
\end{CD}
\end{equation*}
commutes.
\end{lemma}
\begin{proof}
By \eqref{eq:kerD}, an element 
$(x_7,x_8,x_9) \in \Z^3$ uniquely determines
the other $(x_i)_{i=1}^6$
such that $(x_i)_{i=1}^9 \in \Ker(D)$.
This shows that $\Phi:\Ker(D) \longrightarrow \Z^3$
gives rise to an isomorphism. 
 For $(x,y,z) =(x_7,x_8,x_9) \in \Z^3$,
we have
\begin{equation*}
\Phi^{-1}(x,y,z) = (y, x-y, x-y, -x +2y, x, y, x, y, z) \in \Z^9
\end{equation*}
 so that 
\begin{equation*}
{}^t\!M^{c-}(\Phi^{-1}(x,y,z))
 = (x, y, y, x-y, x+y, x, x+y, x, x+ z) \in \Z^9
\end{equation*}
and hence
 \begin{equation*}
(\Phi\circ {}^t\!M^{c-}\circ \Phi^{-1})(x,y,z) = (x+y, x, x+z) \in \Z^9,
\end{equation*}
 proving the commutativity of the diagram.
\end{proof} 
We thus have
\begin{lemma}\label{lem:evenmainc1}
$
K_1(\R_{\Lambda_{\operatorname{ev}}^{c}}) \cong \Z^3.
$
\end{lemma}
\begin{proof}
Put the matrix
$
U= 
\begin{bmatrix}
1 & 0 & 0 \\
0 & 1 & 0 \\
-1 & -1& 1
\end{bmatrix}.
$
The diagram
\begin{equation*}
\begin{CD}
\Z^3
  @>M >> 
\Z^3 \\
@V{U}VV    @VV{U}V \\ 
\Z^3 @>F\oplus\id >> 
 \Z^3
\end{CD}
\end{equation*}
commutes.
By the preceding lemma,
we have 
$
K_1(\R_{\Lambda_{\operatorname{ev}}^{c}}) 
\cong
\Z[\frac{1}{\beta}] \oplus \Z
 \cong \Z^3.
$
\end{proof}
We thus reach the K-group formulas.
\begin{proposition}\label{prop:evenmainc}
$
K_i(\R_{\Lambda_{\operatorname{ev}}^{c}}) \cong \Z^3, \, i=0,1.
$
\end{proposition}
%%%%%%%%%%%%%%%%%%%%%%%%%%%%%%%%%%%%%%%%%%%%%%%%

\medskip

Let us next consider the $\lambda$-graph bisystem 
$({\frak L}_{\Lambda_{\operatorname{ev}}^{\lambda}}^-,
{\frak L}_{\Lambda_{\operatorname{ev}}^{\lambda}}^+)
$ that is figured as Figure \ref{fig:evenlambda}
in the end of this section.
It is a subsystem of 
$({\frak L}_{\Lambda_{\operatorname{ev}}^{c}}^-,
{\frak L}_{\Lambda_{\operatorname{ev}}^{c}}^+)
$
and 
irreducible, satisfies condition (I) in the sense of Definition \ref{def:conditionI}.
Hence the associated $C^*$-algebra 
$\R_{\Lambda_{\operatorname{ev}}^{\lambda}}
$
is simple.
The associated symbolic matrix bisystem
with 
$({\frak L}_{\Lambda_{\operatorname{ev}}^{c}}^-,
{\frak L}_{\Lambda_{\operatorname{ev}}^{c}}^+)
$
is denoted by
$(\M^{\lambda-}_{l,l+1},\M^{\lambda+}_{l,l+1})_{l\in\Zp}$
that is written 
\begin{gather*}
\M^{\lambda-}_{0,1} = [\alpha, \beta], \qquad
\M^{\lambda+}_{0,1} = [\alpha, \beta], \\
\M^{\lambda-}_{1,2} = 
\begin{bmatrix}
\alpha&    0   &\beta&0        \\
\beta &\alpha&0       &\beta \\
 \end{bmatrix}, \quad
\M^{\lambda+}_{1,2} = 
\begin{bmatrix}
\alpha&\beta&0      &0        \\
\beta &0      &\alpha&\beta \\
\end{bmatrix}, \\
\M^{\lambda-}_{l,l+1} = 
\begin{bmatrix}
\alpha&    0   &\beta&0          \\
 0      &\alpha&0       &\beta   \\
 \beta&0       &0       &0         \\
 0      &\beta &0       &0         \\
\end{bmatrix}, \quad
\M^{\lambda+}_{l,l+1} = 
\begin{bmatrix}
\alpha&\beta &0       &0           \\
\beta &0       &0       &0           \\
0       &0       &\alpha&\beta     \\
0      &0        &\beta &0           \\
\end{bmatrix}\quad
\text{for  }l\ge 2.
\end{gather*}
Let us denote by
$(M^{\lambda-}_{l,l+1},M^{\lambda+}_{l,l+1})_{l\in\Zp}$
the associated nonnegative matrix bisystem 
from 
$(\M^{\lambda-}_{l,l+1},\M^{\lambda+}_{l,l+1})_{l\in\Zp}$
obtained by setting all the symbols equal to one, so that we have
\begin{equation*}
M^{\lambda-}_{l,l+1} = F^-, \qquad M^{\lambda+}_{l,l+1} = F^+
\qquad \text{ for } l\ge 2,
\end{equation*}
where $F^-, F^+$ are the matrices appeared in \eqref{eq:nnmbgm}.
Hence by Proposition \ref{prop:7.6}, we have 
$$
K_i(\R_{\Lambda_{\operatorname{ev}}^{\lambda}}) \cong K_i(\R_{\Lambda_F}), \qquad 
i=0,1
$$
so that 
 \begin{proposition}\label{prop:evenmainlambda}
$
K_i(\R_{\Lambda_{\operatorname{ev}}^{\lambda}}) \cong \Z^2, \, 
i=0,1.
$
\end{proposition}

%\begin{proposition}\label{prop:evenmain}
%\hspace{6cm}
%\begin{enumerate}
%\renewcommand{\theenumi}{\roman{enumi}}
%\renewcommand{\labelenumi}{\textup{(\theenumi)}}
%\item $K_0(\R_{\Lambda_{\operatorname{ev}}^{c}}) = \Z^3, \qquad
%K_1(\R_{\Lambda_{\operatorname{ev}}^{c}}) = \Z^3.$
%\item $K_0(\R_{\Lambda_{\operatorname{ev}}^{\lambda}}) = \Z^2, \qquad
%K_1(\R_{\Lambda_{\operatorname{ev}}^{\lambda}}) = \Z^2.$
%\end{enumerate}
%\end{proposition}

%\newpage 

%%%%%%%%%%%%%%%%%%%%%%%%%%%%%%%%%%%%%%%%%%%%%
\begin{figure}[htbp]
\begin{center}
\input{pictureevencm.tex}
\end{center}
\caption{
  The labeled Bratteli diagram 
${\frak L}_{\Lambda_{\operatorname{ev}}^{c}}^-$
of the  canonical $\lambda$-graph bisystem 
$({\frak L}_{\Lambda_{\operatorname{ev}}^{c}}^-,{\frak L}_{\Lambda_{\operatorname{ev}}^{c}}^+)$ 
of the even shift $\Lambda_{\operatorname{ev}}$ 
}
\label{fig:evencm}
\end{figure}
%where upward arrows $\longleftarrow$ and {\mathversion{bold} $\longleftarrow$} (bold)
% are labeled $\alpha$ and $\beta$, respectively.

%\newpage

%%%%%%%%%%%%%%%%%%%%%%%%%%%%%%%%%%%%%%%%%%%%%%
\begin{figure}[htbp]
\begin{center}
\input{pictureevencp.tex}
\end{center}
\caption{
The labeled Bratteli diagram 
${\frak L}_{\Lambda_{\operatorname{ev}}^{c}}^+$
 of the canonical $\lambda$-graph bisystem 
$({\frak L}_{\Lambda_{\operatorname{ev}}^{c}}^-,{\frak L}_{\Lambda_{\operatorname{ev}}^{c}}^+)$ 
of the even shift $\Lambda_{\operatorname{ev}}$ 
}
\label{fig:evencp}
\end{figure}
%where downward arrows $\longleftarrow$ and {\mathversion{bold} $\longleftarrow$} (bold) are labeled $\alpha$ and $\beta$, respectively.  

%\newpage

%%%%%%%%%%%%%%%%%%%%%%%%%%%%%%%%%%%%%%%%%%%%%%
\begin{figure}[htbp]
\begin{center}
\input{pictureevenlambda.tex}
\end{center}
\caption{  
 The $\lambda$-graph bisystem 
$({\frak L}_{\Lambda_{\operatorname{ev}}^{\lambda}}^-,
{\frak L}_{\Lambda_{\operatorname{ev}}^{\lambda}}^+)$ 
of the even shift $\Lambda_{\operatorname{ev}}$ 
}
\label{fig:evenlambda}
\end{figure}
%where arrows $\longleftarrow$ and {\mathversion{bold} $\longleftarrow$} (bold)
% are labeled $\alpha$ and $\beta$, respectively.

%%%%%%%%%%%%%%%%%%%%%%%%%%%%%%%%%%%%%%%%%%%%%%%%%%%
%%%%%%%%%%%%%%%%%%%%%%%%%%%%%%%%%%%%%%%%%

\newpage

%%%%%%%%%%%%%%%%%%%%%%%%%%%%%%%%%%%%%%
{\it Acknowledgments:}
This work was  supported by JSPS KAKENHI Grant Numbers 15K04896, 19K03537.

%%%%%%%%%%%%%%%%%%%%%%%%%%%%%%

%%%%%%%%%%%%%%%%%%%%%%

%\email{kengo{\@}juen.ac.jp}

\end{document}